\documentclass[11pt,leqno,oneside]{amsart}
\usepackage{amsmath,amsfonts,amssymb,amsthm,enumerate,mathtools}
\usepackage{color,framed,mathscinet}
\usepackage[colorlinks, linkcolor=teal, citecolor=violet, urlcolor=violet]{hyperref}
\usepackage[a4paper, left=2cm, right=2cm, top=3cm, bottom=3cm]{geometry}
\usepackage[numbers,sort&compress]{natbib}
\usepackage{doi}
\usepackage[table]{xcolor}
\usepackage{multirow}
\usepackage{enumitem}
\usepackage{microtype}
\setlist[enumerate]{label={\rm(\roman*)},leftmargin=6ex}
\usepackage[a-2u]{pdfx}

\newcommand{\R}{\mathbb{R}}
\newcommand{\rn}{{\mathbb{R}^n}}
\newcommand{\N}{\mathbb{N}}
\newcommand{\Z}{\mathbb{Z}}
\newcommand{\MM}{\mathcal{M}}
\newcommand{\RR}{\mathcal{R}}
\newcommand{\HH}{\mathcal{H}}
\newcommand{\TT}{\mathcal{T}}
\newcommand{\med}{\operatorname{med}}
\newcommand{\mv}{\operatorname{mv}}
\newcommand{\m}{\operatorname{m}}
\newcommand{\sgn}{\operatorname{sgn}}
\newcommand{\esssup}{\operatornamewithlimits{ess\,sup}}

\newcommand{\lgn}{\mathcal L}
\newcommand{\Mpl}{\MM_+}

\newcommand{\half}{{\frac12}}
\newcommand{\thalf}{{\tfrac12}}
\newcommand{\ib}{\frac{1}{\beta}}

\newcommand\ip{\frac1p}
\newcommand\iq{\frac1q}
\newcommand\ir{\frac1r}

\newcommand{\phaxi}{\Phi(|x_1|)}

\newcommand{\RG}{(\rn,\gamma_n)}
\newcommand{\Hg}{\HH^{n-1}_{\gamma_n}}
\renewcommand{\d}{{\mathrm d}}
\newcommand{\dgn}{\d\gamma_n}
\newcommand{\dHg}{\d\Hg}
\newcommand{\weight}{\Theta}

\newcommand{\Wlgn}{W_{\kern-1.2pt\lgn}\kern0.7pt}

\newcommand{\sou}{S}

\newcommand{\xs}{X_{\lgn}}
\newcommand{\xsd}{X_{\lgn}'}
\newcommand{\ys}{Y^{\lgn}}
\newcommand{\tvp}{\overline{\varphi}}
\newcommand{\tps}{\overline{\psi}}

\let\bar\overline
\let\tilde\widetilde

\usepackage{xspace}
\makeatletter
\DeclareRobustCommand\onedot{\futurelet\@let@token\@onedot}
\def\@onedot{\ifx\@let@token.\else.\null\fi\xspace}
\newcommand{\eg}{e.g\onedot} 
\newcommand{\ie}{i.e\onedot} 
\newcommand{\cf}{cf\onedot} 
\renewcommand{\ae}{a.e\onedot} 
\makeatother

\newtheoremstyle{MyPlain}{}{}{\itshape}{}{\bfseries}{.}{5pt plus 4pt minus 3pt}{\thmname{#1}\thmnumber{ #2}\thmnote{ \textbf{[#3]}}}
\theoremstyle{MyPlain}
\newtheorem{theorem}{Theorem}[section]
\newtheorem{lemma}[theorem]{Lemma}
\newtheorem{proposition}[theorem]{Proposition}

\newtheoremstyle{MyRemark}{}{}{\upshape}{}{\bfseries}{.}{5pt plus 1pt minus 1pt}{}
\theoremstyle{MyRemark}

\newtheorem{example}[theorem]{Example}
\numberwithin{equation}{section}

\expandafter\let\expandafter\oldproof\csname\string\proof\endcsname
\let\oldendproof\endproof
\renewenvironment{proof}[1][\proofname]{%
  \oldproof[{{\bf #1.}}]%
}{\oldendproof}

\makeatletter
\newcommand{\onorm}{\@ifstar\@onorms\@onorm}
\newcommand{\@onorms}[1]{%
	\left|\mkern-1.5mu\left|\mkern-1.5mu\left|
	#1
	\right|\mkern-1.5mu\right|\mkern-1.5mu\right|
}
\newcommand{\@onorm}[2][]{%
  \mathopen{#1|\mkern-1.5mu#1|\mkern-1.5mu#1|}
  #2
  \mathclose{#1|\mkern-1.5mu#1|\mkern-1.5mu#1|}
}
\makeatother

\makeatletter
\def\paragraph{\bigskip\@startsection{paragraph}{4}%
  \z@\z@{-\fontdimen2\font}%
  {\normalfont\bfseries}}
\makeatother

\begin{document}

\title{Optimal Sobolev embeddings\\ for the Ornstein-Uhlenbeck operator}

\begin{abstract}
A comprehensive analysis of Sobolev-type inequalities for the Ornstein-Uhlenbeck operator in the Gauss space is offered.
A unified approach is proposed, providing one with criteria for their validity in the class of rearrangement-invariant function norms.
Optimal target and domain norms in the relevant inequalities are characterized via a reduction principle to one-dimensional inequalities for a Calder\'on type integral operator patterned on the Gaussian isoperimetric function.
Consequently, the best possible norms in a variety of specific families of spaces, including Lebesgue, Lorentz, Lorentz-Zygmund, Orlicz and Marcinkiewicz spaces, are detected.
The reduction principle hinges on a preliminary  discussion of the existence and uniqueness of generalized solutions to equations, in the Gauss space, for the Ornstein-Uhlenbeck operator, with a just integrable right-hand side.
A~decisive role is also played by a pointwise estimate, in rearrangement form, for these solutions.
\end{abstract}

\author{Andrea Cianchi\textsuperscript{1}}
\address{\textsuperscript{1}Dipartimento di Matematica e Informatica ``Ulisse Dini'',
University of Florence,
Viale Morgagni 67/A, 50134
Firenze,
Italy}
\email{andrea.cianchi@unifi.it}
\urladdr{0000-0002-1198-8718}

\author{V\'\i t Musil\textsuperscript{2}}
\address{\textsuperscript{2} Department of Computer Science, Faculty of Informatics, Masaryk University,
Botanick\'a 554/68a, 602~00, Brno,
Czech Republic}
\email{musil@fi.muni.cz}
\urladdr{0000-0001-6083-227X}

\author{Lubo\v s Pick\textsuperscript{3}}
\address{\textsuperscript{3}Department of Mathematical Analysis,
Faculty of Mathematics and Physics,
Charles University,
So\-ko\-lo\-vsk\'a~83,
186~75 Praha~8,
Czech Republic}
\email{pick@karlin.mff.cuni.cz}
\urladdr{0000-0002-3584-1454}

\date{\today}

\subjclass[2000]{46E35, 28C20}
\keywords{}

\maketitle

\section*{How to cite this paper}
\noindent
This paper has been accepted for publication in \emph{Journal of Differential Equations} and is available on
\begin{center}
	\url{https://doi.org/10.1016/j.jde.2023.02.035}.
\end{center}
Should you wish to cite this paper, the authors would like to cordially ask you
to cite it appropriately.

\section{Introduction}

The present paper deals with norm estimates for functions in the Gauss space in terms of the Ornstein-Uhlenbeck operator.
Norms depending on global integrability properties of functions, namely rearrangement-invariant function norms, are considered.

Specifically, we are concerned with inequalities of Sobolev type of the form
\begin{equation} \label{aug100}
	\|u - \m(u)\|_{Y\RG}
		\le c \|\lgn u\|_{X\RG}
\end{equation}
for some constant $c$ and for all functions $u\colon\rn\to\R$ such that $\lgn u\in X\RG$.
Here, $\RG$ denotes the Gauss space, namely the space $\rn$ equipped with the Gauss measure
$\gamma_n$ obeying
\begin{equation} \label{gaussmeas}
	\dgn(x) = \tfrac{1}{(2\pi)^{\frac{n}{2}}} e^{-\frac{|x|^2}{2}} \,\d x
		\quad\text{for $x\in\rn$,}
\end{equation}
and $\lgn$ stands for the Ornstein-Uhlenbeck operator, formally defined as
\begin{equation} \label{E:lgn-def}
	\lgn u = \Delta u - x \cdot \nabla u,
\end{equation}
where $\Delta$ and $\nabla$ are the classical Laplace and gradient operator.
Moreover, $X\RG$ and $Y\RG$ are rearrangement-invariant spaces, and $\m(u)$ stands for either the mean value or the median of $u$ over $\RG$.

Being  the infinitesimal generator of the Ornstein-Uhlenbeck semigroup in the Gauss space is one of the reasons that makes the  operator $\lgn$ of primary importance in the Gaussian setting.
It, therefore, plays a role parallel to that of the Laplace operator---the infinitesimal generator of the heat kernel---in the Euclidean space.
The Ornstein-Uhlenbeck operator enters various fields and has hence been extensively investigated in the literature.
For an introduction to its theory, we refer to the monograph by~\citet{Urb:19}, the lecture notes by~\citet{Lun:15}, and the survey papers by~\citet{Sjo:97} and \citet{Bog:18}.

Inequalities of the form \eqref{aug100} can  be regarded as Gaussian analogues of Sobolev inequalities for the Laplacian in the Euclidean framework.
Gaussian Sobolev inequalities typically differ from their Euclidean counterparts because of the behaviour of the density of the Gauss measure near infinity.
This feature also shapes inequalities for the Ornstein-Uhlenbeck operator.
A general trait of the latter inequalities is that the improvement of the degree of integrability for a function $u$ guaranteed from that of $\lgn u$ in the space $\RG$ is considerably lesser than that entailed by $\Delta u$ in domains of finite Lebesgue measure in $\rn$.
For instance, no inequality of the form \eqref{aug100} holds with $Y\RG = L^\infty \RG$, whatever $X\RG$ is.
By contrast, unlike the Euclidean inequalities, constants in the Gaussian inequalities are dimension-free.
This accounts for their use in questions of probability theory, where the dimension  $n$ is usually sent to infinity, and for the derivation of Sobolev-type inequalities in infinite dimensional spaces  -- see \eg the classical papers ~\citep{Nel:73,Fei:75,Rot:85,Wei:79}.

The analysis of Sobolev inequalities in the Gauss space was pioneered by Gross in the paper \cite{Gro:75}, where a sharp first-order inequality for the $L^2\RG$ norm of the gradient was established.
A vast amount of literature on Gaussian Sobolev-type inequalities has flourished over the years on the wake of Gross' work.
A very limited sample of contributions on this topic includes \citep{Bar:06, Bar:08, Bob:97, Bob:98, Bra:07, Car:01, Cia:20, Cia:22, Cip:00, Fei:75, Fuj:11, Mil:09, Pel:93, Rot:85}.
In particular, inequalities involving the Ornstein-Uhlenbeck operator are derived in \citep{Bet:02,Bla:07,Tia:10}.
The latter papers deal, in fact, with even more general second-order elliptic operators but focus on a different class of functions, which are defined in open subsets $\Omega\subsetneq\rn$ and vanish on $\partial\Omega$.

Our purpose is to offer a unified approach in detecting the optimal spaces $X\RG$ and $Y\RG$ in inequalities of the form~\eqref{aug100}.
More precisely, we are aimed at characterizing the optimal (smallest possible) target space $Y\RG$ in inequality~\eqref{aug100} associated with a given domain space $X\RG$, and, conversely,  the optimal (largest possible) domain space $X\RG$ associated with a given target space $Y\RG$, within certain classes of rearrangement-invariant spaces.
One main result of this paper provides us with necessary and sufficient conditions for the existence of these optimal spaces in the class of all rearrangement-invariant spaces and exhibits an expression of their norms.

A critical step in our method is a reduction principle on the equivalence of any inequality of the form~\eqref{aug100} to a one-dimensional inequality for a Calder\'on type operator modelled upon the isoperimetric function of the Gauss space.
This principle also enables us to determine optimal targets and domains in inequalities for special families of rearrangement-invariant spaces, such as Orlicz spaces, Lorentz-Zygmund spaces, Marcinkiewicz spaces.

The point of departure for the reduction principle is, in turn, a pointwise inequality for the decreasing rearrangement of a function $u$ in terms of that of $\lgn u$.
Inequalities of this kind for the Laplacian in Euclidean domains are a special case of a classical result of \citet{Tal:76}.
They rest upon differential inequalities on level sets of functions and on the Euclidean isoperimetric inequality, which were also earlier used by Maz'ya \citep{Maz:61,Maz:69} in the proof of estimates for solutions to boundary value problems for classes of elliptic equations in even more general contexts.
Inequalities in the same vein for solutions to Dirichlet problems, with homogeneous boundary conditions, for $\lgn u$ (and more general differential operators whose ellipticity is governed by the Gaussian density) on subsets of the Gauss space are studied in~\citep{Bet:02}, and are also reproduced in~\citep{Mar:10}.

Although the proof of the rearrangement estimate to be exploited here follows along the same lines as those of the contributions mentioned above, some technical issues arise.
They are due to the fact that  functions $u$ defined on the entire space $\rn$ are considered and, especially, that $\lgn u$ is assumed to belong to an arbitrary rearrangement-invariant space, and hence can possibly suffer from very weak integrability properties.
This calls for an extension of the Ornstein-Uhlenbeck operator $\lgn$ beyond its natural domain via a definition patterned on those introduced in the theory of elliptic partial differential equations on Euclidean domains, with merely integrable right-hand sides.
With this regard, an additional result of independent interest will be presented concerning the existence and uniqueness (up to additive constants) of a generalized solution $u$ to the equation
\begin{equation}	\label{equation}
	\lgn u =  f
	\quad\text{in $\RG$,}
\end{equation}
for every $f\in L^1\RG$.
Under this assumption on $f$, an optimal regularity estimate for $u$ and $\nabla u$ in Marcinkiewicz-type spaces is also offered.

Our discussion begins with equation~\eqref{equation}.
The existence and uniqueness of the solution, and a fundamental rearrangement estimate, are addressed in Section~\ref{solutions}, after recalling the necessary background on the functional setting in Section~\ref{spaces}.
Section~\ref{main} contains the reduction principle and the characterization of the optimal target and domain in inequality~\eqref{aug100} in the class of all rearrangement-invariant spaces.
The rest of the paper is devoted to identifying such optimal spaces in customary and less conventional families of rearrangement-invariant spaces.
Specifically, Orlicz spaces are considered in Section~\ref{orlicz}, Lorentz and Lorenz-Zygmund spaces are the subjects of Section~\ref{LZ},  and Marcinkiewicz type spaces are the content of the final Section~\ref{endpoint}.

\section{Function spaces}\label{spaces}

\paragraph*{Measure spaces}

Let $(\RR,\nu)$ be a $\sigma$-finite non-atomic measure space.
We denote by $\MM(\RR,\nu)$ the set of all $\nu$-measurable functions on~$\RR$ taking their values in $[-\infty,\infty]$.
Moreover, we denote by $\Mpl(\RR ,\nu)$ the subset of all nonnegative functions in $\MM(\RR,\nu)$ and by $\MM_0(\RR,\nu)$ the collection of all functions in $\MM(\RR,\nu)$ which are finite almost everywhere on $\RR$.
If $\RR$ is an interval with endpoints $a,b\in[-\infty,\infty]$, $a<b$, and $\nu$ is the one-dimensional Lebesgue measure, then we simply write $\MM(a,b)$, $\Mpl(a,b)$ and $\MM_0(a,b)$.
Furthermore, the Lebesgue measure of a set $E\subset\R^n$ will be denoted by $|E|$.

\paragraph*{Rearrangements}

The \emph{decreasing rearrangement} of a function $\phi\in\MM(\RR,\nu)$ is the function $\phi^*\colon\break(0,\nu(\RR))\to[0,\infty]$ defined as
\begin{equation}
	\phi^*(s)
	 = \inf\bigl\{ t\in\R: \nu\left(\{x\in \RR : |\phi(x)|>t\}\right)\le s\bigr\}
	\quad\text{for $s\in(0,\nu(\RR))$}.
\end{equation}
The \emph{signed decreasing rearrangement} $\phi^\circ\colon(0,\nu(\RR))\to[-\infty,\infty]$
of $\phi\in\MM(\RR,\nu)$ is defined as
\begin{equation}
	\phi^\circ(s)
		= \inf\bigl\{t\in\R : \nu(\{x\in\RR : \phi(x)>t\}) \le s \bigr\}
		\quad\text{for $s\in(0,\nu(\RR))$.}
\end{equation}

A basic property of the decreasing-rearrangement is the \emph{Hardy-Littlewood inequality}, which tells us that
\begin{equation}\label{HL}
	\int_0^{\nu(\RR)} \phi ^*(s) \psi^*(\nu(\RR)-s)\, \d s
 \le \int_{\RR} |\phi (x) \psi (x)|\,\d\nu (x)
		\le \int_0^{\nu(\RR)} \phi ^*(s) \psi^*(s)\, \d s
\end{equation}
for every $\phi, \psi \in  \MM(\RR,\nu)$, provided that $\nu(\RR)<\infty$.
The \emph{maximal non-increasing rearrangement} of $\phi$ is the function $\phi^{**}\colon(0,\nu(\RR))\to[0,\infty]$, defined by
\begin{equation*}
	\phi^{**}(s)
		= \frac{1}{s} \int_0^s \phi^*(r)\,\d r
	\quad\text{for $s\in(0,\nu(\RR))$}.
\end{equation*}
The function $\phi^{**}$ is non-increasing, and one clearly has $\phi^{*}\le \phi^{**}$ on $(0,\nu(\RR))$. The operation $\phi\mapsto\phi^{**}$ is subadditive in the sense that
\begin{equation}\label{subadd}
	\int_0^s (\phi + \psi)^{*}(r)\,\d r
		\le \int_0^s \phi^{*}(r)\,\d r + \int_0^s \psi^{*}(r)\,\d r
		\quad\text{for $s\in(0,\nu(\RR))$}
\end{equation}
for every $\phi,\psi\in\Mpl(\RR,\nu)$.
On the other hand, the operation $\phi\mapsto\phi^{*}$ is not subadditive.
Still, one has that
\begin{equation}\label{gen5}
	(\phi + \psi)^{*}(s)
		\le  \phi^{*}(s/2) + \psi ^{*}(s/2)
	\quad\text{for $s\in(0,\nu(\RR))$}
\end{equation}
for every $\phi,\psi\in\Mpl(\RR,\nu)$.
The functions $\phi,\psi\in\MM(\RR,\nu)$ will be called \emph{equimeasurable} if $\phi^*=\psi^*$ on $(0,\nu(\RR))$.

Some central results in theory of rearrangements are consequences of \emph{Hardy's lemma}, which states that, given $L\in(0,\infty]$, if the functions $g_1, g_2 \in \Mpl (0,L)$ are such that
\begin{equation*}
	\int_0^s g_1(r)\,\d r
		\le \int_0^s g_2(r)\,\d r
		\quad\text{for $s\in(0,L)$,}
\end{equation*}
then
\begin{equation*}
	\int_0^L g_1(r)h(r)\,\d r
		\le \int_0^L g_2(r)h(r)\,\d r
\end{equation*}
for every nonincreasing function $h\colon(0,L)\to(0,\infty)$.

\paragraph*{Mean values and medians}

Assume that $\nu(\RR)<\infty$.
The \emph{mean value} of  a function  $\phi\in L^{1}(\RR,\nu)$ is given by
\begin{equation} \label{E:equiintegrability-with-signed}
	\mv(\phi)
		= \frac 1{\nu(\RR)}\int_{\RR} \phi(x)\, \d \nu(x)
		= \frac 1{\nu(\RR)} \int_{0}^{\nu(\RR)}\phi^\circ(s)\,\d s.
\end{equation}
Also, for  $\phi \in\MM(\RR,\nu)$, we define its \emph{median} by
\begin{equation} \label{median}
	\med(\phi) = \phi^\circ \bigl(\tfrac {\nu(\RR)}2 \bigr).
\end{equation}
The following lemma provides us with a link between $\nu$-\ae convergence of a sequence of functions and pointwise convergence of their signed rearrangement.

\begin{lemma} \label{L:med}
Assume that the sequence $\{\phi_k\} \subset \MM(\RR,\nu)$ and the function $\phi\in\MM(\RR,\nu)$ are such that $\phi_k\to \phi$ $\nu$-\ae in $\RR$.
Then
\begin{equation} \label{E:limsup}
	\phi^\circ(s)
		\le \liminf_{k\to\infty} \phi_k^\circ(s)
		\le \limsup_{k\to\infty} \phi_k^\circ(s)
		\le \phi^\circ(s-)
	\quad\text{for every $s\in(0,\nu(\RR))$}.
\end{equation}
Here, the notation $\phi^\circ(s-)$ stands for the limit of $\phi^\circ$ at $s$ from the left.

In particular, if $\phi^\circ$ is continuous at ${\nu(\RR)}/2$, then
\begin{equation} \label{E:medians}
	\lim_{k\to\infty} \med(\phi_k) = \med(\phi).
\end{equation}
\end{lemma}

\begin{proof}
The proof of the first inequality in~\eqref{E:limsup} is analogous to that for the decreasing rearrangement $\phi^*$ given in \citep[Chapter~2, Proposition~1.7]{Ben:88}.
We may thus limit ourselves to proving the last inequality.
By our assumptions, there exists a  set $N\subset \RR$ such that $\nu(N)=0$ and
\begin{equation} \label{E:limsup-x}
	\lim_{k\to\infty}\sup_{l\ge k} \phi_l(x) \le \phi(x)
		\quad\text{for $x\in\RR\setminus N$}.
\end{equation}
Given $t\in \R$, we define the sets
\begin{equation*}
	E(t) = \{x\in\RR\setminus N: \phi(x) > t\}
	\quad\text{and}\quad
	E_k(t) = \{x\in\RR\setminus N: \phi_k(x) > t\}
\end{equation*}
for $k\in\N$.
Fix $\varepsilon >0$.
Inequality~\eqref{E:limsup-x} implies that, if $x\in E(t-\varepsilon)^c\setminus N$, then there exists $k\in\N$ such that $x\in E_l(t)^c$ for all $l\ge k$, namely
\begin{equation*}
	E(t-\varepsilon)^c \subset \bigcup_{k=1}^\infty \bigcap_{l=k}^\infty E_l(t)^c
	\quad\text{or equivalently}\quad
	E(t-\varepsilon) \supset \bigcap_{k=1}^\infty \bigcup_{l=k}^\infty E_l(t).
\end{equation*}
Here, the suffix ``$c$'' stands for complement in $\RR$.
Therefore,
\begin{equation} \label{E:limpsup-gamma}
	\nu\bigl(E(t-\varepsilon)\bigr)
		\ge \nu \biggl( \bigcap_{k=1}^\infty \bigcup_{l=k}^\infty E_l(t) \biggr)
    = \lim_{k\to\infty} \nu \biggl(\bigcup_{l=k}^\infty E_l(t) \biggr)
    \ge \lim_{k\to\infty} \sup_{l\ge k} \nu\bigl(E_l(t)\bigr).
\end{equation}
Define the functions $\mu\colon\R\to[0, \infty)$ as $\mu(t)=\nu(E(t))$ and  $\mu_k\colon(0,\infty)\to[0,\infty)$ as $\mu_k(t)=\nu(E_k(t))$ for $k \in \N$.
Inequality~\eqref{E:limpsup-gamma} reads
$\mu(t-\varepsilon)\ge\limsup_{k\to\infty} \mu_k(t)$. Hence,  on passing to the limit as
$\varepsilon\to0^+$ one deduces that
\begin{equation} \label{E:limsup-circ}
	\mu(t-)\ge \limsup_{k\to\infty} \mu_k(t).
\end{equation}
On the other hand, a  property analogous to that established  in
 \citep[Chapter~2, Proposition~1.7]{Ben:88} for $\phi^*$ ensures that
\begin{equation} \label{E:liminf-circ}
	\mu(t)\le \liminf_{k\to\infty} \mu_k(t).
\end{equation}
Since the function  $\mu$ is non-increasing,  combining inequalities
\eqref{E:limsup-circ} and \eqref{E:liminf-circ} implies that $\mu_k\to\mu$ $\nu$-\ae in $\R$.
Since
\begin{equation*}
	\phi^\circ(s)=|\{t\in \R: \mu(t)>s\}|,
\end{equation*}
the last inequality in \eqref{E:limsup} follows from an analogue of~\eqref{E:limsup-circ}, with $\mu$ and $\mu_k$ replaced by $\phi$ and $\phi_k$ respectively, and  the measure $\nu$ by the Lebesgue measure.

Finally, equation \eqref{E:medians} follows from an application of equation \eqref{E:limsup} with $s={\nu(\RR)}/2$.
\end{proof}

\paragraph*{Function norms and rearrangement-invariant spaces}

Let $L\in(0,\infty]$.
A \emph{function norm} is defined as a functional $\|\cdot\|_{X(0,L)}\colon \Mpl(0,L)\to[0,\infty]$ satisfying, for all $g$, $h$ and $\{g_k\}\subset{\Mpl(0,L)}$, and every $\lambda {\in[0,\infty)}$:
\begin{enumerate}[label={\rm(P\arabic*)}, leftmargin=8ex]
\item\label{en:p1}
$\|g\|_{X(0,1)}=0$ if and only if $g=0$ \ae,\\
$\|\lambda g\|_{X(0,1)}= \lambda \|g\|_{X(0,1)}$,\\
$\|g+h\|_{X(0,1)}\le \|g\|_{X(0,1)}+ \|h\|_{X(0,1)}$;
\item\label{en:p2} If $g\le h$ \ae, then $\|g\|_{X(0,L)}\le\|h\|_{X(0,L)}$;
\item\label{en:p3} If $g_k\uparrow g$ \ae, then $\|g_k\|_{X(0,L)}\uparrow\|g\|_{X(0,L)}$;
\item\label{en:p4} $\|\chi_{E}\|_{X(0,L)}<\infty$ for every measurable set $E\subset [0,L]$ of finite measure;
\item\label{en:p5} There exists a constant $c$ such that
$\int_E g(s)\,\d s\le c\|g\|_{X(0,L)}$ for every $g\in\Mpl(0,L)$ and every measurable set $E\subset (0,L)$ of finite measure.
\end{enumerate}
If, in addition,
\begin{enumerate}[resume*]
\item\label{en:p6} $\|g\|_{X(0,L)}=\|h\|_{X(0,L)}$ whenever $g^* = h^*$,
\end{enumerate}
then we say that $\|\cdot\|_{X(0,1)}$ is a \emph{rearrangement-invariant function norm}.

The \emph{associate function norm} $\|\cdot\|_{X'(0,L)}$ of a function norm $\|\cdot\|_{X(0,L)}$ is defined by
\begin{equation} \label{norm}
	\|g\|_{X'(0,L)}
		= \sup
			\left\{
				\int_0^{L} g(s)h(s)\,\d s:
					h\in{\Mpl(0,L)},\, \|h\|_{X(0,L)}\le 1
			\right\}
\end{equation}
for $g\in\Mpl(0,L)$.
Note that
\begin{equation} \label{X''}
	\|\cdot \|_{(X')'(0,L)}
		= \|\cdot \|_{X(0,L)}.
\end{equation}
The  \emph{rearrangement-invariant space} $X(\RR,\nu)$ built upon the function norm $\|\cdot\|_{X(0,\nu(\RR))}$ is defined as the collection of all functions $\phi\in\MM(\RR,\nu)$ such that the quantity $\|\phi\|_{X(\RR,\nu)}$, given by
\begin{equation}\label{ri-norm}
    \|\phi\|_{X(\RR,\nu)} = \|\phi^*\|_{X(0,\nu(\RR))},
\end{equation}
is finite.
The space $X(\RR,\nu)$ is a Banach space, endowed with the norm given by \eqref{ri-norm}.
The space $X(0,\nu(\RR))$ is called the \emph{representation space} of $X(\RR,\nu)$.

We shall also employ the notation
\begin{equation}\label{perp}
  X_{\perp}(\RR,\nu)
		= \bigl\{ u \in X(\RR,\nu): \mv (u)=0\bigr\}.
\end{equation}
The \emph{associate space} $X'(\RR,\nu)$ of a~rearrangement-invariant~space $X(\RR,\nu)$ is the rearrangement-invariant space defined via the function norm $\|\cdot\|_{X'(0,\nu(\RR))}$.
The generalized \emph{H\"older inequality}
\begin{equation}\label{holder}
	\int_{\RR}|\phi(x)\psi(x)|\,\d\nu(x)
		\le \|\phi\|_{X(\RR,\nu)} \|\psi\|_{X'(\RR,\nu)}
\end{equation}
holds for every $\phi$ and $\psi$ in $\MM(\RR,\nu)$.
The \emph{fundamental function} $\varphi_X\colon[0,\nu(\RR))\to[0,\infty)$ of a rearrangement-invariant space $X(\RR,\nu)$ is defined as
\begin{equation}\label{fund}
	\varphi_X (t) = \|\chi _E\|_{X(\RR,\nu)}
		\quad \text{for $t\in [0,\nu(\RR))$,}
\end{equation}
where $E$ is any subset of $\RR$ such that $\nu(E)=t$. One has that
\begin{equation}\label{aug19}
	\varphi_X (t) \varphi_{X'}(t) = t
		\quad \text{for $t\in [0,\nu(\RR))$,}
\end{equation}
for every rearrangement-invariant space $X(\RR,\nu)$.

The following lemma extends \citep[Lemma 2.1]{Mil:09}  to arbitrary rearrangement-invariant spaces.

\begin{lemma} \label{med-mv}
Let $X(\RR,\nu)$ be any rearrangement-invariant space on a finite measure space
$(\RR, \nu)$. Then
\begin{equation}\label{aug20}
	\thalf \|\phi - \mv(\phi)\|_{X(\RR,\nu)}
		\le	 \|\phi - \med(\phi)\|_{X(\RR,\nu)}
		\le 3 \|\phi - \mv(\phi)\|_{X(\RR,\nu)}
\end{equation}
for every function $\phi\in X(\RR,\nu)$.
\end{lemma}

\begin{proof}
\newcommand{\nR}{L}
Set $L=\nu(\RR)$. Owing to the H\"older inequality~\eqref{holder},
\begin{align} \label{aug21}
	|\mv(\phi) - \med(\phi)|
		\le \frac 1{\nR} \int_{\RR} |\phi - \med (\phi)| \,\d\nu
		\le \frac{\varphi _{X'}(\nR)}{\nR}
					\|\phi - \med(\phi)\|_{X(\RR, \nu)}
\end{align}
for $\phi\in X(\RR,\nu)$.
From this inequality and equation~\eqref{aug19} we deduce that
\begin{align} \label{aug22}
	\begin{split}
	\|\phi - \mv(\phi)\|_{X(\RR,\nu)}
		& \le \|\phi - \med(\phi)\|_{X(\RR,\nu)} +
					\|\med(\phi) - \mv(\phi)\|_{X(\RR,\nu)}
			\\
		& \le \|\phi - \med(\phi)\|_{X(\RR,\nu)} +
					\varphi _{X}(\nR) |\mv(\phi) - \med(\phi)|
			\\
		& \le \|\phi - \med(\phi)\|_{X(\RR,\nu)} +
						\frac{\varphi_{X}(\nR)\varphi_{X'}(\nR)}{\nR}
						\|\phi - \med(\phi)\|_{X(\RR,\nu)}
			\\
		& = 2 \|\phi - \med(\phi)\|_{X(\RR,\nu)},
	\end{split}
\end{align}
namely the first inequality in \eqref{aug20}.
As for the second one, we may assume, or replacing $\phi$ by $-\phi$, if necessary, that $\med(\phi)\ge \mv(\phi)$.
Let $E=\{|\phi-\mv(\phi)|\ge\med(\phi)-\mv(\phi)\}$ and observe that
\begin{align}	\label{aug23a}
	\|\phi - \mv (\phi)\|_{X(\RR,\nu)}
		\ge \|(\phi - \mv (\phi)) \chi_{E}\|_{X(\RR,\nu)}
		\ge (\med(\phi)-\mv(\phi)) \|\chi_{E}\|_{X(\RR,\nu)}
\end{align}
and
\begin{align}	\label{aug23b}
	\|\chi_{E}\|_{X(\RR,\nu)}
		= \varphi_X (|E|)
		= \frac{|E|}{\varphi_{X'} (|E|)}
		\ge \frac{|E|}{\varphi_{X'}(\nR)}.
\end{align}
Since
\begin{align}\label{aug24}
	\frac{\nR}2
		\le |\{\phi\ge \med(\phi)\}|
		\le |E|,
\end{align}
inequalities \eqref{aug23a} and \eqref{aug23b} imply that
\begin{align}\label{aug24-bis}
	\med(\phi)-\mv(\phi)
		\le 2 \frac{\varphi_{X'} (\nR)}{\nR}
				\|\phi - \mv (\phi)\|_{X(\RR, \nu)}.
\end{align}
Hence,
\begin{align}\label{aug25}
	\begin{split}
	\|\phi - \med(\phi)\|_{X(\RR,\nu)}
		& \le \|\phi - \mv(\phi)\|_{X(\RR,\nu)}
			+ \|\mv(\phi) - \med(\phi)\|_{X(\RR,\nu)}
		\\
		& = \|\phi - \mv(\phi)\|_{X(\RR,\nu)}
			+ (\med(\phi)-\mv(\phi)   ) \varphi_X(\nR)
		\\
		& \le \|\phi - \mv(\phi)\|_{X(\RR,\nu)}
			+ 2 \frac{\varphi_X(\nR) \varphi_{X'}(\nR)}{\nR}
					\|\phi - \mv(\phi)\|_{X(\RR,\nu)}
 		\\
		& = 3 \|\phi - \mv(\phi)\|_{X(\RR,\nu)}.
	\end{split}
\end{align}
The second inequality in~\eqref{aug20} is thus also established.
\end{proof}

\paragraph*{Embeddings and boundedness of operators}

Let $X(\RR,\nu)$ and $Y(\RR,\nu)$ be rearrangement-invari\-ant spaces.
We write $X(\RR,\nu)\to Y(\RR,\nu)$ to denote that $X(\RR,\nu)$ is continuously embedded into $Y(\RR,\nu)$, in the sense that there exists a constant $c$ such that $\|\phi\|_{Y(\RR,\nu)}\le c\|\phi\|_{X(\RR,\nu)}$ for every $\phi\in\MM(\RR,\nu)$.
Note that the embedding $X(\RR,\nu)\to Y(\RR,\nu)$ holds if and only if there exists a constant $c$ such that $\|g\|_{Y(0,\nu(\RR))}\le c\|g\|_{X(0,\nu(\RR))}$ for every $g\in\Mpl(0,\nu(\RR))$.
A property of function norms ensures that
\begin{equation*}
	X(\RR,\nu) \subset Y(\RR,\nu)
		\quad\text{if and only if}\quad
	X(\RR,\nu)\to Y(\RR,\nu).
\end{equation*}

Let $L\in(0,\infty]$.
We say that an operator $T$ is bounded from a rearrangement-invariant space $X(0,L)$ into a rearrangement-invariant space $Y(0,L)$, and we write
\begin{equation}\label{Tembed}
	T\colon X(0,L)\to Y(0,L),
\end{equation}
if $T$ maps functions from $\Mpl(0,L)$ into functions from  $\Mpl(0,L)$,
and
its norm, defined by
\begin{equation*}
	\|T\| = \sup
		\left\{
			\|Tg\|_{Y(0,L)}:
			g\in X(0,L)\cap\Mpl(0,L),\, \|g\|_{X(0,L)}\le 1
		\right\},
\end{equation*}
is finite.

The space $Y(0,L)$ will be called the \emph{optimal target} space in \eqref{Tembed}, within a certain class of rearrangement-invariant spaces if, whenever $Z(0,L)$ is another rearrangement-invariant space from  the same class such that $T\colon X(0,L)\to Z(0,L)$, then $Y(0,L)\to Z(0,L)$.
Analogously, the \emph{optimal domain} space $X(0,L)$ in \eqref{Tembed} within the relevant class obeys $Z(0,L)\to X(0,L)$ whenever $T\colon Z(0,L)\to Y(0,L)$.

Assume that the operators $T$ and $T'$, acting from $\Mpl(0,L)$ into $\Mpl(0,L)$, are such that
\begin{equation} \label{E:duo}
	\int_0^L Tg(s)h(s)\,\d s
		= \int_0^L g(s)T'h(s)\,\d s
\end{equation}
for every $g,h\in\MM_+(0,L)$. We then call the operator $T'$  \emph{adjoint} to $T$.
A simple argument involving Fubini's
theorem and the definition of the associate norm shows that
\begin{equation}\label{E:novabis}
	T\colon X(0,L)\to Y(0,L)
		\quad\text{if and only if}\quad
	T'\colon Y'(0,L)\to X'(0,L)
\end{equation}
and $\|T\|=\|T'\|$.

\paragraph*{Lorentz--Zygmund spaces}

A pivotal class of examples of rearrangement-invariant function norms is constituted by the Lebesgue norms.
The Lebesgue functional $\|\cdot\|_{L^p(0,L)}$, defined as usual for $p\in(0,\infty]$, is a rearrangement-invariant function norm if and only if $p\in[1,\infty]$.

If $L<\infty$, then a generalization of the Lebesgue functionals is provided by the \emph{Lorentz--Zygmund functionals} $\|\cdot\|_{L^{p,q;\alpha,\beta}(0,L)}\colon \Mpl(0,L)\to[0,\infty]$ and $\|\cdot\|_{L^{(p,q;\alpha,\beta)}(0,L)}\colon \Mpl(0,L)\to[0,\infty]$, defined by
\begin{equation*}
	\|g\|_{L^{p,q;\alpha,\beta}(0,L)}
		= \bigl\|s^{\ip-\iq}\ell(s/L)^{\alpha}\ell\ell(s/L)^{\beta}g^{*}(s)\bigr\|_{L^{q}(0,L)}
\end{equation*}
and
\begin{equation*}
	\|g\|_{L^{(p,q;\alpha,\beta)}(0,L)}
		= \bigl\|s^{\ip-\iq}\ell(s/L)^{\alpha}\ell\ell(s/L)^{\beta}g^{**}(s)\bigr\|_{L^{q}(0,L)}
\end{equation*}
for $g\in\Mpl(0,L)$, respectively.
Here, $p,q\in(0,\infty]$, $\alpha,\beta\in\R$ and
\begin{equation*}
    \ell(s)=1+\log\frac{1}{s},\ \ell\ell(s)=1+\log\Bigl(1+\log\frac{1}{s}\Bigr)
		\quad\text{for $s\in(0,1)$.}
\end{equation*}
For finite measure spaces $(\RR,\nu)$, the corresponding \emph{Lorentz--Zygmund spaces} $L^{p,q;\alpha,\beta}(\RR,\nu)$ and $L^{(p,q;\alpha,\beta)}(\RR,\nu)$ are (equivalent to) the rearrangement-invariant spaces built upon the function norms $\|\cdot\|_{L^{p,q;\alpha,\beta}(0,\nu(\RR))}$ and $\|\cdot\|_{L^{(p,q;\alpha,\beta)}(0,\nu(\RR))}$, respectively, if and only if one of the following conditions holds:
\begin{equation*}
    \begin{cases}
        1<p<\infty,\ 1\le q\le\infty,\ \alpha\in\R,\ \beta\in\R
            \\
        p=1,\ q=1,\ \alpha>0,\ \beta\in\R
            \\
        p=1,\ q=1,\ \alpha=0,\ \beta>0
          \\
        p=\infty,\ 1\le q\le\infty,\ \alpha+\frac{1}{q}<0,\ \beta\in\R
          \\
        p=\infty,\ 1\le q\le\infty,\ \alpha+\frac{1}{q}=0,\ \beta+\frac{1}{q}<0
          \\
        p=\infty,\ q=\infty,\ \alpha=0,\ \beta=0
   \end{cases}
\end{equation*}
in case of $L^{p,q;\alpha,\beta}(\RR,\nu)$, and
\begin{equation*}
    \begin{cases}
        0<p<\infty,\ 1\le q\le\infty,\ \alpha\in\R,\ \beta\in\R
          \\
        p=\infty,\ 1\le q\le\infty,\ \alpha+\frac{1}{q}<0,\ \beta\in\R
          \\
        p=\infty,\ 1\le q\le\infty,\ \alpha+\frac{1}{q}=0,\ \beta+\frac{1}{q}<0
          \\
        p=\infty,\ q=\infty,\ \alpha=0,\ \beta=0
   \end{cases}
\end{equation*}
in case of $L^{(p,q;\alpha,\beta)}(\RR,\nu)$.
We shall also write $L^{p,q}(\log L)^{q\alpha}(\log\log L)^{q\beta}(\RR,\nu)$ instead of $L^{p,q;\alpha,\beta}(\RR,\nu)$, and $L^{p,q}(\log L)^{q\alpha}(\RR,\nu)$ instead of $L^{p,q;\alpha,0}(\RR,\nu)$.

\paragraph*{Orlicz spaces}

A generalization of the Lebesgue spaces in a different direction is that of the Orlicz spaces.
They are generated by the so-called Luxemburg functionals, whose definition, in turn, rests upon that of Young function.
A \emph{Young function} $A\colon[0,\infty)\to[0,\infty]$ is a convex, left-continuous function such that $A(0)=0$ and $A$ is not constant in $(0,\infty)$.
The function $\widetilde A\colon[0,\infty)\to[0,\infty]$ denotes the \emph{Young conjugate} of $A$, and is defined as
\begin{equation*}
	\widetilde A(t)
		= \sup\{\tau t-A(\tau): \tau\ge 0\}
	\quad\text{for $t\ge 0$}.
\end{equation*}
The latter is also a Young function and its conjugate is $A$ again.
One has that
\begin{equation} \label{E:conjugate-invers}
	t\le A^{-1}(t)\,\widetilde{A}^{-1}(t)\le 2t
		\quad\text{for $t\ge 0$},
\end{equation}
where $A^{-1}$ denotes the generalized right-continuous inverse of $A$.
The function $B$, defined as $B(t)=cA(bt)$, where $b,c$ are positive constants, is also a Young function and
\begin{equation} \label{E:young-compl}
	\widetilde{B}(t) = c\widetilde{A}\bigl(\tfrac{t}{bc}\bigr)
		\quad\text{for $t\ge 0$.}
\end{equation}
A Young function $A$ is said to satisfy the \emph{$\Delta_2$-condition} near infinity if it is finite-valued and there exist constants $c>0$ and $t_0> 0$ such that
\begin{equation*}
	A(2t) \le c A(t)
		\quad\text{for $t\ge t_0$.}
\end{equation*}
Moreover, $A$ fulfills the \emph{$\nabla_2$-condition} near infinity if there exist constants $c>2$ and $t_0> 0$ such that
\begin{equation*}
	  A(2t) \ge c A(t)
		\quad\text{for $t\ge t_0$.}
\end{equation*}
These conditions are said to be satisfied  globally if they hold with $t_0=0$.

A Young function $A$ is said to \emph{dominate} another Young function $B$ near infinity if there exist constants $c>0$ and $t_0> 0$ such that
\begin{equation*}
	B(t) \le A(ct)
		\quad\text{for $t\ge t_0$}.
\end{equation*}
The function $A$ dominates $B$ globally if $t_0=0$.
The functions $A$ and $B$ are called \emph{equivalent} near infinity [globally] if they dominate each other near infinity [globally].

Given $L\in(0,\infty]$ and a Young function $A$, the \emph{Luxemburg functional} $\|\cdot\|_{L^A(0,L)}\colon \Mpl(0,L)\to[0,\infty]$ is defined by
\begin{equation*}
	\|g\|_{L^{A}(0,L)} = \inf
		\left\{
			\lambda>0:
				\int_{0}^{L} A\left(\frac{g(s)}{\lambda}\right)\d s\le 1
		\right\}
\end{equation*}
for $g\in\Mpl(0,L)$.
The corresponding \emph{Orlicz space} $L^{A}(\RR,\nu)$ is built on the function norm $ \|\cdot\|_{L^{A}(0,\nu(\RR))}$.

When $\nu (\RR)< \infty$, the alternate notation $A(L)(\RR,\nu)$ will also be employed when convenient to denote an Orlicz space associated with a Young function which agrees
with $A$ near infinity.

Also, if $\phi\in L^{A}(\RR,\nu)$ and $E\subset\RR$ is a measurable set, then we often use the abridged notation
\begin{equation*}
	\|\phi\|_{L^A(E, \nu)}
		= \|\phi\chi_E\|_{L^A(\RR,\nu)}.
\end{equation*}

For Young functions $A$ and $B$,
the embedding $L^A(\RR,\nu)\to L^B(\RR,\nu)$ holds if and only if either $\nu(\RR)<\infty$ and $A$ dominates $B$ near infinity, or $\nu(\RR)=\infty$ and $A$ dominates $B$ globally.
One has that $L^A(\RR,\nu)=L^B(\RR,\nu)$ up to equivalent norms if and only if either $\nu(\RR)<\infty$ and $A$ and $B$ are equivalent near infinity, or $\nu(\RR)=\infty$ and  $A$ and $B$ are equivalent globally.

One has that
\begin{equation*}
	L^A(\RR,\nu)' = L^{\tilde A}(\RR,\nu)
\end{equation*}
up to equivalent norms.

We will also need certain weak and strong versions of Orlicz spaces.
These will be defined in the next paragraph as particular cases of more general families of rearrangement-invariant spaces.

\paragraph*{Endpoint spaces}

Let $L\in(0,\infty]$.
A function $\varphi\colon[0,L]\to[0, \infty)$ will be called quasiconcave if it is non-decreasing, vanishes only at $0$, and the function $\tvp$, defined by
\begin{equation}\label{dec40}
	\tvp(s) = \frac{s}{\varphi(s)}
		\quad\text{for $s\in (0, L]$}
\end{equation}
and $\tvp(0)=0$, is non-increasing on $(0,L]$.
The \emph{Lorentz functional} $\|\cdot\|_{\Lambda_{\varphi}(0,L)}\colon \Mpl(0,L)\to[0,\infty]$ is defined as
\begin{equation*}
	\|g\|_{\Lambda_{\varphi}(0,L)}
		= \int_{0}^{L} g^{*}(s) \,\d \varphi(s)
\end{equation*}
for $g\in\Mpl(0,L)$.
The \emph{Lorentz endpoint space} $\Lambda_{\varphi}(\RR,\nu)$ is built on the Lorentz functional $\|\cdot\|_{\Lambda_{\varphi}(0,\nu(\RR))}$.
The \emph{Marcinkiewicz functional} $\|\cdot\|_{M_{\varphi}(0,L)}\colon \Mpl(0,L)\to[0,\infty]$ is defined by
\begin{equation*}
	\|g\|_{M_{\varphi}(0,L)}
		= \sup_{s\in(0,L)}\varphi(s)g^{**}(s)
\end{equation*}
for $g\in\Mpl(0,L)$, and the corresponding Marcinkiewicz space is defined through the Marcinkiewicz functional $ \|\cdot\|_{M_{\varphi}(0,\nu(\RR))}$.
A variant of the Marcinkiewicz functional is denoted by $\|\cdot\|_{m_{\varphi}(0,L)}$, and is obtained on replacing $g^{**}$ with $g^*$ in its definition.
Namely,
\begin{equation*}
	\|g\|_{m_{\varphi}(0,L)}
		= \sup_{s\in(0,L)}\varphi(s)g^{*}(s)
\end{equation*}
for $g\in\Mpl(0,L)$.
Notice that $\|\cdot\|_{m_{\varphi}(0,L)}$ need not be a rearrangement-invariant function norm in general.  It is however a quasi-norm, in the sense that it satisfies the triangle inequality up to a multiplicative constant.
The space $m_{\varphi}(\RR,\nu)$ will be called the \emph{weak type space} associated with $\varphi$.
We point out that the Lorentz-Zygmund spaces $L^{p,\infty;\alpha,\beta}(\RR,\nu)$ defined above are weak type spaces according to this definition.

Let us recall that
\begin{equation}\label{dic30}
	\Lambda_{\varphi}(\RR,\nu)' = M_{\tvp}(\RR,\nu)
		\quad\text{and}\quad
	M_{\varphi}(\RR,\nu)' = \Lambda_{\tvp}(\RR,\nu),
\end{equation}
see \eg~\citep[Chapter 5, Section 2]{Kre:82}.
For every quasiconcave function $\varphi$, one has
\begin{equation*}
	\|\chi_E\|_{\Lambda_{\varphi}(\RR,\nu)} = \varphi(s)
    \quad\text{and}\quad
	\|\chi_E\|_{M_{\varphi}(\RR,\nu)} = \varphi(s)
		\quad\text{whenever $\nu(E)=s\in[0,\nu(\RR)]$.}
\end{equation*}
Thus,
\begin{equation*}
	\varphi_{\Lambda_{\varphi}(\RR,\nu)}
		= \varphi_{M_{\varphi}(\RR,\nu)}
		= \varphi.
\end{equation*}
Moreover, if $X(\RR,\nu)$ is a rearrangement-invariant space such that $\varphi_X =\varphi$,
then
\begin{equation}\label{E:fundamental-endpoint-embeddings}
    \Lambda_{\varphi}(\RR,\nu) \to X(\RR,\nu) \to M_{\varphi}(\RR,\nu).
\end{equation}

If $A$ is a Young function, then the function $\varphi _A\colon (0, \infty) \to [0, \infty)$, defined as
\begin{equation*}
	\varphi_A(s) = \frac{1}{A^{-1}(\frac{1}{s})}
		\quad\text{for $s>0$,}
\end{equation*}
is concave, and hence quasiconcave.
The \emph{weak Orlicz space} $M^A(\RR,\nu)$ is defined as $M_{\varphi_A}(\RR,\nu)$ and the \emph{strong Orlicz space} $\Lambda^A(\RR,\nu)$ is defined as $\Lambda_{\varphi_A}(\RR,\nu)$.
Owing to equations~\eqref{dic30} and~\eqref{E:conjugate-invers}, one has that
\begin{equation}\label{dic31}
	(\Lambda^A)'(\RR,\nu) = M^{\tilde A}(\RR,\nu)
		\quad\text{and}\quad
	(M^A)'(\RR,\nu) = \Lambda^{\tilde A}(\RR,\nu),
\end{equation}
up to equivalent norms.
The expressions ``weak  Orlicz space'' and ``strong Orlicz spaces'' are adopted consistently with the embeddings
\begin{equation}\label{E:Orlicz-char}
    \Lambda^A(\RR,\nu) \to L^A(\RR,\nu) \to M^A(\RR,\nu),
\end{equation}
which hold for every Young function $A$ and every measure space $(\RR,\nu)$.
Notice that these embeddings can actually be strict.

\section{Extended domain of the Ornstein-Uhlenbeck operator and key estimates}\label{solutions}

The Sobolev space $W^{1,2}\RG$ is defined as
\begin{equation*}
	W^{1,2}\RG
		=\{
				u\in L^2\RG : \text{$u$ is weakly differentiable and $|\nabla u|\in L^2\RG$}
			\}.
\end{equation*}
Similarly,
\begin{equation*}
	W^{2,2}\RG
		= \{
				u\in L^2\RG : \text{$u$ is twice weakly differentiable and
					$|\nabla u|, |\nabla^2 u|\in L^2\RG$}
			\}.
\end{equation*}
The operator $\lgn$ is defined on a function $u\in W^{2,2}\RG$ via equation~\eqref{E:lgn-def}.
One has that $\lgn\colon W^{2,2}\RG\to L^2\RG$.
Moreover,
\begin{equation} \label{E:weak-formulation'}
	\int_{\rn} \nabla u\cdot\nabla v\,\dgn
		= - \int_{\rn} v\,\lgn u\,\dgn
\end{equation}
for every $v\in W^{1,2}\RG$, see \eg~\citep[Theorem 13.1.3]{Lun:15}.

Equation~\eqref{E:weak-formulation'} enables one to extend the operator $\lgn$, and to define $\lgn\colon\Wlgn L^2\RG\to L^2\RG$, where $\Wlgn L^2\RG$ consists of all functions $u\in W^{1,2}\RG$ such that there exists a function $f\in L^2_\perp\RG$ fulfilling
\begin{equation} \label{E:weak-formulation}
	\int_{\rn} \nabla u \cdot \nabla v\,\dgn
		= - \int_{\rn}v\,  f \,\dgn
\end{equation}
for every  $v\in W^{1,2}\RG$.
We then set $\lgn u =f$ for $u \in \Wlgn L^2\RG$.

The operator $\lgn$ can further be extended to the domain $D(\lgn)$, which consist of all functions $u\in L^1\RG$ such that there exists $f\in L^1_\perp\RG$ and a sequence of functions $\{u_k\}\subset\Wlgn L^2\RG$ fulfilling
\begin{equation} \label{apr1}
	u_k \to u
		\quad\text{\ae in $\rn$}
\end{equation}
and
\begin{equation} \label{apr2}
	\lgn u_k \to f
		\quad\text{in $L^1\RG$}.
\end{equation}
We then  set
\begin{equation*}
	\lgn u = f
\end{equation*}
for $u\in D(\lgn)$.
Moreover, given a function space $X\RG$, we define
\begin{equation} \label{Wlgn}
	\Wlgn X\RG
		= \{u\in D(\lgn) : \lgn u\in X\RG\}.
\end{equation}

Our first result, stated in Theorem~\ref{T:existence}, ensures that the operator
\begin{equation}\label{bij}
	\lgn\colon D(\lgn)/c \to L^1_\perp\RG
\end{equation}
is bijective, where $D(\lgn)/c$ denotes the quotient space where two functions in $D(\lgn)$ which differ by a constant are identified.
Theorem \ref{T:existence} also provides us with information on the regularity of functions in $D(\lgn)$.
Their regularity is suitably formulated in terms of membership in spaces of functions whose truncations are Sobolev functions, and in weak type spaces.

Given any $t>0$, denote by $T_{t}\colon\R\rightarrow\R$ the function given by
\begin{equation} \label{502}
	T_{t} (\tau) =
	\begin{cases}
		\tau
			& \text{if $|\tau|\le t$}
			\\
		t \sgn(\tau)
			& \text{if $|\tau|>t$}.
	\end{cases}
\end{equation}
We set
\begin{equation} \label{503}
	\TT^{1,2}\RG
		= \left\{
				u\in\MM\RG :  T_{t}(u)\in W^{1,2}\RG\text{ for every $t>0$}
			\right\}.
\end{equation}
If $u\in \TT^{1,2}\RG$, there exists a (unique) measurable function $Z_u\colon\rn\to\rn$ such that
\begin{equation} \label{504}
	\nabla \big(T_{t}(u)\big)
		= \chi_{\{|u|<t\}} Z_u
		\quad\text{\ae in $\rn$}
\end{equation}
for every $t>0$ \citep[Lemma 2.1]{Ben:95}.
Here, $\chi_E$ denotes the characteristic function of the set $E$.
One has that $u \in W^{1,2}\RG$ if and only if $u\in\TT^{1,2}\RG$ and $Z_u\in L^2\RG$, and, in this case, $Z_u=\nabla u$.
With abuse of notation, if $u\in\TT^{1,2}\RG$, we shall denote $Z_u$ simply by $\nabla u$.

\begin{theorem}[Bijectivity of the map \eqref{bij} and regularity  of functions in $D(\lgn)$] \label{T:existence}
Assume that $f \in L^1_\perp\RG$.
\\ {\rm Part 1.}
There exists a unique (up to additive constants) function $u\in D(\lgn)$, such that
\begin{equation} \label{apr6}
	\lgn u = f.
\end{equation}
Moreover, $u\in\TT^{1,2}\RG$ and any sequence $\{u_k\}$ satisfying \eqref{apr1} and \eqref{apr2} admits a subsequence, still indexed by $k$, such that
\begin{equation} \label{apr10-a}
	\lim_{k\to \infty} \nabla u_k = \nabla u
		\quad\text{\ae in $\rn$}.
\end{equation}
\\ {\rm Part 2.}
Let $u$ be the unique solution to equation~\eqref{apr6}. Then:
\begin{enumerate}
\item $u\in L^{1,\infty}\log L\RG$, and there exists an absolute constant $c$ such that
\begin{equation} \label{apr20}
	\|u - \med(u)\|_{L^{1,\infty}\log L\RG}
		\le c \|f\|_{L^1\RG};
\end{equation}
\item $|\nabla u|\in L^{1,\infty}(\log L)^\half\RG$, and there exists an absolute constant $c$ such that
\begin{equation} \label{apr30}
	\|\nabla u\|_{L^{1,\infty}(\log L)^\half\RG}
		\le c \|f\|_{L^1\RG}.
\end{equation}
\end{enumerate}
The spaces $L^{1,\infty}\log L\RG$ and $L^{1,\infty}(\log L)^\half\RG$ are optimal in inequalities~\eqref{apr20} and~\eqref{apr30}, respectively, among all weak type spaces.
\end{theorem}

We now address the question of a minimal integrability condition on $f$ guaranteeing that the solution $u$ to the equation $\lgn u = f$ be a genuine global Sobolev function, and not just a member of $\TT^{1,2}\RG$.

\begin{proposition}\label{P:classical-sobolev}
Assume that $f\in L^{1,1;\frac{1}{2}}_{\bot}\RG$.
Then the solution $u$ to the equation $\lgn u=f$ obeys $u\in W^{1,1}\RG$.
\end{proposition}

The proof of Theorem~\ref{T:existence} is based on a priori bounds for sequences of approximating functions as in \eqref{apr1}, according to a scheme introduced in \citep{Ben:95}.
These bounds are derived via rearrangement estimates for a function $u \in D(\lgn)$ and its (generalized) gradient $\nabla u$, in terms of $\lgn u$.
The relevant estimates are the subject of Theorem~\ref{T:pointwise-estimates} below, which is a variant of a result from \citep{Bet:02} and relies upon the isoperimetric inequality in the Gauss space.

Recall that the Gauss perimeter $P_{\gamma_n}(E)$ of a measurable set $E\subset \R^n$ can be defined as
\begin{equation*}
	P_{\gamma_n}(E) =
		\Hg(\partial^M E),
\end{equation*}
where
\begin{equation*}
	\dHg (x)=(2\pi)^{-\frac n2}
		 e^{-\frac{|x|^2}2}\,\d\HH^{n-1}(x),
\end{equation*}
$\partial^M E$ denotes the essential boundary of $E$ and $\HH^{n-1}$ is the $(n-1)$-dimensional Hausdorff measure.

The Gaussian isoperimetric inequality tells us that half-spaces minimize Gaussian perimeter among all measurable subsets of $\rn$ with prescribed Gauss measure \citep{Bor:75, Sud:74}.

An analytic formulation of this statement can be given in terms of the isoperimetric function~$I$, also called the isoperimetric profile, of the space $\RG$.
The function $I\colon[0,1]\to [0,\infty)$ is given by
\begin{equation} \label{E:I-def}
	I(s) = \frac{1}{\sqrt{2\pi}} e^{-\frac{\Phi^{-1}(s)^2}{2}}
		\quad\text{ for $s\in(0,1)$,}
\end{equation}
and $I(0)=I(1)=0$, where $\Phi\colon\R\to(0,1)$ is the function defined as
\begin{equation} \label{E:Phi-def}
	\Phi(t) = \frac{1}{\sqrt{2\pi}} \int_{t}^{\infty} e^{-\frac{\tau^2}{2}}\,\d\tau
		\quad\text{for $t\in\R$.}
\end{equation}
The Gaussian isoperimetric inequality then reads
\begin{equation} \label{E:isoperimeric-inequality}
	I(\gamma_n(E)) \le P_{\gamma_n}(E)
\end{equation}
for every measurable subset $E$ of $\R^n$.
Indeed,
\begin{equation*}
	 \gamma_n(\{x\in\rn: x_1\ge t\})= \Phi(t)
		\quad\text{for $t\in\R$,}
\end{equation*}
and
\begin{equation*}
	P_{\gamma_n} (\{x\in\rn: x_1\ge t\})
		=  \frac{1}{\sqrt{2\pi}} e^{-\frac{t^2}{2}}
		\quad\text{for $t\in\R$,}
\end{equation*}
where $x=(x_1,\dots,x_n)$. The function $I$ obeys
\begin{equation}\label{derivative}
\Phi '(t) = - I(\Phi(t)) \quad\text{for $t\in\R$.}
\end{equation}
Also,
\begin{equation} \label{E:I-limit}
	\lim_{s\to 0^+} \frac{I(s)}{s\sqrt{2\ell(s)}} = 1.
\end{equation}
Moreover, on defining the function $\weight\colon\left(0,\half\right)\to(0,\infty)$ by
\begin{equation} \label{E:weight-def}
	\weight(s) = \int_{s}^{\half} \frac{\d r}{I(r)^2}
		\quad\text{for $s\in\left(0,\thalf\right)$},
\end{equation}
one has that
\begin{equation} \label{E:weight-limit}
	\lim_{s\to 0^+} 2s\ell(s)\weight(s) = 1,
\end{equation}
see~\citep[Lemma~4.3]{Cia:21}.

\begin{theorem}[Rearrangement estimates for $u$ and $\nabla u$ via $\lgn u$] \label{T:pointwise-estimates}
Let $u\in D(\lgn)$.
Then
\begin{equation} \label{E:u-rearrangement}
	(u-\med(u))_{\pm}^*(s)
		\le \weight(s) \int_{0}^{s} \left( \lgn u \right)^*_{\mp}(r)\,\d r
			+ \int_{s}^{\half} \left( \lgn u \right)^*_ {\mp}(r)\,\weight(r)\,\d r
		\quad\text{for $s\in\left(0,\thalf\right]$}
\end{equation}
and
\begin{equation} \label{E:gradient-rearrangement}
	|\nabla (u-\med(u))_{\pm}|^*(s)
		\le
			\bigg(
				\frac 2s \int_{\frac s2}^{\half}
					\bigg(
						\int_0^r(\lgn u )^*_{\mp}(\varrho)\,\d \varrho
					\bigg)^2 \frac {\d r}{I(r)^2}
			\bigg)^{\half}
		\quad\text{for $s\in (0,\thalf)$}.
\end{equation}
\end{theorem}

As premised above, the proofs of Theorems~\ref{T:existence} and \ref{T:pointwise-estimates} are mutually related.
The outline is as follows.
We begin by proving inequalities \eqref{E:u-rearrangement} and \eqref{E:gradient-rearrangement} of Theorem~\ref{T:pointwise-estimates} for the restricted class of functions in $\Wlgn L^2\RG$.
These estimates are then applied in the proof of Part~1 of Theorem~\ref{T:existence} to obtain suitable a priori estimates for the approximating functions $u_k$ entering the definition of $\lgn u$ for $u\in D(\lgn)$, and for their gradients.
With this part of Theorem~\ref{T:existence} at our disposal, we are able to pass to the limit in inequalities \eqref{E:u-rearrangement} and \eqref{E:gradient-rearrangement} applied to $u_k$, and conclude their proof for every $u\in D(\lgn)$.
This completes the proof of Theorem \ref{T:pointwise-estimates}.
Inequalities~\eqref{E:u-rearrangement} and \eqref{E:gradient-rearrangement} of Theorem~\ref{T:pointwise-estimates} are a key tool in the proof of estimates~\eqref{apr20} and \eqref{apr30} of Part~2 of Theorem~\ref{T:existence}.

\begin{proof}[Proof of Theorem~\ref{T:pointwise-estimates} for $u\in \Wlgn L^2\RG$] Assume that $u\in \Wlgn L^2\RG$ and set, for simplicity,
$f=\lgn u$. Equation \eqref{E:u-rearrangement} is then equivalent to the couple of inequalities:
\begin{equation} \label{E:pointwise-1}
	0 \le u^\circ(s) - u^\circ(\thalf)
		\le \weight(s) \int_{0}^{s} f^*_-(r)\,\d r
			+ \int_{s}^{\half} f^*_-(r)\,\weight(r)\,\d r
		\quad\text{for $s\in\left(0,\thalf\right]$}
\end{equation}
and
\begin{equation} \label{E:pointwise-2}
	0 \le u^\circ(\thalf) - u^\circ(1-s)
		\le \weight(s) \int_{0}^{s} f^*_+(r)\,\d r
			+ \int_{s}^{\half} f^*_+(r)\,\weight(r)\,\d r
		\quad\text{for $s\in\left(0,\thalf\right]$}.
\end{equation}
We pattern their proof on arguments from \citep{Bet:02} and \citep{Cia:90}, which are in turn adapted from \citep{Tal:76}.
On replacing $u$ by $u-u^\circ(\thalf)$, which is still a solution to equation \eqref{E:weak-formulation}, we may assume that $u^\circ(\thalf)=0$.
We choose the test function $v$ in \eqref{E:weak-formulation} as
\begin{equation*}
	v =
		\begin{cases}
			1
				& \text{if $u>t+h$} \\
			\frac{u-t}{h}
				& \text{if $t<u\le t+h$} \\
			0
				& \text{if $u\le t$}
		\end{cases}
\end{equation*}
for $t\in(0, \esssup u)$ and $h>0$.
Equation \eqref{E:weak-formulation} thus yields
\begin{equation} \label{E:weak-formulation-plugged}
	\frac{1}{h} \int_{\{t<u\le t+h\}} |\nabla u|^2\,\dgn
		=    - \frac 1h\int_{\{t<u\le t+h\}} f(u-t)\,\dgn - \int_{\{u>t+h\}} f\,\dgn.
\end{equation}
Taking the limit as $h\to0^+$ in \eqref{E:weak-formulation-plugged},
we get, by dominated convergence theorem,
\begin{equation} \label{E:diff-ineq}
	-\frac{\d}{\d t} \int_{\{u>t\}} |\nabla u|^2\,\dgn
		= - \int_{\{u>t\}} f\,\dgn
	\quad\text{for \ae $t>0$.}
\end{equation}
Note that the left-hand side of \eqref{E:diff-ineq} is nonnegative.
The Hardy-Littlewood inequality tells us that
\begin{equation} \label{E:HL-f}
	- \int_{\{u>t\}} f\,\dgn
		\le \int_{\{u>t\}} f_-\,\dgn
		\le \int_{0}^{\mu(t)} f^*_-(s)\,\d s
	\quad\text{for $t>0$,}
\end{equation}
where we have set
\begin{equation} \label{E:mu-def}
	\mu(t)=\gamma_n(\{x\in\rn:u(x)>t\})
	\quad\text{for $t\ge 0$.}
\end{equation}
Applying the Cauchy-Schwartz inequality to the difference quotients, we have
\begin{align} \label{E:diff-CS}
	\begin{split}
	-\frac{\d}{\d t} \int_{\{u>t\}} |\nabla u|\,\dgn
		& \le \biggl( -\frac{\d}{\d t} \int_{\{u>t\}} \dgn \biggr)^\half \biggl( -\frac{\d}{\d t} \int_{\{u>t\}} |\nabla u|^2\,\dgn \biggr)^\half
			\\
		& \le \left( -\mu'(t) \right)^\half \biggl( \int_{0}^{\mu(t)} f_-^* (s)\,\d s\biggr)^\half
	\quad\text{for \ae $t>0$,}
	\end{split}
\end{align}
where we have also used \eqref{E:diff-ineq} and \eqref{E:HL-f}.
By the coarea formula we infer that
\begin{equation} \label{E:coarea}
	\int_{\{u>t\}} |\nabla u|\,\dgn
		= \int_{t}^{\infty} P_{\gamma_n}(\{u>\tau\})\,\d\tau
	\quad\text{for $t>0$.}
\end{equation}
Recall that $\mu (t) \le \frac 12$, since we are assuming that $u^\circ(\thalf)=0$.
Thereby, differentiation of \eqref{E:coarea} and the isoperimetric inequality~\eqref{E:isoperimeric-inequality} give
\begin{equation} \label{E:d-nabla-u-I}
	-\frac{\d}{\d t} \int_{\{u>t\}} |\nabla u|\,\dgn
		= P_{\gamma_n}(\{u>t\})
		\ge I\bigl(\mu(t)\bigr)
	\quad\text{for \ae $t>0$.}
\end{equation}
Coupling inequalities \eqref{E:diff-CS} and \eqref{E:d-nabla-u-I} tells us that
\begin{equation} \label{E:1frac}
	1\le \frac{-\mu'(t)}{I\bigl(\mu(t)\bigr)^2} \int_{0}^{\mu(t)} f_-^*(s)\,\d s
		\quad\text{for \ae $t\in(0,\esssup u)$}.
\end{equation}
By integrating inequality \eqref{E:1frac} over  $(0, \tau)$, we get
\begin{equation}\label{E:taufrac} \tau
		\le \int_{ 0}^{\tau} \frac{-\mu'(t)}{I\bigl(\mu(t)\bigr)^2}
			\int_{0}^{\mu(t)} f_-^*(s)\,\d s\,\d t \le \int_{\mu(\tau)}^{\half} \frac{1}{I(\sigma)^2}
		\int_{0}^{\sigma} f_-^*(s)\,\d s\,\d\sigma
		\quad\text{for $\tau\in[0,\esssup u)$}.
\end{equation}
Owing to the definition of the function $u^\circ$, equation \eqref{E:taufrac} yields
\begin{equation} \label{E:u-circ-oper-1}
	0 \le u^\circ(s)
		\le \int_{s}^{\half} \frac{1}{I(r)^2}
			\int_{0}^{r} f_-^*(\varrho)\,\d \varrho\,\d r
		\quad\text{for $s\in\left(0,\thalf\right]$}.
\end{equation}
If $s\in\left(\thalf,1\right)$, an analogous argument and the fact that $I(s)=I(1-s)$ for $s\in [0,1]$ imply that
\begin{equation} \label{E:u-circ-oper-2}
	0 \le - u^\circ(1-s)
		\le \int_{s}^{\half} \frac{1}{I(r)^2}
			\int_{0}^{r} f^*_+(\varrho)\,\d \varrho\,\d r
		\quad\text{for $s\in\left[\thalf,1\right)$}.
\end{equation}
Since, by Fubini's theorem,
\begin{align*}
	\int_{s}^{\half} \frac{1}{I(r)^2} \int_{0}^{r} f_{\pm}^*(\varrho)\,\d \varrho\,\d r
		& = \int_{s}^{\half} \frac{1}{I(r)^2} \int_{0}^{s} f^*_{\pm}(\varrho)\,\d \varrho\,\d r
			+ \int_{s}^{\half} \frac{1}{I(r)^2} \int_{s}^{r} f^*_{\pm}(\varrho)\,\d \varrho\,\d r
			\\
		& = \int_{s}^{\half} \frac{\d r}{I(r)^2} \int_{0}^{s} f^*_{\pm}(\varrho)\,\d \varrho
			+ \int_{s}^{\half} f^*_{\pm}(\varrho)  \int_{\varrho}^{\half} \frac{\d r}{I(r)^2}\,\d \varrho
			\\
		& = \weight(s) \int_{0}^{s} f^*_{\pm}(\varrho)\,\d \varrho
			+ \int_{s}^{\half} f^*_{\pm}(\varrho) \weight(\varrho)\,\d \varrho
\end{align*}
for $s\in\left(0,\thalf\right]$, inequalities~\eqref{E:pointwise-1} and \eqref{E:pointwise-2} follow from \eqref{E:u-circ-oper-1} and \eqref{E:u-circ-oper-2}, respectively.

Let us now focus on inequality \eqref{E:gradient-rearrangement}.
Without loss of generality, we may assume that $\med(u)=0$.
The function
\begin{equation*}
[0,\infty) \ni	t\mapsto \int_{\{u_{+}\le t\}}|\nabla u_{+}|^{2}\,\dgn
\end{equation*}
is absolutely continuous, since by the coarea formula,
\begin{equation*}
	\int_{\{u_{+}\le t\}}|\nabla u_{+}|^{2}\, \dgn = \int_{\{0\le u\le t\}}|\nabla u_{+}|^{2}\, \dgn
		= \int_{0}^{t} \int_{\{u=\tau\}}|\nabla u|\,\dHg\,\d \tau.
\end{equation*}
Thus,
\begin{equation}\label{E:a-2}
	\int_{\{u_{+}\le t\}}|\nabla u_{+}|^{2}\,\dgn
		= \int_{0}^{t} \biggl(\frac{\d}{\d\tau}\int_{\{u_{+}\le\tau\}}|\nabla u_{+}|^{2}\,\dgn \biggr)\d \tau
	\quad\text{for $t\ge0$.}
\end{equation}
By inequalities~\eqref{E:diff-ineq} and~\eqref{E:HL-f}, with $u$ replaced with  $u_+$,
\begin{equation}\label{E:a-3}
	\frac{\d}{\d\tau}\int_{\{u_{+}\le\tau\}}|\nabla u_{+}|^{2}\,\dgn
		\le \int_{0}^{\mu(\tau)}f^{*}_{-}(s)\,\d s
	\quad\text{for \ae $\tau\ge0$},
\end{equation}
where $\mu(t)$ is as in~\eqref{E:mu-def}.
Note that here we have made use of the fact that the right-hand side of \eqref{E:mu-def} is not altered for $t>0$ if $u$ is replaced with $u_+$.
From inequalities~\eqref{E:a-3} and~\eqref{E:1frac} one infers that
\begin{equation}\label{E:a-4}
	\frac{\d}{\d\tau}\int_{\{u_{+}\le\tau\}}|\nabla u_{+}|^{2}\,\d x
		= - \frac{\d}{\d\tau}\int_{\{u_{+}>\tau\}}|\nabla u_{+}|^{2}\,\d x
		\le \frac{-\mu'(\tau)}{I(\mu(\tau))^{2}}\biggl(\int_{0}^{\mu(\tau)}f^{*}_{-}(s)\,\d s\biggr)^{2}
\end{equation}
for \ae $\tau\ge0$. Combining equations~\eqref{E:a-2} and~\eqref{E:a-4} tells us that
\begin{align} \label{E:a-5}
	\begin{split}
		\int_{\{u_{+}\le t\}}|\nabla u_{+}|^{2}\,\d x
		& \le \int_{0}^{t}\frac{-\mu'(\tau)}{I(\mu(\tau))^{2}}\biggl(\int_{0}^{\mu(\tau)}f^{*}_{-}(s)\,\d s\biggr)^{2}\,\d\tau
		   \le \int_{\mu(t)}^{\half}\frac{1}{I(\sigma)^{2}}\biggl(\int_{0}^{\sigma}f^{*}_{-}(s)\,\d s\biggr)^{2}\,\d \sigma
	\end{split}
\end{align}
for $t\ge0$. On the other hand, by the first inequality in~\eqref{HL},
\begin{align}\label{E:a-6}
	\begin{split}
		\int_{\{u_{+}\le t\}}|\nabla u_{+}|^{2}\,\d x
		& \ge \int_{\mu(t)}^{\half}|\nabla u_{+}|^{*}(s)^{2}\,\d s
			\ge \int_{\mu(t)}^{2\mu(t)}|\nabla u_{+}|^{*}(s)^{2}\,\d s\ge \mu(t)|\nabla u_{+}|^{*}\bigl(2\mu(t)\bigr)^{2}
	\end{split}
\end{align}
for $t\ge t_0$, where we used that $\gamma_n(\{u_+\le t\})=\half-\mu(t)$ and we chose $t_0$ such that $\mu(t_0)=\frac{1}{4}$.
Inequalities~\eqref{E:a-5} and~\eqref{E:a-6} yield
\begin{equation}\label{E:a-6'}
   \mu(t)|\nabla u_{+}|^{*}\bigl(2\mu(t)\bigr)^{2}
    \le \int_{\mu(\sigma)}^{\half}\frac{1}{I(\sigma)^{2}}\biggl(\int_{0}^{\sigma}f^{*}_{-}(s)\,\d s\biggr)^{2}\,\d \sigma,
\end{equation}
whence inequality~\eqref{E:gradient-rearrangement} follows for $u_{+}$.
The argument for $u_{-}$ is analogous.
\end{proof}

\begin{proof}[Proof of Theorem~\ref{T:existence}, Part 1]
\hypertarget{h:s1}Step 1.
Let $\{f_k\}$ be a sequence of functions from $L^{2}_\perp \RG$ such that $f_k\to f$ in $L^{1}\RG$.
We may also assume that
\begin{equation}\label{E:a-7}
    \|f_k\|_{L^{1}\RG} \le 2 \|f\|_{L^{1}\RG}.
\end{equation}
For each $k\in\N$, let $u_{k}\in W^{1,2}\RG$ be the unique solution to the equation
\begin{equation}\label{E:a-7'}
    \lgn u_k = f_k
\end{equation}
such that $\med(u_k)=0$.
Thus
\begin{equation}\label{E:a-8}
    \int_{\rn}\nabla u_k \cdot \nabla v\,\dgn = - \int_{\rn}f_k v\,\dgn
\end{equation}
for every $v \in W^{1,2}\RG$ and $k\in\N$.

\hypertarget{h:s2}Step 2.
We prove that there exists a function $u\in \MM\RG$ such that
\begin{equation}\label{E:a-9}
    u_k\to u
	\quad\text{\ae in $\rn$ (up to subsequences).}
\end{equation}
We shall do this by showing that $\{u_k\}$ is a Cauchy sequence in measure.
To this purpose, observe that, given $t,\tau>0$, one has
\begin{align}\label{E:a-10}
	\gamma_n\left(\left\{|u_k-u_m|>\tau\right\}\right)
		\le \gamma_n\left(\left\{|u_k|>t\right\}\right)
		+ \gamma_n\left(\left\{|u_m|>t\right\}\right)
		+ \gamma_n\left(\left\{|T_{t}(u_k)-T_{t}(u_m)|>\tau\right\}\right)
\end{align}
for every $k,m\in\N$.
Choosing the function $v=T_t(u_k)$ in equation~\eqref{E:a-8} enables us to deduce that
\begin{align}\label{E:a-11}
		\int_{\rn}|\nabla T_t(u_k)|^{2}\,\dgn
			= \int_{\rn}\nabla u_k\cdot \nabla T_t(u_k)\,\dgn
			 = -\int_{\rn}f_kT_t(u_k)\,\dgn
					\le 2t\|f\|_{L^{1}\RG}
\end{align}
for  $k\in\N$.
Set \begin{equation*}
    \mu^{\pm}_{k}(t) = \gamma_n\left(\{(u_k)_{\pm}>t\}\right)
	\quad\text{for $t \ge 0$.}
\end{equation*}
From inequality~\eqref{E:taufrac}, applied to $u_k$, we obtain that
\begin{align}\label{E:a-12}
		t  \le \int_{\mu^{\pm}_{k}(t)}^{\half}\frac{1}{I(s)^{2}}\int_{0}^{s}(f_k)^{*}_{\mp}(r)\,\d r\,\d s
			\le 2\|f\|_{L^{1}\RG} \int_{\mu^{\pm}_{k}(t)}^{\half}\frac{\d s}{I(s)^{2}}
			=  2\|f\|_{L^{1}\RG} \weight\bigl(\mu^{\pm}_{k}(t)\bigr)
\end{align}
for $t>0$.
Hence,
\begin{equation}\label{E:a-13}
  \gamma_n\left(\{|u_k|>t\}\right)
    =  \mu^{+}_{k}(t) + \mu^{-}_{k}(t)
		\le 2\weight^{-1}\biggl(\frac{t}{2\|f\|_{L^{1}\RG}}\biggr)
   \quad\text{for $t>0$.}
\end{equation}
Since $\lim _{t \to \infty}\weight^{-1}(t)=0$, given $\varepsilon >0$, we have that
\begin{equation}\label{E:a-14}
	\gamma_n\left(\{|u_k|>t\}\right)<\varepsilon
	\quad\text{and}\quad
	\gamma_n\left(\{|u_m|>t\}\right)<\varepsilon
   \quad\text{for $k,m\in\N$}
\end{equation}
if $t$ is sufficiently large.
Fix any such $t$.
Since $\med(u_k)=0$, we have that $\med(T_t(u_k))=0$ as well.
Hence, by the Sobolev inequality in the Gauss space, \cf\eg~\citep{Cia:09} and inequality~\eqref{E:a-11},
\begin{equation*}
	\|T_t(u_k)\|_{L^{2}\RG}
		\le c \|\nabla T_t(u_k)\|_{L^{2}\RG}
		\le c \sqrt{2t\|f\|_{L^{1}\RG}}
\end{equation*}
for some constant $c$.
Consequently, the sequence $\{T_t(u_k)\}$ is bounded in $W^{1,2}\RG$.
By the compactness of the embedding~\citep[Theorem~7.3]{Sla:15}
\begin{equation*}
    W^{1,2}\RG \to L^{2}\RG,
\end{equation*}
one has that $T_t(u_k)$ converges to some function $L^{2}\RG$ (up to subsequences).
In particular, $\{T_t(u_k)\}$ is a Cauchy sequence in measure.
Thus,
\begin{equation}\label{E:a-15}
	\gamma_n\left(\{|T_t(u_k)-T_t(u_m)|>\tau\}\right)
		< \varepsilon
\end{equation}
if $k,m$ are sufficiently large.
From~\eqref{E:a-10},~\eqref{E:a-14} and~\eqref{E:a-15}, we conclude that $\{u_k\}$ is a Cauchy sequence in measure.
Hence, equation~\eqref{E:a-9} follows.

\hypertarget{h:s3}Step 3.
We show that
\begin{equation}\label{E:a-16}
    \{\nabla u_k\} \quad\text{is a Cauchy sequence in measure.}
\end{equation}
To this purpose, note that, given $t,\tau,\delta>0$,
\begin{align}\label{E:a-17}
    \begin{split}
        \gamma_n & \left(\{|\nabla u_k - \nabla u_m|>t\}\right)
                \\
            & \le \gamma_n\left(\{|\nabla u_k|>\tau\}\right) + \gamma_n\left(\{|\nabla u_m|>\tau\}\right)
            + \gamma_n\left(\{|u_k-u_m|>\delta\}\right)
                \\
            &\quad + \gamma_n\left(\{|u_k-u_m|\le\delta,\,|\nabla u_k|\le\tau,\, |\nabla u_m|\le\tau,\, |\nabla u_k - \nabla u_m|>t\}\right).
    \end{split}
\end{align}
By estimate~\eqref{E:gradient-rearrangement}, on setting
\begin{equation*}
    \nu_{k}^{\pm}(t) = \gamma_n\left(\{|\nabla (u_k)_{\pm}|>t\}\right),
\end{equation*}
one infers that
\begin{align}\label{E:a-18}
    \begin{split}
        t & \le \left(\frac{2}{\nu_k^{\pm}(t)}\int_{\frac{\nu_{k}^{\pm}(t)}{2}}^{\half}
        \left(\int_{0}^{s}(f_k)^{*}_{\mp}(r)\,\d r\right)^{2}\frac{\d s}{I(s)^{2}}\right)^{\half}
        \le 2\|f\|_{L^{1}\RG}\left(\frac{2}{\nu_k^{\pm}(t)}\weight\left(\frac{\nu_k^{\pm}(t)}{2}\right)\right)^{\half}.
    \end{split}
\end{align}
Thus, if we define $\Psi(s)=\frac{s}{\weight(s)}$, an increasing function such that $\lim_{s\to 0^+}\Psi(s)=0$, we obtain from~\eqref{E:a-18}
\begin{equation}\label{E:a-19}
    \gamma_n\left(\{|\nabla u_k|>t\}\right)
        = \nu_{k}^{+}(t) + \nu_{k}^{-}(t) \le 2 \Psi^{-1}\left(\left(\frac{2\|f\|_{L^{1}\RG}}{t}\right)^{2}\right).
\end{equation}
Hence, we deduce that, given any $\varepsilon >0$,
\begin{equation}\label{E:a-20}
	\gamma_n\left(\{|\nabla u_k|>\tau\}\right)<\varepsilon
		\quad\text{and}\quad
	\gamma_n\left(\{|\nabla u_m|>\tau\}\right)<\varepsilon
	\quad\text{for $k,m\in\N$},
\end{equation}
provided that $\tau$ is sufficiently large.
Fix such a $\tau$.
Define the set
\begin{equation*}
    G=\bigl\{|u_k-u_m|\le\delta,\,
				|\nabla u_k|\le\tau,\,
				|\nabla u_m|\le\tau,\,
				|\nabla u_k-\nabla u_m|>\tau\bigr\}.
\end{equation*}
Then, by~\eqref{E:a-8} again, since $T_{\delta}(u_k-u_m)\in W^{1,2}\RG$,
\begin{align}\label{E:a-21}
	\begin{split}
		t^2 \int_{G}\dgn
			& \le \int_{G}|\nabla u_k-\nabla u_m|^{2}\,\dgn
				\le \int_{\{|u_k-u_m|\le\delta\}}|\nabla u_k-\nabla u_m|^{2}\,\dgn
					\\
			& = \int_{\rn}(\nabla u_k-\nabla u_m)\cdot \nabla \bigl(T_{\delta}(\nabla u_k-\nabla u_m)\bigr)\,\dgn
				= - \int_{\rn}(f_k-f_m)T_{\delta}(u_k-u_m)\,\dgn
                \\
            & \le 4\delta\|f\|_{L^{1}\RG}.
	 \end{split}
\end{align}
Choosing $\delta$ small enough, we get $\gamma_n(G)<\varepsilon$.
From Step~\hyperlink{h:s2}{2} we already know that $\{u_k\}$ is a Cauchy sequence in measure.
Thus,
\begin{equation}\label{E:a-22}
    \gamma_n\left(\{|u_k- u_m|>\delta\}\right)<\varepsilon
\end{equation}
if $k,m$ are large enough.
From~\eqref{E:a-17}, \eqref{E:a-20}, \eqref{E:a-21} and \eqref{E:a-22} we deduce that
\begin{equation*}
    \gamma_n\left(\{|\nabla u_k - \nabla u_m|>\delta\}\right)<\varepsilon
\end{equation*}
if $k,m$ are sufficiently large.
Property~\eqref{E:a-16} is thus established.

\hypertarget{h:s4}Step 4.
We prove that
\begin{equation}\label{E:a-23}
    u\in\TT^{1,2}\RG
\end{equation}
and
\begin{equation}\label{E:a-24}
	\nabla u_k\to\nabla u
		\quad\text{\ae in $\rn$}
\end{equation}
(up to subsequences), where $\nabla u$ denotes the ``surrogate'' gradient $Z_u$ in the sense of definition~\eqref{504}.

By property~\eqref{E:a-16}, there exists a measurable function $U\colon\rn\to\rn$ such that
\begin{equation}\label{E:a-25}
	\nabla u_k\to U
		\quad\text{\ae in $\rn$}
\end{equation}
(up to subsequences).
As shown in Step~\hyperlink{h:s2}{2}, the sequence $\{T_t(u_k)\}$ is bounded in $W^{1,2}\RG$.
Therefore, inasmuch as the latter space is reflexive, there exists a function $\hat u_t\in W^{1,2}\RG$ such that
\begin{equation}\label{E:a-26}
	T_t(u_k)\rightharpoonup \hat u_t
		\quad\text{weakly in $W^{1,2}\RG$,}
\end{equation}
and \ae in $\rn$ (up to subsequences).
Hence, by the uniqueness of the limit,
\begin{equation}\label{E:a-27}
	\hat u_t = T_t(u)
		\quad\text{\ae in $\rn$,}
\end{equation}
since $T_t(u_k)\to T_t(u)$ by~\eqref{E:a-9}.
Consequently, $T_t(u)\in W^{1,2}\RG$ and
\begin{equation}\label{E:a-28}
	T_t(u_k)\rightharpoonup T_t(u)
		\quad\text{weakly in $W^{1,2}\RG$.}
\end{equation}
Thanks to the arbitrariness of $t$, property~\eqref{E:a-23} holds, and
\begin{equation}\label{E:a-29}
	\nabla T_t(u) = \chi_{\{|u|<t\}}\nabla u
		\quad\text{\ae~in $\rn$}
\end{equation}
for $t>0$.
From equations~\eqref{E:a-9} and~\eqref{E:a-25} one also deduces that
\begin{equation}\label{E:a-30}
  \lim_{k\to\infty} \nabla T_t(u_k)
		= \lim_{k\to\infty} \chi_{\{|u_k|<t\}}\nabla u_k
		= \chi_{\{|u|<t\}}U
		\quad\text{\ae in $\rn$}
\end{equation}
for $t>0$, and, owing to~\eqref{E:a-28},
\begin{equation}\label{E:a-31}
  \nabla T_t(u) = \chi_{\{|u|<t\}}U
		\quad\text{\ae in $\rn$,}
\end{equation}
for $t>0$.
By~\eqref{E:a-31},
\begin{equation}\label{E:a-32}
  U=\nabla u.
\end{equation}
Property~\eqref{E:a-24} follows from equations~\eqref{E:a-25} and~\eqref{E:a-32}.

\hypertarget{h:s5}Step 5.
This step is devoted to the proof of the uniqueness of the solution $u$ (up to additive constants) to equation~\eqref{apr6}.
Assume that $u$ and $\hat u$ are solutions to~\eqref{apr6}.
Then there exist sequences $\{f_k\}$ and $\{\hat f_k\}$ in $L^2_\perp\RG$ such that  $f_k\to f$ and $\hat f_k\to\hat f$ in $L^{1}\RG$, the solutions $u_k$ to~\eqref{E:a-7'} satisfy $u_k\to u$ \ae in $\rn$, and the solutions $\hat u_k$ to problem~~\eqref{E:a-7'} with $f_k$ replaced by $\hat f_k$ satisfy $\hat u_k\to \hat u$ \ae in $\rn$.

Given any $t>0$, we make use of the test function $\phi=T_t(u_k-\hat u_k)$ in equation~\eqref{E:a-8}, and in a parallel equation with $u_k$ and $f_k$ replaced by $\hat u_k$ and $\hat f_k$.
On subtracting the equations so obtained one gets that
\begin{equation}\label{E:a-33}
	\int_{\{|u_k-\hat u_k|\le t\}} |\nabla u_k-\nabla\hat u_k|^{2}\,\dgn
    = - \int_{\rn}(f_k-\hat f_k)T_t(u_k-\hat u_k)\,\dgn
\end{equation}
for $k\in\N$.
The right-hand side of equation~\eqref{E:a-33} tends to $0$ as $k\to\infty$, since $|T_t(u_k-\hat u_k)|\le t$ and $f_k-\hat f_k\to 0$ in $L^{1}\RG$.
Moreover, the same arguments as in the proofs of Steps~\hyperlink{h:s3}{3} and \hyperlink{h:s4}{4} ensure that $\nabla u_k\to\nabla u$ and $\nabla \hat u_k\to\nabla\hat u$ \ae in $\rn$ (up to subsequences).
Therefore, on passing to the limit in equation~\eqref{E:a-33} as $k\to\infty$, we infer, via Fatou's lemma, that
\begin{equation*}
	\int_{\{|u-\hat u|\le t\}}|\nabla u-\nabla\hat u|^{2}\,\dgn=0,
\end{equation*}
whence, by the arbitrariness of $t$,
\begin{equation}\label{E:a-34}
	\nabla u = \nabla \hat u
		\quad\text{\ae in $\rn$.}
\end{equation}
Now, observe that, by Step~\hyperlink{h:s4}{4}, given $t,\tau>0$, we have
\begin{equation*}
	T_{\tau}(u-T_t(\hat u))\in W^{1,2}\RG.
\end{equation*}
An application of the Gaussian--Sobolev inequality to this function tells us that
\begin{align}\label{E:a-35}
	\begin{split}
		& \int_{\rn}
			\bigl|
				T_{\tau}(u-T_t(\hat u)) - \med\bigl(T_{\tau}(u-T_t(\hat u))\bigr)
			\bigr|^{2}\,\dgn
        \le c\int_{\rn}
					\bigl|
						\nabla T_{\tau}(u-T_t(\hat u))
					\bigr|^{2} \,\dgn
      \\
    &\quad  = c
			\biggl(
				\int_{\{t<|u|<t+\tau\}}|\nabla u|^{2}\,\dgn
				+ \int_{\{t-\tau<|u|<t\}}|\nabla u|^{2} \,\dgn
			\biggr).
	 \end{split}
\end{align}
Notice that we have also made use of equation~\eqref{E:a-34} in the last equality.
The choice of the test function $\phi=T_{\tau}(u_k-T_t(u_k))$ in equation~\eqref{E:a-8} enables us to obtain that
\begin{equation}\label{E:a-36}
	\int_{\{t<|u_k|<t+\tau\}}|\nabla u_k|^{2}\,\dgn
		\le \tau \int_{\{|u_k>t\}}|f_k|\,\dgn.
\end{equation}
Passing to the limit in~\eqref{E:a-36} and making use of Fatou's lemma tell us that
\begin{equation}\label{E:a-37}
	\int_{\{t<|u|<t+\tau\}}|\nabla u|^{2}\,\dgn
		\le \tau \int_{\{|u>t\}}|f|\,\dgn.
\end{equation}
Thus, the first integral on the rightmost side of inequality~\eqref{E:a-35} approaches $0$ as $t\to\infty$.
A similar argument employing the test functions $\phi=T_{\tau}(u_k-T_{t-\tau}(u_k))$ ensures that also the second integral tends to $0$ as $t\to\infty$.
On the other hand,
\begin{equation*}
	\lim_{t\to\infty}
		\bigl[
			T_{\tau}(u-T_t(\hat u)) - \med\bigl(T_{\tau}(u-T_t(\hat u))\bigr)
		\bigr]
		= T_{\tau}(u-\hat u)-\med\bigl(T_{\tau}(u-\hat u)\bigr)
		\quad\text{\ae in $\rn$.}
\end{equation*}
Therefore, passing to the limit in~\eqref{E:a-35} as $t\to\infty$ and making use of Fatou's lemma yield
\begin{equation*}
  \int_{\rn}
		\bigl|
			T_{\tau}(u-\hat u) - \med\bigl((T_{\tau}(u-\hat u)\bigr)
		\bigr|^{2}\,\dgn = 0
\end{equation*}
for $\tau>0$. Hence,
\begin{equation*}
	T_{\tau}(u-\hat u) - \med\bigl(T_{\tau}(u-\hat u)\bigr) = 0
\end{equation*}
\ae in $\rn$ for every $\tau>0$, and passing to the limit as $\tau\to\infty$, we get that
\begin{equation*}
	u-\hat u = \med(u-\hat u)
		\quad\text{\ae in $\rn$.}
\end{equation*}
Thus, the function $u-\hat u$ is constant on $\rn$.
\end{proof}

\begin{proof}[Proof of Theorem~\ref{T:pointwise-estimates}, completed]
Let $u\in D(\lgn)$.
By definition, there exists a sequence $\{u_k\}\subset\Wlgn L^2\RG$ fulfilling \eqref{apr1} and \eqref{apr2}.
Inequalities \eqref{E:pointwise-1} and \eqref{E:pointwise-2} have already been shown to hold with $u$ replaced by $u_k$.
Namely, on setting $f_k=\lgn u_k$, one has that
\begin{equation} \label{E:pointwise-1k}
	0 \le u_k^\circ(s) - u_k^\circ(\thalf)
		\le \weight(s) \int_{0}^{s} (f_k)^*_-(r)\,\d r
			+ \int_{s}^{\half}(f_k)^*_-(r)\,\weight(r)\,\d r
		\quad\text{for $s\in\left(0,\thalf\right]$}
\end{equation}
and
\begin{equation} \label{E:pointwise-2k}
	0 \le u_k^\circ(\thalf) - u_k^\circ(1-s)
		\le \weight(s) \int_{0}^{s} (f_k)^*_+(r)\,\d r
			+ \int_{s}^{\half}(f_k)^*_+(r)\,\weight(r)\,\d r
		\quad\text{for $s\in\left(0,\thalf\right]$}.
\end{equation}
We claim that $\med(u_k)\to\med(u)$.
By Lemma~\ref{L:med}, it suffices to verify that $u^\circ$ is continuous at~$\half$.
Theorem~\ref{T:existence} ensures that $u\in\TT^{1,2}\RG$ and hence $T_t(u)$ belongs to $W^{1,2}\RG$ and, in particular, $T_t(u)^\circ$ is continuous, see \citep{Ehr:84} or \citep[Lemma~3.3]{Cia:09}.
Therefore, as $T_t(u)^\circ=u^\circ$ on the set $\{|u^\circ|<t\}$ containing $\half$ for sufficiently large $t$, the function $u^\circ$ is also continuous at $\half$.

Property~\eqref{apr1} ensures that
\begin{equation} \label{apr7}
	\bigl(u_k(x) - \med(u_k)\bigr)_{\pm}
		\to \bigl(u(x)- \med(u)\bigr)_{\pm}
	\quad \text{for \ae $x\in\rn$}.
\end{equation}
Fatou's property of rearrangements tells us that
\begin{equation} \label{apr8}
	\liminf_{k\to \infty} \bigl( u_k - \med(u_k) \bigr)_{\pm}^*(s)
		\ge \bigl( u - \med(u) \bigr)_{\pm}^*(s)
	\quad\text{for  $s\in(0,1)$}.
\end{equation}
Moreover,
\begin{equation} \label{apr9}
	\bigl( u_k - \med(u_k) \bigr)_{+}^*(s)
		= u_k^\circ (s) - u_k^\circ(\thalf)
		\quad\text{and}\quad
	\bigl( u_k - \med(u_k) \bigr)_{-}^*(s)
		= u_k^\circ(\thalf) - u_k^\circ (1-s)
\end{equation}
for \ae  $s \in (0,\thalf)$.
Hence \eqref{apr8} and \eqref{apr9} yield
\begin{equation}\label{apr5}
	\liminf_{k\to \infty} u_k^\circ (s) - u_k^\circ(\thalf)
		\ge \bigl( u - \med(u)\bigr)_+^*(s)
	\quad\text{and}\quad
	\liminf_{k\to \infty} u_k^\circ(\thalf) - u_k^\circ (1-s)
		\ge \bigl( u - \med(u)\bigr)_-^*(s)
\end{equation}
for $s\in(0,\thalf)$.
On the other hand, owing to assumption \eqref{apr2}, one has that $(f_k)_\pm \to f_\pm$ in $L^1\RG$.
By \citep[Chapter~3, Theorem~7.4]{Ben:88}, The operation of decreasing rearrangement is a contraction in $L^1$, thus, one also has that $(f_k)_\pm^* \to f_\pm^*$ in $L^1(0,1)$.
As a consequence,
\begin{equation} \label{apr3}
	\lim_{k\to \infty}
			\int_{0}^{s} (f_k)^*_\pm(r)\,\d r
			= \int_{0}^{s} f^*_\pm (r)\,\d r.
\end{equation}
and
\begin{align} \label{apr4}
	\lim_{k\to \infty}
			\int_{s}^{\half}(f_k)^*_\pm (r)\,\weight(r) \d r
			= \int_{s}^{\half}f^*_\pm (r)\,\weight(r) \d r.
\end{align}
for every $s \in (0,\thalf]$.
Inequality~\eqref{E:u-rearrangement} then follows from~\eqref{E:pointwise-1k} and \eqref{E:pointwise-2k} via~\eqref{apr5}, \eqref{apr3} and \eqref{apr4}.

As for estimate \eqref{E:gradient-rearrangement}, it has already been shown to hold if $u\in \Wlgn L^2\RG$.
Assume now that $u\in D(\lgn)$ and set $f=\lgn u$.
Let $\{f_k\}$ is a sequence in $L^2\RG$ such that $\mv(f_k)=0$, $f_k\to f$ in $L^1\RG$ and \eqref{E:a-7} holds.
Let $u_k$ be the solution to equation \eqref{E:a-7'}.
We have that
\begin{equation}\label{20nov10}
	|\nabla (u_k-\med(u_k))_{\pm}|^*(s)
		\le
			\bigg(
				\frac 2s \int_{\frac s2}^{\half}
					\bigg(
						\int_0^r(f_k)^*_{\mp}(\varrho)\,\d \varrho
					\bigg)^2 \frac {\d r}{I(r)^2}
			\bigg)^{\half}
		\quad\text{for $s\in (0,\thalf)$}.
\end{equation}
By equation \eqref{apr3},
\begin{equation*}
	\int_{0}^{r} (f_k)^*_\mp(\varrho)\,\d \varrho
		\to \int_{0}^{r} f^*_\mp (\varrho)\,\d \varrho
	\quad\text{for $r\in (0,1)$}.
\end{equation*}
Moreover, owing to \eqref{E:a-7},
\begin{equation*}
	\int_{0}^{r} (f_k)^*_\pm(\varrho)\,\d \varrho
		\le 2 \|f\|_{L^1\RG}
	\quad\text{for $r\in (0,1)$,}
\end{equation*}
for $k\in\N$. Hence, since the function $I^{-2}$ is integrable in $(s,\half)$ for every $s\in(0,\half)$, by the dominated convergence theorem for Lebesgue integrals,
\begin{equation} \label{20nov11}
	\int_{\frac s2}^{\half}
		\bigg(
		\int_0^r (f_k)^*_{\mp}(\varrho)\,\d \varrho
		\bigg)^2 \frac {\d r}{I(r)^2}
	\to \int_{\frac s2}^{\half}
		\bigg(
			\int_0^r f^*_{\mp}(\varrho)\,\d \varrho
		\bigg)^2 \frac {\d r}{I(r)^2}
	\quad\text{for $s\in (0,\thalf)$}.
\end{equation}
Next, owing to equations \eqref{apr7} and \eqref{E:a-24}, one has that
\begin{equation} \label{20nov12}
	|\nabla  (u_k-\med(u_k))_{\pm}|
		\to | \nabla (u-\med(u))_{\pm}|
	\quad\text{\ae in $\rn$}.
\end{equation}
Hence, via Fatou's property of the decreasing rearrangement, we infer that
\begin{equation} \label{20nov13}
	\liminf_{k\to \infty}
		|\nabla  (u_k-\med(u_k))_{\pm}|^*(s)
			\ge | \nabla (u-\med(u))_{\pm}|^*(s)
	\quad\text{for $s\in (0,1)$.}
\end{equation}
Thanks to properties~\eqref{20nov11} and \eqref{20nov13}, inequality~\eqref{E:gradient-rearrangement} follows on passing to the limit in inequality~\eqref{20nov10} as $k\to\infty$.
\end{proof}

\begin{proof}[Proof of Theorem~\ref{T:existence}: inequalities \eqref{apr20} and \eqref{apr30}]
Since the function $\weight$ is decreasing, inequality~\eqref{E:u-rearrangement} yields
\begin{equation*}
	(u-\med(u))_\pm^*(s)
		\le \weight(s) \int_{0}^{s} f_{\mp}^*(r)\,\d r
			+ \weight(s) \int_{s}^{\half} f_{\mp}^*(r)\,\d r
		\le \weight(s) \|f\|_{L^1\RG}
\end{equation*}
for $s\in(0,\thalf)$. Therefore,
\begin{equation*}
	(u-\med(u))^*(s)
		\le (u-\med(u))_+^*(\tfrac{s}{2}) + (u-\med(u))_-^*(\tfrac{s}{2})
		\le 2\weight(\tfrac{s}{2}) \|f\|_{L^1\RG}
\end{equation*}
for $s\in(0,1)$. Hence,
\begin{equation*}
	\|u-\med(u)\|_{L^{1,\infty}\log L\RG}
		= \sup_{s\in(0,1)} s\ell(s)(u-\med(u))^*(s)
		\le \|f\|_{L^1\RG} \sup_{s\in(0,1)} 2s\ell(s)\weight(\tfrac{s}{2}),
\end{equation*}
where the last supremum is finite because of \eqref{E:weight-limit}.
This proves inequality \eqref{apr20}.

As far as inequality~\eqref{apr30} is concerned, estimate~\eqref{E:gradient-rearrangement} implies
\begin{align*}
	|\nabla (u-\med(u))_{\pm}|^*(s)
		& \le \|f\|_{L^{1}\RG}
        	\bigg(
						\frac 2s \int_{\frac s2}^{\half} \frac {\d r}{I(r)^2}
					\bigg)^{\half}
      = \|f\|_{L^{1}\RG} \sqrt{\tfrac{2}{s}\weight(\tfrac{s}{2})}
\end{align*}
for $s\in (0,\thalf)$.
Thus, by inequality \eqref{gen5},
\begin{align*}
	|\nabla u|^*(s)
		& = |\nabla (u-\med(u))|^*(s)
		= \big|\nabla \big((u-\med(u))_+ + (u-\med(u))_-\big)\big|^*(s)
		\\ & \le |\nabla(u-\med(u))_+|^*(\tfrac{s}{2})
			+ |\nabla(u-\med(u))_-|^*(\tfrac{s}{2})
\end{align*}
for $s\in(0,1)$, whence we infer that
\begin{equation*}
	\|\nabla u\|_{L^{1,\infty}(\log L)^\half\RG}
		= \sup_{s\in(0,1)} s\ell(s)^{\half} |\nabla u|^*(s)
		\le \|f\|_{L^{1}\RG} \sup_{s\in(0,1)}
					2\sqrt{2s\ell(s)\weight(\tfrac{s}{2})},
\end{equation*}
the last supremum being finite  by \eqref{E:weight-limit} again.

We conclude the proof of Theorem~\ref{T:existence} by showing that the norms on the left sides of~\eqref{apr20} and~\eqref{apr30} cannot be replaced by any stronger weak-type norm.

Fix $\delta>0$ and define the function $g_\delta\colon(0,\half)\to[0,\infty)$ by
\begin{equation*}
	g_\delta = \frac{1}{2\delta}\chi_{(0,\delta)}.
\end{equation*}
Also, define the functions $u_\delta\colon\rn\to\R$ and $f_\delta\colon\rn\to\R$ by
\begin{equation*}
	u_\delta(x)
		= \sgn(x_1)\int_{\phaxi}^{\half} \frac{1}{I(s)^2} \int_{0}^{s}g_\delta(r)\,\d r \,\d s
	\quad\text{for $x\in\rn$}
\end{equation*}
and
\begin{equation*}
	f_\delta(x) = -\sgn(x_1)g_\delta(\phaxi)
		\quad\text{for $x\in\rn$.}
\end{equation*}
One can verify that $\lgn u_\delta\in\Wlgn L^2\RG$ and $\lgn u_\delta = f_\delta$ -- see \eg the proof of Theorem~\ref{T:reduction-principle}. Moreover, by a change of variables,
\begin{align*}
	\|f_\delta\|_{L^1\RG}
		& = \int_{\{x_1>0\}} g_\delta(\Phi(x_1))\,\dgn
			+ \int_{\{x_1<0\}} g_\delta(\Phi(-x_1))\,\dgn
			\\
		& = 2 \int_{0}^{\infty} g_\delta(\Phi(x_1))\,\d\gamma_1
			= 2\int_{0}^{\half} g_\delta(s)\,\d s
			= 1,
\end{align*}
whence $f_\delta\in L^1_{\bot}\RG$, since, obviously, $\int_{\rn}f_\delta\,\dgn=0$.
Also $\med(u_\delta)=0$.

Set $E_\delta = \{x\in\rn: \Phi(|x_1|)\le \delta\}$. Then
\begin{equation*}
	|u_\delta(x)|
		= \int_{\phaxi}^{\half} \frac{1}{I(s)^2}\int_{0}^{s}g_\delta(r)\,\d r \,\d s
		\ge \int_{\delta}^{\half} \frac{\d s}{I(s)^2}\int_{0}^{\delta}g_\delta(r)\,\d r
		= \thalf \weight(\delta)
	\quad\text{for $x\in E_\delta$.}
\end{equation*}
Thus, $|u_\delta|\ge \thalf\weight(\delta)$ in $E_\delta$.
Since
\begin{equation*}
	\gamma_n(E_\delta)
		= 2\gamma_n\bigl(\{x\in\rn : x_1>\Phi^{-1}(\delta)\}\bigr)
		= 2\Phi\bigl(\Phi^{-1}(\delta)\bigr)
		= 2\delta,
\end{equation*}
one has that
\begin{equation*}
	u_\delta^* \ge \thalf\weight(\delta) \chi_{(0,2\delta)}.
\end{equation*}
Now, assume  that $\varphi$ is any quasiconcave function satisfying
\begin{equation} \label{jan1}
	\limsup_{s \to 0^+} \frac{\varphi(s)}{s\ell(s)} = \infty.
\end{equation}
Then
\begin{align*}
	\|u_\delta - \med(u_\delta)\|_{m_\varphi\RG}
		= \sup_{s\in(0,1)} u_\delta^{*}(s)\varphi (s)
		\ge \sup_{s\in(0,\delta)} u_\delta^{*}(s)s\ell(s) \frac{\varphi(s)}{s\ell(s)}
		\ge \thalf \weight(\delta)\delta\ell(\delta)
			\sup_{s\in(0,\delta)} \frac{\varphi (s) }{s\ell(s)}.
\end{align*}
In limit as $\delta\to0^+$, the latter term tends to infinity thanks to equation~\eqref{jan1}, since $\delta\ell(\delta)\weight(\delta)\to\thalf$, by equation~\eqref{E:weight-limit}.
The optimality of inequality~\eqref{apr20} is thus established.

Next, we have that
\begin{equation*}
	|\nabla u_\delta(x)|
		= \frac{1}{I\bigl(\phaxi\bigr)} \int_{0}^{\phaxi} \frac{1}{2\delta}\chi_{(0,\delta)}\,\d r
		\quad\text{for \ae $x\in\rn$}.
\end{equation*}
Hence
\begin{equation} \label{jan2}
	|\nabla u_\delta(x)|
		\ge \frac{1}{2\delta}\frac{\phaxi}{I\bigl(\phaxi\bigr)} \chi_{E_\delta}(x)
	\quad\text{for \ae $x\in\rn$}.
\end{equation}
To evaluate the measure of the level sets of the function on the right-hand side of equation~\eqref{jan2}, define the function $G(s)=s/I(s)$ for $s\in(0,\thalf)$.
Computations show that
\begin{align*}
	G'(\Phi(t))
		& = \frac {1}{I(\Phi (t))^2}\frac{1}{\sqrt{2\pi}}
				\bigg(
					e^{-\frac{t^2}2} - t \int_t^\infty e^{-\frac {\tau ^2}2}\,\d\tau
				\bigg)
			\\
		& > \frac {1}{I(\Phi (t))^2} \frac 1{\sqrt{2\pi}}
				\bigg(
					e^{-\frac {t^2}2} - \int_t^\infty \tau e^{-\frac {\tau ^2}2}\,\d\tau
				\bigg)
		= 0
	\quad\text{for $t>0$.}
\end{align*}
Thereby, the function $G$ is strictly increasing.
Since
\begin{align*}
  \gamma_n\left(\left\{x\in E_\delta : \frac{\phaxi}{2\delta I(\phaxi)}>t\right\}\right)
  	& = \gamma_n\left(\left\{x\in\rn : \Phi^{-1}(\delta) \le |x_1| \le \Phi^{-1}\bigl(G^{-1}(2\delta t)\bigr)\right\}\right)
			\\
		& = 2 \gamma_1\biggl(\Bigl(\Phi^{-1}(\delta), \Phi^{-1}\bigl(G^{-1}(2\delta t)\bigr)\Bigr)\biggr)
			= 2\bigl(\delta-G^{-1}(2\delta t)\bigr)
\end{align*}
for $t>0$, one infers that
\begin{align*}
	|\nabla u_\delta|^*(s)
    & \ge \inf\left\{t\in\R :  2\bigl(\delta-G^{-1}(2\delta t)\bigr) \le s\right\}
		= \frac{1}{2\delta} G(\delta - \tfrac{s}{2})
	\quad\text{for $s\in(0,2\delta)$}.
\end{align*}
Consequently, by the monotonicity of $G$,
\begin{align*}
	|\nabla u_\delta|^*
		\ge \frac{1}{2\delta} G\left(\tfrac{\delta}{2}\right) \chi_{(0,\delta)}
		= \frac{1}{4I\!\left(\frac\delta2\right)} \chi_{(0,\delta)}.
\end{align*}
Assume that a quasiconcave function $\varphi$ satisfies
\begin{equation*}
	\limsup_{s\to 0^+} \frac{\varphi(s)}{s\sqrt{\ell(s)}} = \infty.
\end{equation*}
Then, by equation~\eqref{E:I-limit},
\begin{align*}
	\|\nabla u_\delta\|_{m_\varphi\RG}
		& = \sup_{s\in(0,1)} |\nabla u_\delta|^{*}(s)\varphi (s)
    	\ge  \sup_{s\in\left(0,\tfrac\delta 2\right)}|\nabla u_\delta|^{*}(s)s\sqrt{\ell(s)} \frac{\varphi (s) }{s\sqrt{\ell(s)}}
        \\
    & \ge \frac{\frac\delta 2\sqrt{\ell(\frac\delta 2)}}{4I\!\left(\frac{\delta}{2}\right)} \sup_{s\in\left(0,\frac\delta 2\right)} \frac{\varphi (s) }{s\sqrt{\ell(s)}}
      \to \infty
			\quad\text{as $\delta\to0^+$.}
\end{align*}
This shows that the bound given by~\eqref{apr30} is the best possible.
\end{proof}

\begin{proof}[Proof of Proposition \ref{P:classical-sobolev}]
It suffices to show that the solution $u$ to the equation $\lgn u=f$, with $f\in L^{1,1;\half}_\perp\RG$, satisfies
\begin{equation} \label{E:W11-estimate}
	\int_{\rn} |\nabla u|\,\dgn
		\le c \int_{0}^{1} f^*(s)\sqrt{\ell(s)}\,\d s
\end{equation}
for some absolute constant $c$.

Let us first prove estimate~\eqref{E:W11-estimate} in the case when $u\in\Wlgn L^2\RG$. Assume, without loss of generality, that $\med(u)=0$.
As shown in the proof of Theorem~\ref{T:pointwise-estimates}, under these assumptions the function $u$ obeys the following inequality:
\begin{equation} \label{E:W11-estimate-step-1}
	\biggl( -\frac{\d}{\d t} \int_{\{u>t\}} |\nabla u|\,\dgn \biggr)^2
		\le -\mu'(t) \int_{0}^{\mu(t)} f_-^* (s)\,\d s
	\quad\text{for \ae $t>0$},
\end{equation}
where $\mu(t)=\gamma_n(\{x\in\rn:u(x)>t\})$ -- see estimate~\eqref{E:diff-CS}.
Next, by inequality \eqref{E:d-nabla-u-I},
\begin{equation}
	- \frac{\d}{\d t} \int_{\{u>t\}} |\nabla u|\,\dgn
		\ge I\bigl(\mu(t)\bigr)
	\quad\text{for \ae $t>0$}.
\end{equation}
Coupling this inequality with estimate~\eqref{E:W11-estimate-step-1} tells us that
\begin{equation}
	- \frac{\d}{\d t} \int_{\{u>t\}} |\nabla u|\,\dgn
		\le - \frac{\mu'(t)}{I\bigl(\mu(t)\bigr)}
			\int_{0}^{\mu(t)} f_-^* (s)\,\d s
	\quad\text{for \ae $t>0$}.
\end{equation}
An integration over $(0,\infty)$ yields
\begin{align*}
	\int_{\rn} |\nabla u_+|\,\dgn
		& = \int_{\{u>t\}} |\nabla u|\,\dgn
			= \int_0^\infty \biggl( - \frac{\d}{\d t} \int_{\{u>t\}} |\nabla u|\,\dgn \biggr)\d t
			\\
		& \le \int_0^\infty \biggl( - \frac{\mu'(t)}{I\bigl(\mu(t)\bigr)} \int_{0}^{\mu(t)} f_-^* (s)\,\d s\biggr)\d t
			\le \int_0^\half \frac{1}{I(r)} \int_0^r f_-^*(s)\,\d s\,\d r
			\\
		& = \int_0^\half f_-^*(s) \int_s^\half \frac{\d r}{I(r)}\,\d s
			\le  c \int_0^\half f_-^*(s) \sqrt{\ell(s)}\,\d s,
\end{align*}
for some absolute constant $c$.
Notice that here we used the asymptotic behavior of $I$ from~\eqref{E:I-limit}.
Since a parallel estimate holds for $u_-$ in terms of $f_+$, inequality~\eqref{E:W11-estimate} follows.

Now, assume that $f\in L^{1,1;\half}_\perp\RG$ and let $\{f_k\}\subset L^2\RG$ be a sequence such that $f_k\to f$ in $L^{1,1;\half}\RG$ and $\mv(f_k)=0$ for every $k\in\N$.
For instance, we may take $f_k=T_kf-\mv(T_kf)$.
Then, by the uniqueness of the solution to equation \eqref{apr6}, the sequence of functions $\{u_k\}\subset\Wlgn L^2\RG$ such that $\lgn u_k=f_k$ satisfies $u_k\to u$ \ae.
Finally, from Fatou's lemma and estimate~\eqref{E:W11-estimate} applied to  $u_k$ we deduce that
\begin{equation*}
	\|\nabla u\|_{L^1\RG}
		\le \liminf_{k\to\infty} \|\nabla u_k\|_{L^1\RG}
		\le c \lim_{k\to\infty} \|f_k\|_{L^{1,1;\half}\RG}
		= c\|f\|_{L^{1,1;\half}\RG}
\end{equation*}
for some absolute constant $c$.
Inequality~\eqref{E:W11-estimate} is thus established.
\end{proof}

\section{Reduction principle, and optimal target and domain spaces in Ornstein-Uhlenbeck embeddings}\label{main}

The main result of this section is stated in Theorem~\ref{T:reduction-principle}, which asserts that the validity of a Sobolev-type inequality for the Ornstein-Uhlenbeck operator is equivalent to the boundedness of the one-dimensional operator $\sou$ defined as
\begin{equation} \label{E:S-def}
	\sou g(s)
		= \frac{1}{s\ell(s)} \int_{0}^{s} g(r)\,\d r
			+ \int_{s}^{1} \frac{g(r)}{r\ell(r)} \,\d r
    \quad\text{for $s\in(0,1)$,}
\end{equation}
for $g\in\MM_+(0,1)$.
Note that
\begin{equation} \label{E:S-selfadojoint}
	\int_0^{1} \sou g(s) h(s) \,\d s
		= \int_0^{1}g(s) \sou h(s) \,\d s
		\quad\text{for every $g,h\in\MM_+(0,1)$.}
\end{equation}
Therefore $\sou\colon X(0,1)\to Y(0,1)$ if and only if $\sou\colon Y'(0,1)\to X'(0,1)$ for any pair of rearrange\-ment-invariant spaces.

In what follows, we shall write $A \lesssim B$ if there exists a positive constant $c$ independent of appropriate quantities involved in both $A$ and $B$ and such that $A\le cB$.
The symbol $A\gtrsim B$ is then defined in the obvious way.
If both $A\lesssim B$ and $A\gtrsim B$ hold, then we write $A\approx B$.

\begin{theorem}[Reduction principle for Ornstein-Uhlenbeck embeddings] \label{T:reduction-principle}
Let $X\RG $ and $Y\RG$ be rearrangement-invariant spaces.
The following statements are equivalent:
\begin{enumerate}
\item \label{en:full-inequality}
	There exists a constant $c_1>0$ such that
	\begin{equation*}
		\|u - \med(u)\|_{Y\RG}
			\le c_1 \|\lgn u\|_{X\RG}
	\end{equation*}
	for every $u\in\Wlgn X\RG$.
\item \label{en:reduced-inequality}
	There exists a constant $c_2>0$ such that
	\begin{equation} \label{E:reduced-inequality}
		\biggl\|
			\frac{1}{s\ell(s)} \int_{0}^{s} g(r)\,\d r
			+ \int_{s}^{1} \frac{g(r)}{r\ell(r)}\d r
		\biggr\|_{Y(0,1)}
			\le c_2 \|g\|_{X(0,1)}
	\end{equation}
	for every nonnegative function $g\in X(0,1)$.
\end{enumerate}
Moreover, the constants $c_1$ and $c_2$ depend only on each other.
\end{theorem}

Theorem~\ref{T:reduction-principle} enables us to characterize the optimal rearrangement-invariant spaces $X\RG$ and $Y\RG$ in Sobolev-type inequalities of the form
\begin{equation}\label{E:inequality-optimal}
	\|u-\med(u)\|_{Y\RG} \le C\|\lgn u\|_{X\RG}
\end{equation}
for every $u\in\Wlgn X\RG$.
With slight abuse of notation, inequality~\eqref{E:inequality-optimal} will often be written in embedding form as
\begin{equation}\label{E:boundedness-embedding}
	\Wlgn X\RG\to Y\RG.
\end{equation}

Let us begin with the identification of the optimal target space associated with a given domain.
This requires the following preliminary result.

\begin{lemma}\label{optimalfn}
Assume that $\|\cdot\|_{X(0,1)}$ is a rearrangement-invariant function norm such that
\begin{equation} \label{E:target-condition}
  \ell\ell \in X'(0,1).
\end{equation}
Then the functional given by
\begin{equation}\label{E:assoc-of-optimal-target-norm}
	\left\|\sou g^{*}\right\|_{X'(0,1)}
\end{equation}
for $g\in\MM_+(0,1)$ is a rearrangement-invariant function norm.
\end{lemma}

Denote by $\|\cdot\|_{\xs(0,1)}$ the rearrangement-invariant function norm whose associate norm is given by
\begin{equation}\label{july1}
	\|g\|_{\xs'(0,1)} = \left\|\sou g^{*}\right\|_{X'(0,1)}
\end{equation}
for $g\in\MM_+(0,1)$.

\begin{theorem}[Optimal target for Ornstein-Uhlenbeck embeddings] \label{T:OU-target}
Let $X\RG$ be a rearrangement-invariant space satisfying condition~\eqref{E:target-condition} and let $\xs\RG$ be the rearrangement-invariant space defined via equation \eqref{july1}.
Then
\begin{equation}\label{E:boundedness-embedding-optimal}
	\Wlgn X\RG\to \xs\RG.
\end{equation}
Moreover, $\xs\RG$ is the optimal (smallest) rearrangement-invariant space for which embedding~\eqref{E:boundedness-embedding-optimal} holds.

If condition~\eqref{E:target-condition} is not satisfied, then embedding~\eqref{E:boundedness-embedding} fails for every rearrangement-invariant space $Y\RG$.
\end{theorem}

A characterization of the optimal domain space in inequality~\eqref{E:inequality-optimal} is the subject of Theorem~\ref{C:OU-domain} below and requires the next lemma.

\begin{lemma}\label{optimaldm}
Assume that $\|\cdot\|_{Y(0,1)}$ is a rearrangement-invariant function norm such that
\begin{equation} \label{E:domain-condition}
 \ell\ell\in Y(0,1).
\end{equation}
Then the functional $\|\cdot\|_{\ys(0,1)}$ given by
\begin{equation}\label{july5}
	\|g\|_{\ys(0,1)}=\left\|\sou g^{*}\right\|_{Y(0,1)}
\end{equation}
for $g\in\MM_+(0,1)$ is a rearrangement-invariant function norm.
\end{lemma}

\begin{theorem}[Optimal domain for Ornstein-Uhlenbeck embeddings]\label{C:OU-domain}
Let $Y\RG$ be a rearrangement-invariant space satisfying condition \eqref{E:domain-condition} and let $ \ys\RG$ be the rearrangement-invariant space defined via equation \eqref{july5}.
Then
\begin{equation}\label{E:boundedness-embedding-optimal-domain}
	\Wlgn \ys\RG\to Y\RG.
\end{equation}
Moreover, $\ys\RG$ is the optimal (largest) rearrangement-invariant space for which embedding~\eqref{E:boundedness-embedding-optimal-domain} holds.

If condition~\eqref{E:domain-condition} is not satisfied, then embedding~\eqref{E:boundedness-embedding} fails for every rearrangement-invariant space $X\RG$.
\end{theorem}

The remaining part of this section is devoted to the proofs of the results stated above.

\begin{proof}[Proof of Theorem \ref{T:reduction-principle}]
First note that \ref{en:reduced-inequality} is equivalent to the following condition:
\begin{enumerate}[label={\rm(\roman*)'}]
\setcounter{enumi}{1}
\item \label{en:reduced-inequality-half}
	There exists a constant $c_2'>0$ such that
	\begin{equation*}
		\biggl\|
			\weight(s) \int_{0}^{s} g(r)\,\d r
			+ \int_{s}^{\half} g(r)\weight(r)\,\d r
		\biggr\|_{Y(0,\half)}
			\le c_2' \|g\|_{X(0,\half)}
	\end{equation*}
	for every nonnegative $g\in X(0,\half)$.
\end{enumerate}
Furthermore, $c_2'$ and $c_2$ depend only on each other.
This follows by a standard argument in rearrangement-invariant spaces involving rearrangements and the dilatation operator together with the fact that, by equation \eqref{E:weight-limit}, the function $\weight(s)$ is equivalent to $1/s\ell(s)$ near zero.

Let us show that~\ref{en:reduced-inequality-half} implies~\ref{en:full-inequality}.
Assume that $u\in\Wlgn X\RG$.
By Theorem~\ref{T:pointwise-estimates},
\begin{align*}
	\|(u - \med(u))_+\|_{Y\RG}
		& = \|(u - \med(u))_+^*\|_{Y(0,\half)}
			\\
		& \le \biggl\|
						\weight(s) \int_{0}^{s} (\lgn u)_-^*(r)\,\d r
							+ \int_{s}^{\half} (\lgn u)_-^*(r)\weight(r)\,\d r
					\biggr\|_{Y(0,\half)}
			\\
		& \le c_2' \|(\lgn u)_-^*\|_{X(0,\half)}
			\le c_2' \|\lgn u\|_{X\RG}
\end{align*}
and, analogously, $\|(u - \med(u))_-\|_{Y\RG}\le c_2' \|\lgn u\|_{X\RG}$.
Thus,
\begin{equation*}
	\|u - \med(u)\|_{Y\RG}
		\le \|(u - \med(u))_+\|_{Y\RG} + \|(u - \med(u))_-\|_{Y\RG}
		\le 2c_2' \|\lgn u\|_{X\RG},
\end{equation*}
whence inequality \ref{en:full-inequality} follows.

Conversely, assume that inequality \ref{en:full-inequality} holds and let $g$ be a nonnegative function in $X(0,\half)\cap L^2(0,\half)$.
Define the function $u\colon\rn\to\R$ as
\begin{equation} \label{E:u-def}
	u(x) = \sgn(x_1) \int_{\phaxi}^{\frac 12} \frac{1}{I(s)^2} \int_{0}^{s} g(r)\,\d r\d s
		\quad\text{for $x\in\rn$}.
\end{equation}
Then $u$ is weakly differentiable and, thanks to equation~\eqref{derivative},
\begin{equation*}
	\frac{\partial u}{\partial x_1}(x)
		= \frac{1}{I\bigl(\phaxi\bigr)} \int_{0}^{\phaxi} g(r)\,\d r
		\quad\text{for \ae $x\in\rn$}
\end{equation*}
and $\displaystyle \frac{\partial u}{\partial x_j}=0$ for $j=2,\ldots,n$.
Consequently,
\begin{equation}\label{E:gradient}
	|\nabla u(x)| = \frac{1}{I\bigl(\phaxi\bigr)} \int_{0}^{\phaxi} g(r)\,\d r
		\quad\text{for \ae $x\in\rn$}
\end{equation}
and
\begin{align*}
	\|\nabla u\|_{L^2\RG}
		&	= \biggl\| \frac{1}{I\bigl(\phaxi\bigr)} \int_{0}^{\phaxi} g(s)\,\d s \biggr\|_{L^2\RG}
			= 2 \biggl\| \frac{1}{I(s)} \int_{0}^{s} g(r)\,\d r \biggr\|_{L^2(0,\half)}
			\\
		& \le c \biggl\| \frac{1}{s\sqrt{\ell(s)}} \int_{0}^{s} g(r)\,\d r \biggr\|_{L^2(0,\half)}
			\le c \biggl\| \frac{1}{\sqrt{\ell(s)}} g(s) \biggr\|_{L^2(0,\half)}
			\le c \|g\|_{L^2(0,\half)}
			< \infty
\end{align*}
for some absolute constant $c$.
Hence, $u\in W^{1,2}\RG$.
Here, we used the asymptotic behaviour of $I$ from \eqref{E:I-limit} and the Hardy inequality \citep[Theorem~1.3.2.2]{Maz:11}.
Therefore $u$ satisfies equation~\eqref{E:weak-formulation} with
\begin{equation} \label{E:f-def}
	f(x) = - \sgn(x_1) g\bigl(\phaxi\bigr)
		\quad\text{for $x\in\rn$}
\end{equation}
for any $v\in W^{1,2}\RG$, as an integration by parts shows.
Thus $\lgn u\in\Wlgn L^2\RG$ and $\lgn u = f$.
We have that
\begin{equation} \label{E:lgnu-le-g}
	\|\lgn u\|_{X\RG}
		\le \|f_+^*\|_{X(0,\half)} + \|f_-^*\|_{X(0,\half)}
		= 2 \|g\|_{X(0,\half)}.
\end{equation}
Moreover, inasmuch  as $\med(u)=0$,
\begin{align} \label{E:u-ge-Sg}
	\begin{split}
	& \|u\|_{Y\RG}
			\ge \|u_+\|_{Y\RG}
			= \biggl\|
					\chi_{\{x_1>0\}} \int_{\Phi(x_1)}^\half
						\frac{1}{I(s)^2} \int_{0}^{s} g(r)\,\d r\d s
				\biggr\|_{Y\RG}
				\\
	&\quad  = \biggl\|
					\int_{s}^\half
						\frac{1}{I(r)^2} \int_{0}^{r} g(\varrho)\,\d \varrho\d r
				\biggr\|_{Y(0,\half)}
			= \biggl\|
					\weight(s) \int_{0}^{s} g(r)\,\d r
						+ \int_{s}^\half g(r)\weight(r)\,\d r
				\biggr\|_{Y(0,\half)}.
	\end{split}
\end{align}
Thus inequality \ref{en:reduced-inequality-half} follows, with $c_2'=2c_1$, via equations \ref{en:full-inequality}, \eqref{E:lgnu-le-g} and \eqref{E:u-ge-Sg}.
Next, assume that the nonnegative function $g$ belongs to $X(0,\half)$.
Since, in particular, $g\in L^1(0,\half)$, there exists a sequence of nonnegative functions $g_k\in L^2(0,\half)$ such that $g_k\uparrow g$ in $L^1(0,\half)$.
If we define $u_k\colon\rn\to\R$ as in \eqref{E:u-def}, with $g$ replaced by $g_k$, then, as shown above, $u_k\in\Wlgn L^2\RG$ and $\lgn u_k=f_k$, where $f_k(x)=-\sgn(x_1)g_k(\phaxi)$ for $x\in\rn$.
Hence, $\lim_{k\to\infty}u_k(x)$ exists for \ae $x\in\rn$ and the limiting function, $u(x)$ say, obeys the representation formula \eqref{E:u-def}.
Also $\med(u)=0$, and $f_k\to f$ in $L^1\RG$, where $f$ is as in \eqref{E:f-def}.
Therefore, $u\in D(\lgn)$ and $\lgn u=f$.
Inequality \ref{en:reduced-inequality-half} then follows again by \eqref{E:lgnu-le-g} and \eqref{E:u-ge-Sg}.
\end{proof}

\begin{proof}[Proof of Lemma \ref{optimalfn}]
The rearrangement invariance of the functional $\|\cdot\|_{\xsd(0,1)}$ is obvious.
Properties \ref{en:p2},
\ref{en:p3} are readily verified.
The properties formulated in \ref{en:p1} are also clearly fulfilled, except the triangle inequality.
The latter can be shown as follows.
Since the function $r\mapsto 1/r\ell(r)$ is decreasing on $(0,1)$, for each fixed $s\in(0,1)$ the function
\begin{equation*}
    r\mapsto \min\left\{\tfrac{1}{s\ell(s)},\tfrac{1}{r\ell(r)}\right\}
\end{equation*}
is also decreasing on $(0,1)$.
Moreover,
\begin{equation}\label{E:kernel}
	\sou g(s)
		= \int_0^{1} g(r)\min\left\{\tfrac{1}{s\ell(s)},\tfrac{1}{r\ell(r)}\right\}\d r
		\quad\text{for $s\in(0,1)$.}
\end{equation}
Thus, by Hardy's lemma \citep[Section~2, Proposition~3.6]{Ben:88}, the operator $g\mapsto\sou g^{*}$ is subadditive on $\MM_0(0,1)$.
This implies the triangle inequality for $\|\cdot\|_{\xsd(0,1)}$.
Next, observe that
\begin{equation*}
	\|\chi_{(0,1)}\|_{\xsd(0,1)}
		= \left\|\sou \chi_{(0,1)} \right\|_{X'(0,1)}
		\approx \left\|\frac{1}{\ell(s)}+\ell\ell(s)\right\|_{X'(0,1)}.
\end{equation*}
It thus follows from~\eqref{E:target-condition} that $\|\chi_{(0,1)}\|_{\xsd(0,1)}<\infty$.
Property \ref{en:p4} is thus also proved.
Finally, for every $g\in\MM_+(0,1)$, one has that
\begin{align*}
	\|g\|_{\xsd(0,1)}
		& = \left\|\sou g^{*} \right\|_{X'(0,1)}
			\ge \left\|\chi_{(\half,1)}\sou g^{*} \right\|_{X'(0,1)}
   	 	\ge \left\|
						\frac{\chi_{(\half,1)}(s)}{s\ell(s)}
							\int_{0}^{\half} g^{*}(r)\,\d r
					\right\|_{X'(0,1)}
        \\
    & = \int_{0}^{\half} g^{*}(r)\,\d r
					\left\| \frac{\chi_{(\half,1)}(s)}{s\ell(s)} \right\|_{X'(0,1)}
       \ge \int_{0}^{\half} g^{*}(r)\,\d r \,
							\|\chi_{(\half,1)}\|_{X'(0,1)},
\end{align*}
where in the last inequality we used that $1/s\ell(s)$ decreases to $1$ on $(0,1)$.
Since $X'$ is a rearrangement-invariant space, $\|\chi_{(0,\half)}\|_{X'(0,1)}<\infty$.
Consequently,
\begin{equation*}
	\int_0^{1}g(t)\,\d t
		= \int_0^{1}g^*(s)\,\d s
		\le 2 \int_0^{\half}g^*(s)\,\d s
		\le 	\frac 2{\|\chi_{(0,\half)}\|_{X'(0,1)}}\|g\|_{\xsd(0,1)}.
\end{equation*}
This establishes property~\ref{en:p5}.
The proof is complete.
\end{proof}

\begin{proof}[Proof of Theorem \ref{T:OU-target}]
We begin by showing that
\begin{equation}\label{july2}
	\sou\colon X(0,1)\to \xs(0,1).
\end{equation}
Hence, embedding~\eqref{E:boundedness-embedding-optimal} will follow, owing to Theorem~\ref{T:reduction-principle}.
By equation~\eqref{E:kernel} and the monotonicity of the kernel of the operator $\sou$, one obtains, via the Hardy--Littlewood inequality, that $\sou g\le\sou g^{*}$ for every $g\in\MM_0(0,1)$.
Therefore,
\begin{equation} \label{E:S-HL}
	\|\sou g\|_{X'(0,1)}
		\le \|\sou g^*\|_{X'(0,1)}.
\end{equation}
This shows that $\|\sou g\|_{X'(0,1)}\le\|g\|_{\xsd(0,1)}$, namely $\sou\colon \xsd(0,1)\to X'(0,1)$.
By property~\eqref{E:S-selfadojoint},
we hence deduce that $\sou\colon X(0,1) \to \xs(0,1)$.

Assume now that $\|\cdot\|_{Z(0,1)}$ is a rearrangement-invariant function norm such that
\begin{equation*}
	\|u-\med(u)\|_{Z\RG} \le c\|\lgn u\|_{X\RG}
\end{equation*}
for some constant $c$ and for every $u\in\Wlgn X\RG$.
Then, by Theorem~\ref{T:reduction-principle}, one has that $\sou\colon X(0,1)\to Z(0,1)$.
Hence, $\sou\colon Z'(0,1) \to X'(0,1)$ as well.
In particular,
\begin{equation*}
	\|g\|_{\xsd(0,1)}
		= \|\sou g^*\|_{X'(0,1)}
		\le C \|g^*\|_{Z'(0,1)}
		= C \|g\|_{Z'(0,1)}
\end{equation*}
for every $g\in\MM_+(0,1)$.
This shows $Z'(0,1)\to \xsd(0,1)$, which in turn implies $\xs(0,1)\to Z(0,1)$.
Therefore, $\xs\RG\to Z\RG$.
The optimality of the space $\xs\RG$ in embedding~\eqref{E:boundedness-embedding-optimal} is thus established.

Finally assume that condition~\eqref{E:target-condition} is not satisfied and yet embedding~\eqref{E:boundedness-embedding} holds for some rearrangement-invariant space $Y\RG$.
Then, by Theorem~\ref{T:reduction-principle}, $\sou\colon X(0,1)\to Y(0,1)$.
This in turn implies that $\sou\colon Y'(0,1)\to X'(0,1)$.
Hence, there exists a constant $c$ such that $\|\sou g^*\|_{X'(0,1)} \le c \|g^*\|_{Y'(0,1)}$ for every $g\in\MM_+(0,1)$.
Applying this inequality to the function $g=\chi_{(0,1)}$, and using the fact that $Y'(0,1)$ is a rearrangement-invariant space an hence satisfies~\ref{en:p4}, tell us that
\begin{equation*}
	\infty
		> c\|\chi_{(0,1)}\|_{Y'(0,1)}
		\ge \|\sou\chi_{(0,1)}\|_{X'(0,1)}
		\approx \bigl\|\tfrac{1}{\ell(s)} + \ell\ell(s)\bigr\|_{X'(0,1)}
		\ge \|\ell\ell(s)\|_{X'(0,1)}
		= \infty.
\end{equation*}
This contradiction shows that no space $Y(0,1)$ enjoying this property can exist.
\end{proof}

\begin{proof}[Proof of Lemma~\ref{optimaldm}, sketched]
The fact that $\|\cdot\|_{\ys(0,1)}$ is a rearrangement invariant function norm can be deduced from Lemma~\ref{optimalfn} and a close inspection of its proof.
\end{proof}

\begin{proof}[Proof of  Theorem~\ref{C:OU-domain}, sketched]
It follows from Theorem~\ref{T:OU-target} and its proof that $\sou\colon \ys(0,1)\to Y(0,1)$ and that $\ys(0,1)$ is the largest space enjoying this property.
Embedding~\eqref{E:boundedness-embedding-optimal-domain} and the optimality of its domain space thus follow from Theorem~\ref{T:reduction-principle}.

Now assume that condition~\eqref{E:domain-condition} fails and that embedding~\eqref{E:boundedness-embedding} holds for some rearrangement-invariant function norm $\|\cdot\|_{X(0,1)}$.
Then, from Theorem~\ref{T:reduction-principle} again we infer that $\sou\colon X(0,1)\to Y(0,1)$, and we arrive at an analogous contradiction as in the proof of Theorem~\ref{T:OU-target}.
\end{proof}

\section{Ornstein-Uhlenbeck embeddings in Orlicz spaces}\label{orlicz}

This section is devoted to a description of Sobolev type inequalities for the Ornstein-Uhlenbeck operator with optimal target and domain in the class of Orlicz spaces.
Before presenting our general results, we collect  a few embeddings in the following example.
They follow as special cases of the general results  and  are yet sufficient to exhibit some peculiar traits of the inequalities in question.

\begin{example}
The following embeddings hold:
\begin{equation} \label{E:Orlicz-examples}
	\left\{
	\renewcommand{\arraystretch}{1.2}
	\begin{array}{@{}l@{{}\to{}}l@{\quad}l}
		\Wlgn L(\log\log L)^{\alpha+1}
			& L(\log\log L)^{\alpha}
				& \text{if $\alpha\ge 0$}
					\\
		\Wlgn L^{1}(\log L)^{\alpha}
			& L^{1}(\log L)^{\alpha}
				& \text{if $\alpha\ge 0$}
					\\
		\Wlgn L^p(\log L)^\alpha
			& L^p(\log L)^{\alpha+p}
				& \text{if $p\in(1,\infty)$ and $\alpha\in\R$}
					\\
		\Wlgn\exp L^{\beta}
			& \exp L^{\beta}
				& \text{if $\beta>0$}
					\\
		\Wlgn\exp\exp L^{\beta+1}
			& \exp \exp L^{\beta}
		 		& \text{if $\beta>0$}
					\\
		\Wlgn L^{\infty}
			& \exp \exp L,
	\end{array}
	\right.
\end{equation}
where all the spaces are over $\RG$.
\end{example}

The distinguishing features of the embeddings in \eqref{E:Orlicz-examples} can be summarized as follows.

First, all the domain spaces and the target spaces in equation~\eqref{E:Orlicz-examples} are not only optimal within the class of Orlicz spaces, but also among all rearrangement-invariant spaces.
This property is in sharp contrast with Sobolev embeddings in Euclidean domains, including those for the Laplace operator, where the optimal target and domain rearrangement-invariant space is always better (namely, essentially smaller on the target side and essentially larger one the domain side) than that in the smaller class of Orlicz spaces.

Next, observe that the norm in the target space can either be stronger, equivalent or weaker than that in the domain space.
This means that there can be a gain, or a draw or a loss in the degree of integrability of a function inherited from that of the Ornstein-Uhlenbeck operator.
On the contrary, a function vanishing on the boundary of a domain with finite Lebesgue measure  always enjoys stronger integrability properties than its Laplacian.

Finally, the embeddings above are worth being compared with standard second-order Gaussian Sobolev embeddings.
In particular, one has that
\begin{equation} \label{E:Orlicz-2nd}
	\left\{
	\renewcommand{\arraystretch}{1.2}
		\begin{array}{@{}l@{{}\to{}}l@{\quad}l}
		W^2 L^p (\log L)^\alpha
			& L^p (\log L)^{\alpha +1}
				& \text{if $p\in[1,\infty)$ and $\alpha\ge 0$}
					\\
		W^2\exp L^{\beta}
			& \exp L^{\frac{\beta}{\beta+1}}
		 		& \text{if $\beta>0$}
					\\
		W^2 L^{\infty}
			& \exp L,
	\end{array}
	\right.
\end{equation}
where all the spaces are over $\RG$, see~\citep[Theorems~7.8 and 7.13, Corollary 7.14]{Cia:15}.
The norms of the target spaces in the embeddings displayed in \eqref{E:Orlicz-2nd} are always weaker than those in the respective embeddings with the same domain norms in \eqref{E:Orlicz-examples}, save when $p=1$ in the last one of \eqref{E:Orlicz-2nd}, in which case the norm in the target space is stronger.
This phenomenon can be explained by the fact that, because of a multiplying factor blowing up near infinity, the first-order term in the Ornstein-Uhlenbeck operator plays a dominant role with respect not only to the Laplacian, but also to the full Hessian of a function, when norms sufficiently far from the $L^1\RG$ endpoint are taken.
Viceversa, when getting close to this endpoint, at which no embedding into a rearrangement-invariant space for the Ornstein-Uhlenbeck operator holds, the impact of the missing second-order derivatives and the gap between the properties of the Laplacian and the Hessian become apparent.

Let us now come to our discussion in full generality.
Let $A$ be a given Young function.
It is not restrictive to assume that
\begin{equation} \label{E:A-0}
	\int_{0} \frac{\tilde A(t)}{t^2}\,\d t < \infty.
\end{equation}
Indeed, one can replace, if necessary, $A$ with an equivalent Young function near infinity in such a way that condition \eqref{E:A-0} is fulfilled.
This replacement leaves the Orlicz space built upon $A$ unchanged, up to an equivalent norm.
Define $A_\lgn\colon[0,\infty)\to[0,\infty)$ by
\begin{equation} \label{E:AL-def}
	A_\lgn(t) = \int_{0}^{t} \frac{G_\lgn^{-1}(\tau)}{\tau}\,\d \tau
		\quad\text{for $t\ge 0$},
\end{equation}
where $G_\lgn\colon(0,\infty)\to[0,\infty)$ is given by
\begin{equation} \label{E:GL-def}
	G_\lgn(t)
		= \biggl\|
				\frac{1}{\tau\bar\ell(\tau)}
			\biggr\|_{L^{\tilde A}\left(\frac{1}{t},\infty\right)}
		\quad\text{for $t> 0$}
\end{equation}
and $\bar\ell(t)=\max\{\ell(t),1\}$.

\begin{theorem}[Reduction principle for embeddings in Orlicz spaces] \label{T:sou-Orlicz-reduction}
Let $A$ and $B$ be Young functions. Then
\begin{equation} \label{E:embedding-orlicz}
\Wlgn L^A\RG  \to L^B\RG
\end{equation}
 if and only if
\begin{equation} \label{E:sou-red-1}
	B \text{ is dominated by } A_\lgn
\end{equation}
and
\begin{equation} \label{E:sou-red-2}
	\tilde A \text{ is dominated by } \tilde B_\lgn,
\end{equation}
where $A_\lgn$ is the Young function given by equation~\eqref{E:AL-def}, and $\tilde B_\lgn$ is the Young function defined by the same equation, with $A$ replaced with $\tilde B$.

In particular, if $A\in\nabla_2$, then inequality~\eqref{E:embedding-orlicz} holds if and only if condition~\eqref{E:sou-red-1} is fulfilled, and if $B\in\Delta_2$, then inequality~\eqref{E:embedding-orlicz} holds if and only if condition~\eqref{E:sou-red-2} is fulfilled.
\end{theorem}

We shall derive this theorem as a corollary of a  more general result for weighted Hardy-type integral operators.
Let $\omega\colon(0,\infty)\to(0,\infty)$ be a weight, namely a nonnegative measurable function.
We consider the Hardy-type operator defined as
\begin{equation} \label{E:Rw-def}
	g\mapsto \int_{s}^{1} g(r)\omega(r)\,\d r
		\quad\text{for $s\in(0,1)$,}
\end{equation}
for $g \in \Mpl (0,1)$.
It  immediately follows via Fubini's theorem that its adjoint operator is given by
\begin{equation} \label{E:Rw'-def}
	g\mapsto \omega(s) \int_{0}^{s} g(r)\,\d r
		\quad\text{for $s\in(0,1)$,}
\end{equation}
for $g \in \Mpl (0,1)$.

With the choice of the  special weight
\begin{equation} \label{E:w-log}
	\omega(r) = \frac{1}{r\bar\ell(r)}
		\quad\text{for $r>0$},
\end{equation}
the operator $\sou$ defined in \eqref{E:S-def} takes the form
\begin{equation} \label{E:S-is-Rw+Rw'}
	\sou g(s)
		= \omega(s) \int_{0}^{s} g(r)\,\d r + \int_{s}^{1} g(r)\omega(r)\,\d r
	\quad\text{for $s \in (0,1)$,}
\end{equation}
for $g\in\Mpl(0,1)$.
Theorem~\ref{T:reduction-principle} thus enables us to transfer the study of Ornstein-Uhlenbeck Sobolev-type embeddings to the analysis of boundedness properties of the operators \eqref{E:Rw-def} and \eqref{E:Rw'-def} with $\omega$ as in \eqref{E:w-log}.

The method that will be developed is applicable to a large class of weights $\omega$ satisfying the mild assumptions that
\begin{equation} \label{E:w-prop}
	\omega\text{ is decreasing}
		\quad\text{and}\quad
	s\mapsto s\omega(s)\text{ is increasing in $(0,\infty)$}.
\end{equation}
A weight $w$ fulfilling conditions \eqref{E:w-prop} will be called an \emph{admissible weight} in what follows.

Given a Young function $A$, we define the Young function $A_\omega\colon[0,\infty)\to[0,\infty)$ by
\begin{equation} \label{E:Aw-def}
	A_\omega(t) = \int_{0}^{t} \frac{G_\omega^{-1}(\tau)}{\tau}\,\d \tau
		\quad\text{for $t\ge 0$},
\end{equation}
where $G_\omega\colon(0,\infty)\to[0,\infty)$ is given by
\begin{equation} \label{E:Gw-def}
	G_\omega(\tau) = \| \omega\|_{L^{\tilde A}\left(\frac{1}{\tau},\infty\right)}
		\quad\text{for $\tau>0$}.
\end{equation}
Clearly, if $\omega$ is as in \eqref{E:w-log}, then $G_\omega=G_\lgn$ and $A_\omega=A_\lgn$.

The definition of the function $A_\omega$ arises from the following basic result.

\begin{lemma} \label{L:young}
Let $\omega$ be an admissible weight and let $A$ be a Young function such that
\begin{equation} \label{E:A-w}
	\omega\chi_{(1,\infty)} \in L^{\tilde A}(0,\infty).
\end{equation}
Then the function  $A_\omega$ defined by \eqref{E:Aw-def} is a Young function, and
\begin{equation} \label{E:Aw-eq}
	A_\omega(t)\le G_\omega^{-1}(t)\le A_\omega(2t)
		\quad\text{for $t\ge 0$}.
\end{equation}
\end{lemma}

\begin{proof}
Condition~\eqref{E:A-w} ensures that $G_\omega$ is finite-valued.
By the definition of Luxemburg norm and a change of variables one has that
\begin{align} \label{E:Gw-monotone}
	\begin{split}
	\frac{G_\omega(\tau)}{\tau}
		& = \frac{1}{\tau}\inf\left\{\lambda>0:
					\int_{{1}/{\tau}}^\infty
						\tilde A\left( \frac{\omega(t)}{\lambda} \right)\d t\le 1
					\right\}
		= \inf\left\{\lambda>0:
					\int_{{1}/{\tau}}^\infty
						\tilde A\left( \frac{\omega(t)}{\lambda\tau} \right)\d t\le 1
					\right\}
			\\
		& = \inf\left\{\lambda>0:
					\frac1\tau\int_{1}^\infty
						\tilde A\left( \frac{\omega(t/\tau)}{\lambda\tau} \right)\d t\le 1
					\right\}
			\quad \text{for $\tau >0$.}
	\end{split}
\end{align}
Owing to assumption \eqref{E:w-prop}, the function $\tilde A\left(\frac{\omega(t/\tau)}{\lambda\tau} \right)$ is decreasing in $\tau$ for every $t$ and $\lambda$, and therefore $G_\omega(\tau)/\tau$ is decreasing.
This implies that $G_\omega^{-1}(\tau)/\tau$ is increasing and, consequently, $A_\omega$ is a Young function.
Furthermore, inequalities~\eqref{E:Aw-eq} follow from the monotonicity of the function $G_\omega^{-1}(\tau)/\tau$.
\end{proof}

The Young function $A_\omega$ enters the definition of the optimal Orlicz target space for the operator~\eqref{E:Rw-def} on the domain $L^A(0,1)$.
This is the content of the following result.

\begin{theorem} \label{T:Rw-Orlicz-reduction}
Let $A$ and $B$ be Young functions, let $\omega$ be an admissible weight, and let $A_\omega$ be the Young function given by \eqref{E:Aw-def}.
Then there exists a constant $c$ such that
\begin{equation} \label{E:Rw-LA-LB}
	\left\|\int_{s}^{1} g(r)\omega(r)\,\d r\right\|_{L^B(0,1)}
		\le c \|g\|_{L^A(0,1)}
\end{equation}
for every $g\in L^A(0,1)$
if and only if
\begin{equation} \label{E:BleAw}
	B \text{ is dominated by } A_\omega.
\end{equation}
\end{theorem}

The next theorem can be derived from  Theorem \ref{T:Rw-Orlicz-reduction} via a duality argument.

\begin{theorem} \label{T:Rw'-Orlicz-reduction}
Let $A$ and $B$ be Young functions, let $\omega$ be an admissible weight, and let $\tilde B_\omega$ be the Young function associated with $\tilde B$ as in \eqref{E:Aw-def}.
Then there exists a constant $c$ such that
\begin{equation} \label{E:Rw-dual-LA-LB}
	\left\|\omega(s)\int_{0}^{s} g(r)\,\d r\right\|_{L^B(0,1)}
		\le c \|g\|_{L^A(0,1)}
\end{equation}
for every $g\in L^A(0,1)$ if and only if
\begin{equation}
	\tilde A \text{ is dominated by } \tilde{B}_\omega.
\end{equation}
\end{theorem}

The proof of Theorem \ref{T:Rw-Orlicz-reduction} consists of two steps.
First, in Proposition~\ref{P:Rw-LA-MB} we show that condition \eqref{E:BleAw} characterizes the boundedness of the operator in \eqref{E:Rw-def} between $L^A$ and the weak Orlicz space~$M^B$.
Second, in Proposition~\ref{P:Rw-from-weak-to-strong} we prove a self-improving property which ensures that the boundedness into a Marcinkiewicz space $L^B$ can be lifted to the Orlicz space $L^B$.

A similar scheme appeared earlier in the literature in the treatment of various Sobolev and Hardy-type inequalities, starting with the pioneering work of V.~Maz'ya in the early sixties of the last century, as recorded in \citep{Maz:11}.

For technical reasons, we first prove a version of our results for Orlicz spaces defined on $(0,\infty)$ instead of $(0,1)$.

\begin{proposition} \label{P:Rw-LA-MB}
Let $A$ and $B$ be Young functions and let $\omega$ be an admissible weight.
The following conditions are equivalent.
\begin{enumerate}
\item There exists a constant $c_1$ such that
\begin{equation} \label{E:Rw-LA-MB-infty}
	\left\|\int_{t}^{\infty} g(\tau)\omega(\tau)\,\d\tau\right\|_{M^B(0,\infty)}
		\le c_1 \|g\|_{L^A(0,\infty)}
\end{equation}
for every $g\in L^A(0,\infty)$.
\item $A$ satisfies condition \eqref{E:A-w} and there exists a constant $c_2$ such that
\begin{equation} \label{E:BleAw-infty}
	B(t) \le A_\omega(c_2t)
		\quad\text{for $t\ge 0$},
\end{equation}
where $A_\omega$ is the Young function given by~\eqref{E:Aw-def}.
\end{enumerate}
Moreover, the constants $c_1$ and $c_2$ depend only on each other.
\end{proposition}

\begin{proof}
Thanks to property~\eqref{E:novabis}, inequality~\eqref{E:Rw-LA-MB-infty} is equivalent to
\begin{equation} \label{E:Rw'-weak2}
	\left\|\omega(t)\int_{0}^{t} g(\tau)\,\d\tau\right\|_{L^{\tilde A}(0,\infty)}
		\le c_1' \|g\|_{\Lambda^{\tilde B}(0,\infty)}
\end{equation}
for every $g\in\Lambda^{\tilde B}(0,\infty)$, where $c_1'>0$ depends only in $c_1$.
In turn, inequality~\eqref{E:Rw'-weak2} is equivalent to
\begin{equation} \label{E:Rw'-weak-star}
	\left\|\omega(t)\int_{0}^{t} g^*(\tau)\,\d\tau\right\|_{L^{\tilde A}(0,\infty)}
		\le c_1' \|g^*\|_{\Lambda^{\tilde B}(0,\infty)}
\end{equation}
for every $g\in\Lambda^{\tilde B}(0,\infty)$.
The latter equivalence is a consequence of inequality~\eqref{HL} and of the fact that the norm in $\Lambda^{\tilde B}(0,\infty)$ is rearrangement invariant.

Next, by \citep[Proposition~3.4]{Mus:16}, inequality~\eqref{E:Rw'-weak-star} is equivalent to the same inequality restricted to characteristic functions of the form $\chi_{(0,\rho)}$ for $\rho>0$.
Namely,
\begin{equation} \label{E:Rw'-weak-char}
	\left\|\omega(t)\int_{0}^{t} \chi_{(0,\rho)}(r)\, \d r\,\right\|_{L^{\tilde A}(0,\infty)}
		\le c_1' \|\chi_{(0,\rho)}\|_{\Lambda^{\tilde B}(0,\infty)}
	\quad\text{for $\rho>0$}.
\end{equation}
By equations~\eqref{E:Orlicz-char} and~\eqref{E:conjugate-invers},
\begin{equation} \label{E:Lambda-char}
	\tfrac12\rho B^{-1}\bigl( \tfrac{1}{\rho} \bigr)
		\le \|\chi_{(0,\rho)}\|_{\Lambda^{\tilde B}(0,\infty)}
		\le \rho B^{-1}\bigl( \tfrac{1}{\rho} \bigr)
	\quad\text{for $\rho>0$}.
\end{equation}
Notice that
\begin{equation*}
	\omega(t)\int_{0}^{t} \chi_{(0,\rho)}  (\tau)\, \d \tau\
		= \chi_{(0,\rho)}(t)t\omega(t)
			+ \chi_{(\rho, \infty)}(t)\rho \omega(t)
		\quad\text{for $t,\rho>0$}.
\end{equation*}
Thanks to properties~\eqref{E:w-prop} of the weight $\omega$,
\begin{align*}
	\rho \|\omega\|_{L^{\tilde A}(\rho,\infty)}
		& \ge \rho \|\omega\|_{L^{\tilde A}(\rho,2\rho)}
			\ge \rho \omega(2\rho) \|\chi_{(\rho,2\rho)}\|_{L^{\tilde A}(0,\infty)}
			\ge \thalf \rho \omega(\rho) \|\chi_{(0,\rho)}\|_{L^{\tilde A}(0,\infty)}
			\ge \thalf \|s\omega(s)\|_{L^{\tilde A}(0,\rho)}
\end{align*}
for $\rho>0$, whence
\begin{equation*}
	\rho \|\omega\|_{L^{\tilde A}(\rho,\infty)}
		\le \left\|\omega(t)\int_{0}^{t} \chi_{(0,\rho)}(\tau)\,\d \tau\right\|_{L^{\tilde A}(0,\infty)}
		\le 3 \rho \|\omega\|_{L^{\tilde A}(\rho,\infty)}
	\quad\text{for $\rho>0$}.
\end{equation*}
Altogether, inequality~\eqref{E:Rw'-weak-char} holds if and only if there exists a constant $c_2'>0$ such that
\begin{equation} \label{E:Gw-le-Binv}
	G_\omega\bigl(\tfrac{1}{\rho}\bigr)
		= \|\omega\|_{L^{\tilde A}(\rho,\infty)}
		\le c_2' B^{-1}\bigl(\tfrac{1}{\rho}\bigr)
		\quad\text{for $\rho>0$}.
\end{equation}
Since $B$ is a Young function, $B^{-1}$ is finite-valued and, consequently, condition~\eqref{E:A-w} follows from \eqref{E:Gw-le-Binv}.
Inequality~\eqref{E:Gw-le-Binv} implies that $B(t)\le G_\omega^{-1}(c_2' t)$ for $t>0$, an equivalent form of \eqref{E:BleAw-infty}, owing to inequalities~\eqref{E:Aw-eq}.
\end{proof}

\begin{proposition} \label{P:Rw-from-weak-to-strong}
Let $A$ and $B$ be Young functions and let $\omega$ be an admissible weight.
Assume that there exists a constant $c_1$ such that
\begin{equation} \label{E:Rw-LA-MB-infty-2}
	\left\|\int_{t}^{\infty} g(\tau)\omega(\tau)\,\d\tau\right\|_{M^B(0,\infty)}
		\le c_1 \|g\|_{L^A(0,\infty)}
\end{equation}
for every $g\in L^A(0,\infty)$.
Then there exists a constant $c_2$, depending only on $c_1$, such that
\begin{equation} \label{E:Rw-LA-LB-infty}
	\left\|\int_{t}^{\infty} g(\tau)\omega(\tau)\,\d\tau\right\|_{L^B(0,\infty)}
		\le c_2 \|g\|_{L^A(0,\infty)}
\end{equation}
for every $g\in L^A(0,\infty)$.
\end{proposition}

\begin{proof}
Denote by $R$ the operator defined as
\begin{equation}\label{R}
	R g (t) = \int_{t}^{\infty} g(\tau)\omega(\tau)\,\d\tau
		\quad\text{for $t>0$,}
\end{equation}
for $g\in\Mpl(0, \infty)$.

Assume that $A$ and $B$ obey \eqref{E:Rw-LA-MB-infty-2}.
Let $N\in(0,1]$ and set
\begin{equation*}
	A_N(t) = \frac{A(t)}{N}
		\quad\text{and}\quad
	B_N(t) = \frac{B(t)}{N}
		\quad\text{for $t\ge 0$.}
\end{equation*}
Clearly, $A_N$ and $B_N$ are Young functions.
We shall show that
\begin{equation} \label{E:Rw-LA-MB-scaled}
	\|R g\|_{M^{B_N}(0,\infty)} \le c_2' \|g\|_{L^{A_N}(0,\infty)}
\end{equation}
for $g\in L^{A_N}(0,\infty)$, where the constant and $c_2'>0$ is independent of $N$.
By Proposition~\ref{P:Rw-LA-MB}, inequality~\eqref{E:Rw-LA-MB-infty-2} implies
\begin{equation} \label{E:BleAw-2}
	B(t) \le A_\omega(c_1't)
		\quad\text{for $t\ge 0$}
\end{equation}
for some constant $c_1'>0$, where $A_\omega$ is defined by \eqref{E:Aw-def}.
Let $(G_N)_\omega$ and $(A_N)_\omega$ be the functions associated with $G_N$ and $A_N$ as in \eqref{E:Gw-def} and \eqref{E:Aw-def}, respectively.
Since $\tilde{A_N}(t) = \tilde A(Nt)/N$ for $t\ge 0$, we have that
\begin{align*}
	(G_N)_\omega(t)
		& = \inf\left\{\lambda>0:
					\frac1N \int_{{1}/{t}}^\infty
							\tilde A\left( \frac{N\omega(r)}{\lambda} \right)\d r\le 1
						\right\}
			= \inf\left\{\lambda>0:
					\int_{{1}/{Nt}}^\infty
							\tilde A\left( \frac{N\omega(Nr)}{\lambda} \right)\d r\le 1
						\right\}
				\\
		& \le \inf\left\{\lambda>0:
					\int_{{1}/{Nt}}^\infty
							\tilde A\left( \frac{\omega(r)}{\lambda} \right)\d r\le 1
						\right\}
			= G_\omega(Nt)
	\quad\text{for $t>0$},
\end{align*}
where the inequality follows since $N\omega(Nr)\le \omega(r)$ for $N\le 1$ and $r>0$.
Therefore,
\begin{equation*}
	G_\omega^{-1}(t)\le N(G_N)_\omega^{-1}(t)
		\quad\text{for $t>0$},
\end{equation*}
and, consequently,
\begin{equation*}
	A_\omega(t) \le N(A_N)_\omega(t)
		\quad\text{for $t\ge 0$}.
\end{equation*}
Inequality~\eqref{E:BleAw-2} yields
\begin{equation*}
	B_N(t)
		= \frac{1}{N}B(t)
		\le \frac{1}{N} A_\omega(c_1't)
		\le (A_N)_\omega(c_1't)
	\quad\text{for $t\ge 0$},
\end{equation*}
which in turn implies \eqref{E:Rw-LA-MB-scaled} by Proposition~\ref{P:Rw-LA-MB}.

Let $g\in\MM_+(0,\infty)$ be such that
\begin{equation} \label{E:Afle1}
	\int_{0}^{\infty} A\bigl(g(s)\bigr)\d s \le 1.
\end{equation}
On setting
\begin{equation*}
	N = \int_{0}^{\infty} A\bigl(g(s)\bigr)\d s,
\end{equation*}
we have that $\|g\|_{L^{A_N}(0,\infty)}\le 1$ and, by inequality~\eqref{E:Rw-LA-MB-scaled},
\begin{equation} \label{E:Rw-norm-weak-le-C}
	\|R g\|_{M^{B_N}(0,\infty)} \le c_2'.
\end{equation}
By the definition of the Marcinkiewicz norm, inequality \eqref{E:Rw-norm-weak-le-C} implies that
\begin{equation*}
	c_2'
		\ge \|R g\|_{M^{B_N}(0,\infty)}
		\ge \sup_{\tau\in(0,\infty)} \frac{(Rg)^*(\tau)}{B_N^{-1}(1/\tau)}
		= \sup_{t\in(0,\infty)} \frac{t}{B_N^{-1}\bigl( 1/|\{R g \ge t\}| \bigr)}.
\end{equation*}
The latter inequality is equivalent to
\begin{equation} \label{E:Rw-weak-ineq}
	|\{R g \ge t\}|\, B\left(\tfrac{t}{c_2'}\right)
		\le \int_{0}^{\infty} A\bigl(g(s)\bigr)\d s
	\quad\text{for $t>0$,}
\end{equation}
for every $g\in\MM_+(0,\infty)$ satisfying condition~\eqref{E:Afle1}.

We conclude the proof by showing that the weak type estimate~\eqref{E:Rw-weak-ineq} implies strong type estimate
\begin{equation} \label{E:Rw-strong-ineq}
	\int_{0}^{\infty} B\left( \frac{R g(s)}{4c_2'} \right)\d s
		\le \int_{0}^{\infty} A\bigl(g(s)\bigr)\d s
\end{equation}
for every $g\in\MM_+(0,\infty)$ obeying~\eqref{E:Afle1}.
To this end, we use a classical discretization argument.
Given $g\in\MM_+(0,\infty)$, let $\{s_k\}$ be a sequence in $(0,\infty)$ such that
\begin{equation} \label{E:sk-def}
	R g(s_k) = 2^k
	\quad\text{for $k\in\Z$}.
\end{equation}
If $R g$ is bounded, then $k$ ranges from $-\infty$ to the smallest $K\in\Z$ such that $Rg(s_K)\le 2^K$.
Then, $s_K=0$ and $s_k$ is given by \eqref{E:sk-def} for $k<K$.
In the latter computations, $K$ thus denotes either $\infty$ or an integer.
Since the function $Rf$ is non-increasing, the sequence $\{s_k\}$ is non-increasing as well.
Thereby
\begin{equation*}
	Rg(s) \le Rg(s_{k+1}) = 2^{k+1}
		\quad\text{for $s\in[s_{k+1}, s_k)$}
\end{equation*}
and
\begin{align} \label{E:discretization}
    \begin{split}
	\int_{0}^{\infty} B\left( \frac{R g(s)}{4c_2'} \right)\d s
		&= \sum_{k<K} \int_{s_{k+1}}^{s_k}
				B\left( \frac{R g(s)}{4c_2'} \right)\d s
				\\
		&\le \sum_{k<K} \int_{s_{k+1}}^{s_k}
				B\left( \tfrac{2^{k+1}}{4c_2'} \right)\d s
		= \sum_{k<K} (s_k - s_{k+1})\,
				B\left( \tfrac{2^{k-1}}{c_2'} \right).
        \end{split}
\end{align}
Next, we define $g_k = g\chi_{[s_{k},s_{k-1})}$ for $k<K-1$.
If $s\in[s_{k+1},s_k)$, then
\begin{equation*}
	Rg_k(s)
		= \int_{s}^{\infty} g(r)\chi_{[s_{k},s_{k-1})}(r)\omega(r)\,\d r
		\ge \int_{s_{k}}^{s_{k-1}} g(r)\omega(r)\,\d r
		= Rg(s_{1}) - Rg(s_{k-1})
		= 2^{k-1}.
\end{equation*}
Hence, $\{Rg_k\ge 2^{k-1}\}\supseteq[s_{k+1},s_k)$.
Coupling this piece of information with the weak type estimate~\eqref{E:Rw-weak-ineq}, with $g$ replaced by $g_k$ and $t=2^{k-1}$, enables us to infer that
\begin{equation} \label{E:discretization-2}
	(s_k - s_{k+1})\, B\left( \tfrac{2^{k-1}}{c_2'} \right)
		\le |\{Rg_k\ge 2^{k-1}\}|\, B\left( \tfrac{2^{k-1}}{c_2'} \right)
		\le \int_{0}^{\infty} A\bigl(g_k(s)\bigr)\d s
		\le \int_{s_{k}}^{s_{k-1}} A\bigl(g(s)\bigr)\d s.
\end{equation}
Inequalities~\eqref{E:discretization} and \eqref{E:discretization-2} yield inequality~\eqref{E:Rw-strong-ineq}, which, in turn, implies~\eqref{E:Rw-LA-LB-infty}.
\end{proof}

The following result is a straightforward consequence of Propositions~\ref{P:Rw-LA-MB} and \ref{P:Rw-from-weak-to-strong}.

\begin{theorem} \label{T:Rw-Orlicz-reduction-infty}
Let $A$ and $B$ be Young functions, let $\omega$ be an admissible weight, and let $A_\omega$ be the Young function given by~\eqref{E:Aw-def}.
The following conditions are equivalent.
\begin{enumerate}
\item There exists a constant $c_1$ such that
\begin{equation} \label{E:Rw-LA-LB-infty-2}
	\left\|\int_{t}^{\infty} g(\tau)\omega(\tau)\,\d\tau\right\|_{L^B(0,\infty)}
		\le c_1 \|g\|_{L^A(0,\infty)}
\end{equation}
for every $g\in L^A(0,\infty)$.
\item The function $A$ satisfies condition \eqref{E:A-w} and there exists a constant $c_2$ such that
\begin{equation}
	B(t) \le A_\omega(c_2t)
		\quad\text{for $t\ge 0$.}
\end{equation}
\end{enumerate}
Moreover, the constants $c_1$ and $c_2$ depend only on each other.
\end{theorem}

We are now in a position to prove Theorem~\ref{T:Rw-Orlicz-reduction}.

\begin{proof}[Proof of Theorem~\ref{T:Rw-Orlicz-reduction}]
Throughout this proof, $c$ denotes a constant whose value may differ at various occurrences.
Assume that $A_\omega$ and $B$ satisfy condition~\eqref{E:BleAw}, \ie there exists  $t_0\ge 0$ such that $B(t)\le A_\omega(ct)$ for $t\ge t_0$.
Let $\widehat A$ and $\widehat B$ be Young functions that agree with $A$ and $B$ near infinity, and such that $\widehat A$ obeys condition~\eqref{E:A-w} and
\begin{equation} \label{E:BleAw-hats}
	\widehat B(t)\le \widehat A_\omega(ct)
		\quad\text{for $t\ge 0$}.
\end{equation}
Here, $\widehat A_\omega$ denotes the Young function associated with $\widehat A$ as in \eqref{E:Aw-def}.
By Theorem~\ref{T:Rw-Orlicz-reduction-infty}, condition~\eqref{E:BleAw-hats} ensures that
\begin{equation} \label{E:Rw-LA-MB-infty-hats}
	\left\|\int_{t}^{\infty} g(\tau)\omega(\tau)\,\d\tau\right\|_{L^{\widehat B}(0,\infty)}
		\le  c\|g\|_{L^{\widehat A}(0,\infty)}
\end{equation}
for every $g\in L^{\widehat A}(0,\infty)$.
Observe that, if $g\in L^{\widehat A}(0,1)$, then an application of inequality~\eqref{E:Rw-LA-MB-infty-hats} to its extension by $0$ outside $(0,1)$ yields an analogous inequality with $L^{\widehat A}(0,\infty)$ and $L^{\widehat B}(0,\infty)$ replaced with $L^{\widehat A}(0,1)$ and $L^{\widehat B}(0,1)$.
Since the latter spaces agree with $L^A(0,1)$ and $L^B(0,1)$ (up to equivalent norms), inequality~\eqref{E:Rw-LA-LB} follows.

Conversely, assume that inequality~\eqref{E:Rw-LA-LB} holds.
We may assume that $A$ satisfies condition~\eqref{E:A-w}, since the function $A$ can be modified near zero, if necessary, without changing the corresponding Orlicz space $L^A(0,1)$, up to equivalent norms.
By property~\eqref{E:novabis},
\begin{equation} \label{E:Rw-LA-LB-dual}
	\left\|\omega(s)\int_{0}^{s} g(r)\,\d r \right\|_{L^{\tilde A}(0,1)}
		\le c \|g\|_{L^{\tilde B}(0,1)}
\end{equation}
for every $g\in L^{\tilde B}(0,1)$.
Next, setting $g=\chi_{(0,\rho)}$ for $\rho\in(0,1)$ in inequality~\eqref{E:Rw-LA-LB-dual} results in
\begin{equation*}
	\rho \|\omega\|_{L^{\tilde A}(\rho,1)}
		\le \left\|\omega(s)\int_{0}^{s} \chi_{(0,\rho)}(r)\,\d r\right\|_{L^{\tilde A}(0,1)}
		\le c \|\chi_{(0,\rho)}\|_{L^{\tilde B}(0,1)}
		\le c\rho B^{-1}\bigl( \tfrac{1}{\rho} \bigr)
	\quad\text{for $\rho\in(0,1)$}.
\end{equation*}
Since, owing to assumption \eqref{E:A-w}, $\|\omega\|_{L^{\tilde A}(1,\infty)}<\infty$, we have that
\begin{equation} \label{E:Gw-leBi}
	G_\omega\bigl( \tfrac{1}{\rho} \bigr)
		= \|\omega\|_{L^{\tilde A}(\rho,\infty)}
		\le cB^{-1}\bigl( \tfrac{1}{\rho} \bigr)
		\quad\text{for $\rho\in(0,1)$}.
\end{equation}
Hence, equation~\eqref{E:BleAw} follows.
\end{proof}

\begin{proof}[Proof of Theorem~\ref{T:sou-Orlicz-reduction}]
Owing to Theorem~\ref{T:reduction-principle}, embedding~\eqref{E:embedding-orlicz} holds if and only if the operator $\sou$ is bounded between $L^A(0,1)$ and
$L^B(0,1)$.
In turn, the operator $\sou$ is bounded if and only if both the operator $R$, defined as in~\eqref{E:Rw-def} with $\omega$ given by \eqref{E:w-log}, and its adjoint $R'$, having the form~\eqref{E:Rw'-def}, are bounded.
Now, Theorem~\ref{T:Rw-Orlicz-reduction} asserts that $R$ is bounded from $L^A(0,1)$ into $L^B(0,1)$ if and only if condition \eqref{E:sou-red-1} holds.
Condition~\eqref{E:sou-red-2} is therefore necessary and sufficient for the boundedness of $R'$.
Note that condition~\eqref{E:A-w} agrees with~\eqref{E:A-0} with the special choice \eqref{E:w-log}.

To prove the assertion concerning the case when $A \in \nabla_2$, observe that
\begin{equation} \label{E:sou-RwP}
	\sou \approx R\circ P,
\end{equation}
with absolute equivalence constants, where $P$ is the averaging operator given by $Pg(s)=\frac 1s\int_{0}^{s} g(r)\,\d r$ for $s\in(0,1)$ and $g\in \Mpl(0,1)$.
Indeed,
\begin{align*}
	RPg(s)
		& = \int_{s}^{1} \ir \int_{0}^{r} g(\varrho)\,\d\varrho\, \omega(r)\,\d r
				\\
		& = \int_{0}^{s} g(\varrho)\,\d\varrho \int_{s}^{1} \frac{\omega(r)}{r}\,\d r
				+ \int_{s}^{1} g(\varrho)\int_{\varrho}^{1} \frac{\omega(r)}{r}\,\d r\,\d\varrho
				\\
		& \approx \omega(s) \int_{0}^{s} g(\varrho)\,\d\varrho
				+ \int_{s}^{1} g(\varrho)\omega(\varrho)\,\d\varrho
			= \sou g(s)
		\quad\text{for $s\in(0,1)$}.
\end{align*}
Here, we have made use of the fact that, if $\omega$ is given by~\eqref{E:w-log}, then
\begin{equation}
	\int_{\varrho}^{1} \frac{\omega(r)}{r}\,\d r
		\approx \omega(\varrho)
		\quad\text{for $\varrho\in(0, \thalf)$},
\end{equation}
with absolute equivalence constants.
Therefore, inequality \eqref{E:embedding-orlicz} holds if and only if $R\circ P$ is bounded from $L^A(0,1)$ into $L^B(0,1)$.
If $A\in\nabla_2$, then the operator $P$ is bounded on $L^A(0,1)$, see \eg~\citep{Gal:88}.
Consequently, inequality \eqref{E:embedding-orlicz} holds provided that $R$ is bounded from $L^A(0,1)$ into $L^B(0,1)$, and this boundedness is equivalent to condition~\eqref{E:sou-red-1}.
Conversely, the necessity of \eqref{E:sou-red-1} is a consequence of the first part of the statement.

The assertion about the case when $B \in \Delta_2$ can be verified via a duality argument.
Since $\sou$ is self-adjoint, inequality~\eqref{E:embedding-orlicz} holds if and only if $\sou$ is bounded from $L^{\tilde B}(0,1)$ into $L^{\tilde A}(0,1)$.
Since $B\in\Delta_2$, we have that $\widetilde B \in \nabla_2$.
Hence, $P$ is bounded on $L^{\tilde B}(0,1)$.
Therefore, thanks to condition~\eqref{E:sou-RwP} again, inequality~\eqref{E:embedding-orlicz} holds if $R$ is bounded from $L^{\tilde B}(0,1)$ into $L^{\tilde A}(0,1)$.This fact is guaranteed under condition~\eqref{E:sou-red-2}.
The necessity of the latter condition~\eqref{E:sou-red-2} follows from the first part of the statement.
\end{proof}

\section{Ornstein--Uhlenbeck embeddings in Lorentz-Zygmund spaces}\label{LZ}

Here we exploit our general results to derive Sobolev inequalities for the Ornstein--Uhlenbeck  operator in Lorentz-Zygmund spaces.
This is the subject of the following theorem.

\begin{theorem}[Optimal embeddings for Lorentz--Zygmund spaces]\label{T:optimal-embedding-lz}
Let $p,q\in[1,\infty]$ and $\alpha,\beta\in\R$.
Then
\begin{equation}\label{E:optimal-embedding-lz}
    \Wlgn L^{p,q;\alpha,\beta}\to
        \begin{cases}
            L^{1,1;0,\beta-1} &\text{if $p=q=1$, $\alpha=0$, $\beta\ge1$}
                \\
            L^{1,1;\alpha,\beta} &\text{if $p=q=1$, $\alpha>0$, $\beta\in\R$}
                \\
            L^{p,q;\alpha+1,\beta}&\text{if $p\in(1,\infty)$, $\alpha\in\R$, $\beta\in\R$}
                 \\
            L^{\infty,\infty;\alpha,\beta}&\text{if $p=q=\infty$, $\alpha<0$, $\beta\in\R$}
                 \\
            L^{\infty,\infty;0,\beta-1}&\text{if $p=q=\infty$, $\alpha=0$, $\beta\le0$,}
       \end{cases}
\end{equation}
where all the spaces are over $\RG$.
Moreover, in each case, the target space is optimal (smallest) among all rearrangement-invariant spaces, and, simultaneously, the domain is optimal (largest) among all rearrangement-invariant spaces.
\end{theorem}

Let us mention that some cases of the embeddings in~\eqref{E:optimal-embedding-lz} can be found in~\citep{Bla:07}.
However, their optimality is not discussed in that paper.

In this section, without further explicit reference, we shall repeatedly use well-known characterizations of the associate spaces of Lorentz--Zygmund spaces, which can be found for instance in~\citep[Section~9.6]{Pi:13}.
We will also use without further warnings the fact that the function $\sou g$ is non-increasing for every function $g\in \MM_+(0,1)$.

\begin{proof}[Proof of Theorem \ref{T:optimal-embedding-lz}]
Let $\beta\ge0$ and set $X(0,1)=L^{1,1;0,\beta+1}(0,1)$.
Then condition~\eqref{E:target-condition} is satisfied, and by equation \eqref{july1} one has that
\begin{equation*}
	\|g\|_{X_{\lgn}'(0,1)}
		= \|\sou g^{*}\|_{L^{\infty,\infty;0,-\beta-1}(0,1)}
		= \sup_{s\in(0,1)}\ell\ell(s)^{-\beta-1}\sou g^*(s)
\end{equation*}
for $g\in\Mpl(0,1)$.
Next,
\begin{equation*}
	g^*(r) \le  \|g\|_{L^{\infty,\infty;0,-\beta}(0,1)}\, \ell\ell^{\beta}(r)
		\quad\text{for $r\in(0,1)$,}
\end{equation*}
whence
\begin{align*}
	\|g\|_{X_{\lgn}'(0,1)}
		& = \sup_{s\in(0,1)} \frac{1}{\ell\ell(s)^{\beta+1}}
				\left(
					\frac{1}{s\ell(s)} \int_{0}^{s} g^*(r)\,\d r
						+ \int_{s}^{1} \frac{g^*(r)}{r\ell(r)}\d r
				\right)
    	\\
    & \le \sup_{s\in(0,1)} \frac{1}{\ell\ell(s)^{\beta+1}}
				\left(
					\frac{1}{s\ell(s)} \int_{0}^{s} \ell\ell^{\beta}(r)\,\d r
					+ \int_{s}^{1} \frac{\ell\ell^{\beta}(r)}{r\ell(r)}\d r
				\right)
				\|g\|_{L^{\infty,\infty;0,-\beta}(0,1)}
			\\
    & \lesssim\, \|g\|_{L^{\infty,\infty;0,-\beta}(0,1)}
\end{align*}
for $g\in\Mpl(0,1)$, up to a constant depending on $\beta$.
This proves the embedding
\begin{equation} \label{E:emb-110b}
	X_{\lgn}\RG \to L^{1,1;0,\beta}\RG.
\end{equation}
In order to establish the converse embedding, we define the function $\eta\colon\left(0,e^{1-e}\right)\to(0,1)$ by
\begin{equation}\label{E:eta}
    \eta(s) = \ell\ell^{-1}\bigl(\tfrac12\ell\ell(s)\bigr) \quad\text{for $s\in(0,e^{1-e})$,}
\end{equation}
where $\ell\ell^{-1}$ denotes the inverse of the function $\ell\ell$.
Then $\eta$ is increasing and $\eta\bigl(\left(0,e^{1-e}\right)\bigr)=(0,1)$.
Moreover, one has
\begin{equation*}
	\ell\ell(\eta(s))
		= \thalf\ell\ell(s)
	\quad\text{for $s\in\left(0,e^{1-e}\right)$}
\end{equation*}
and
\begin{equation*}
    \ell\ell(\eta^{-1}(\tau)) = 2\ell\ell(\tau)
	\quad\text{for $\tau\in\left(0,1\right)$.}
\end{equation*}
Thus,
\begin{align*}
   \|g\|_{X_{\lgn}'(0,1)}
		& \ge \sup_{s\in(0,e^{1-e})} \frac{1}{\ell\ell(s)^{\beta+1}} \int_{s}^{\eta(s)} \frac{g^*(r)}{r\ell(r)}\d r
        \\
    &\ge \sup_{s\in(0,e^{1-e})} g^{*}\bigl(\eta(s)\bigr) \frac{1}{\ell\ell(s)^{\beta+1}} \int_{s}^{\eta(s)} \frac{\d r}{r\ell(r)}
         \\
    &= \sup_{s\in(0,e^{1-e})} g^{*}\bigl(\eta(s)\bigr)\ell\ell(s)^{-\beta-1}\bigl[\ell\ell(s)-\ell\ell\bigl(\eta(s)\bigr)\bigr]
         \\
    &= \thalf \sup_{s\in(0,e^{1-e})} g^{*}\bigl(\eta(s)\bigr)\ell\ell(s)^{-\beta}
     = \thalf \sup_{\tau\in(0,1)}g^{*}(\tau)\ell\ell\bigl(\eta^{-1}(\tau)\bigr)^{-\beta}
        \\
    & = 2^{-1-\beta}\sup_{s\in(0,1)}g^{*}(s)\ell\ell(s)^{-\beta} = 2^{-1-\beta}\|g\|_{L^{\infty,\infty;0,-\beta}(0,1)}.
\end{align*}
This chain implies the embedding $L^{1,1;0,\beta}\RG \to X_{\lgn}\RG$.
Coupling the latter embedding with \eqref{E:emb-110b} yields
\begin{equation*}
    X_{\lgn}\RG = L^{1,1;0,\beta}\RG,
\end{equation*}
up to equivalent norms.
Owing to Theorem~\ref{T:OU-target}, this shows that the target space in the first embedding in~\eqref{E:optimal-embedding-lz} is optimal among all rearrangement-invariant target spaces.

Now let $\alpha>0$, $\beta\in\R$ and $X(0,1)=L^{1,1;\alpha,\beta}(0,1)$.
Then condition~\eqref{E:target-condition} is satisfied again, and
\begin{align*}
	\|g\|_{X_{\lgn}'(0,1)}
		& = \|\sou g^*\|_{L^{\infty,\infty;-\alpha,-\beta}(0,1)}
				\\
		& = \sup_{s\in(0,1)}
				\left(
					\frac{1}{s\ell(s)^{\alpha+1}\ell\ell(s)^{\beta}} \int_{0}^{s} g^{*}(r)\,\d r
					  + \frac{1}{\ell(s)^{\alpha}\ell\ell(s)^{\beta}} \int_{s}^{1} \frac{g^{*}(r)}{r\ell(r)}\d r
				\right).
\end{align*}
On the other hand,
\begin{equation*}
	g^*(r)\le \|g\|_{L^{1,1;\alpha,\beta}(0,1)}\, \ell(r)^{\alpha}\ell\ell(r)^{\beta}
		\quad\text{for $r\in(0,1)$}.
\end{equation*}
Therefore,
\begin{align*}
	& \|g\|_{X_{\lgn}'(0,1)}
			\\
	&\quad \le \sup_{s\in(0,1)}
			\left(
				\frac{1}{s\ell(s)^{\alpha+1}\ell\ell(s)^{\beta}} \int_{0}^{s}\ell(r)^{\alpha}\ell\ell(r)^{\beta}\,\d r
				+ \frac{1}{\ell(s)^{\alpha}\ell\ell(s)^{\beta}} \int_{s}^{1} \frac{\ell\ell(r)^{\beta}}{r\ell(r)^{1-\alpha}}\d r
			\right)
			\|g\|_{L^{1,1;\alpha,\beta}(0,1)}
			\\
	&\quad \lesssim \, \|g\|_{L^{1,1;\alpha,\beta}(0,1)},
\end{align*}
up to a constant depending on $\alpha$ and $\beta$.
The embedding
\begin{equation} \label{E:emb-11ab}
	X_{\lgn}\RG\to L^{1,1;\alpha,\beta}\RG
\end{equation}
is thus established.
In order to prove the converse embedding, define the function $\sigma\colon(0,e^{1-e})\to(0,1)$ as
\begin{equation}\label{E:sigma}
	\sigma(s) = \ell\ell^{-1}\bigl(\ell\ell(s)-1\bigr)
		\quad\text{for $s \in (0,e^{1-e})$.}
\end{equation}
Then $\sigma$ is increasing and $\sigma\bigl((0,e^{1-e})\bigr)=(0,1)$.
Thus,
\begin{align*}
   \|g\|_{X_{\lgn}'(0,1)}
		&\ge \sup_{s\in(0,e^{1-e})} \frac{1}{\ell(s)^{\alpha}\ell\ell(s)^{\beta}}
						\int_{s}^{\sigma(s)} \frac{g^*(r)}{r\ell(r)}\d r
     \ge \sup_{s\in(0,e^{1-e})} \frac{g^{*}\bigl(\sigma(s)\bigr)}{\ell(s)^{\alpha}\ell\ell(s)^{\beta}}
						\int_{s}^{\sigma(s)} \frac{\d r}{r\ell(r)}
         \\
    &= \sup_{s\in(0,e^{1-e})}g^{*}\bigl(\sigma(s)\bigr)\ell(s)^{-\alpha}\ell\ell(s)^{-\beta}\bigl[\ell\ell(s)-\ell\ell\bigl(\sigma(s)\bigr)\bigr]
         \\
    &= \sup_{s\in(0,e^{1-e})}g^{*}\bigl(\sigma(s)\bigr)\ell(s)^{-\alpha}\ell\ell(s)^{-\beta}
    = \sup_{s\in(0,1)}g^{*}(s)\ell\bigl(\sigma^{-1}(s)\bigr)^{-\alpha}\ell\ell\bigl(\sigma^{-1}(s)\bigr)^{-\beta}.
\end{align*}
Since
\begin{equation*}
	\sigma^{-1}(r) = \ell\ell^{-1}\bigl(\ell\ell(r)+1\bigr)
		\quad\text{for $r\in(0,1)$},
\end{equation*}
one has that
\begin{equation*}
	\ell\bigl(\sigma^{-1}(r)\bigr) = e\ell(r)
		\quad\text{and}\quad
	\ell\ell\bigl(\sigma^{-1}(r)\bigr) = 1 + \ell\ell(r)
		\quad\text{for $r\in(0,1)$}.
\end{equation*}
Consequently,
\begin{align*}
   \|g\|_{X_{\lgn}'(0,1)} \gtrsim \sup_{s\in(0,1)}g^{*}(s)\ell(s)^{-\alpha}\ell\ell(s)^{\beta}
    = \|g\|_{L^{\infty,\infty;-\alpha,\beta}(0,1)},
\end{align*}
up to a constant depending on $\alpha$ and $\beta$.
This yields the converse embedding to~\eqref{E:emb-11ab}.
Altogether, we obtain
\begin{equation*}
    X_{\lgn}\RG = L^{1,1;\alpha,\beta}\RG,
\end{equation*}
up to a constant depending on $\alpha$ and $\beta$.
Thanks to Theorem~\ref{T:OU-target}, this shows that the target space in the second embedding in~\eqref{E:optimal-embedding-lz} is optimal among all rearrangement-invariant target spaces.

Assume that $p\in(1,\infty)$, $q\in[1,\infty]$, $\alpha,\beta\in\R$, and $X(0,1)=L^{p,q;\alpha,\beta}(0,1)$.
Then~\eqref{E:target-condition} is satisfied and
\begin{align*}
	\|g\|_{X_{\lgn}'(0,1)}
		& = \|\sou g^{*}\|_{L^{p',q';-\alpha,-\beta}(0,1 )}
			\\
    & \approx
				\left\|
					\frac{s^{\frac{1}{p'}-\frac{1}{q'}-1}}{\ell(s)^{\alpha+1}\ell\ell(s)^{\beta}}
						\int_{0}^{s} g^{*}(r) \,\d r
				\right\|_{L^{q'}(0,1 )}
			+ \left\|
					\frac{s^{\frac{1}{p'}-\frac{1}{q'}}}{\ell(s)^{\alpha}\ell\ell(s)^{\beta}}
						\int_{s}^{1} \frac{g^{*}(r)}{r\ell(r)}\d r
				\right\|_{L^{q'}(0,1 )}
\end{align*}
for $g\in\Mpl(0,1)$, with equivalence constants depending on $p,q,\alpha,\beta$.
From classical weighted Hardy-type inequalities -- see \eg \citep[Theorems~1.3.2.2 and~1.3.2.3]{Maz:11} -- we deduce that
\begin{equation*}
	\|g\|_{X_{\lgn}'(0,1)}
		\lesssim
			\left\|
				s^{\frac{1}{p'}-\frac{1}{q'}} \ell(s)^{-\alpha-1}\ell\ell(s)^{-\beta} g^{*}(s)
			\right\|_{L^{q'}(0,1)}
	\quad\text{for $g\in\Mpl(0,1)$},
\end{equation*}
up to a constant depending on $p,q,\alpha,\beta$.
Since $g^{**}\ge g^{*}$, one also has that
\begin{equation*}
	\|g\|_{X_{\lgn}'(0,1)}
		\gtrsim
			\left\|
				s^{\frac{1}{p'}-\frac{1}{q'}} \ell(s)^{-\alpha-1}\ell\ell(s)^{-\beta} g^{*}(s)
			\right\|_{L^{q'}(0,1)}
	\quad\text{for $g\in\Mpl(0,1)$,}
\end{equation*}
up to a constant depending on $p,q,\alpha,\beta$.
Altogether,
\begin{equation*}
	\|g\|_{X_{\lgn}'(0,1)}
		\approx
			\left\|
				s^{\frac{1}{p'}-\frac{1}{q'}} \ell(s)^{-\alpha-1}\ell\ell(s)^{-\beta} g^{*}(s)
			\right\|_{L^{q'}(0,1)}
	\quad\text{for $g\in\Mpl(0,1)$,}
\end{equation*}
with equivalence constants depending on $p,q,\alpha,\beta$.
Thus, $X_{\lgn}\RG = L^{p,q;\alpha+1,\beta}\RG$, up to equivalent norms, and the optimality of this target in the third embedding in~\eqref{E:optimal-embedding-lz} follows by Theorem~\ref{T:OU-target}.

Let $\alpha<0$ and let $X(0,1)=L^{\infty,\infty;\alpha,\beta}(0,1)$.
Then condition~\eqref{E:target-condition} is fulfilled and
\begin{align*}
	\|g\|_{X_{\lgn}'(0,1)}
		& = \|\sou g^{*}\|_{L^{1,1;-\alpha,-\beta}(0,1)}
			= \|\ell(s)^{-\alpha}\ell\ell(s)^{-\beta}\sou g^{*}(s)\|_{L^{1}(0,1)}
        \\
    & \approx \int_{0}^{1} \frac{1}{s\ell(s)^{1+\alpha}\ell\ell(s)^{\beta}}
								\int_{0}^{s}g^{*}(r)\,\d r\,\d s
        + \int_{0}^{1} \frac{1}{\ell(s)^{\alpha}\ell\ell(s)^{\beta}}
						\int_{s}^{1} \frac{g^{*}(r)}{r\ell(r)}\,\d r\,\d s
        \\
    & = \int_{0}^{1} g^{*}(r)
					\int_{r}^{1} \frac{\d s}{s\ell(s)^{1+\alpha}\ell\ell(s)^{\beta}}\,\d r
        + \int_{0}^{1} \frac{g^{*}(r)}{r\ell(r)}
						\int_{0}^{r} \frac{\d s}{\ell(s)^{\alpha}\ell\ell(s)^{\beta}}\,\d r
        \\
    & \approx \int_{0}^{1} g^{*}(r)
								\int_{r}^{1} \frac{\d s}{s\ell(s)^{1+\alpha}\ell\ell(s)^{\beta}}\,\d r
    \approx \|g\|_{L^{1,1;-\alpha,\beta}(0,1)},
\end{align*}
with equivalence constants depending on $\alpha$ and $\beta$.
Hence $X_{\lgn}'\RG = L^{1,1;-\alpha,-\beta}\RG$, whence $X_{\lgn}\RG = L^{\infty,\infty;\alpha,\beta}\RG$.
By Theorem~\ref{T:OU-target}, the target space in the fourth embedding in~\eqref{E:optimal-embedding-lz} is optimal.

Let $\beta\le0$ and let $X(0,1)=L^{\infty,\infty;0,\beta}(0,1)$.
Then~\eqref{E:target-condition} is satisfied and
\begin{align*}
	\|g\|_{X_{\lgn}'(0,1)}
		& = \|\sou g^{*}\|_{L^{1,1;0,-\beta}(0,1)}
			= \|\ell\ell(s)^{-\beta}\sou g^{*}\|_{L^{1}(0,1)}
        \\
    & \approx \int_{0}^{1} \frac{1}{s\ell(s)\ell\ell(s)^{\beta}}
								\int_{0}^{s}g^{*}(r)\,\d r\,\d s
        + \int_{0}^{1} \frac{1}{\ell\ell(s)^{\beta}}
						\int_{s}^{1} \frac{g^{*}(r)}{r\ell(r)}\d r\,\d s
        \\
    & = \int_{0}^{1} g^{*}(r)
					\int_{r}^{1} \frac{\d s}{s\ell(s)\ell\ell(s)^{\beta}}\,\d r
        + \int_{0}^{1} \frac{g^{*}(r)}{r\ell(r)}
						\int_{0}^{r} \frac{\d s}{\ell\ell(s)^{\beta}}\,\d r
        \\
    & \approx \int_{0}^{1} g^{*}(r)
								\int_{r}^{1} \frac{\d s}{s\ell(s)\ell\ell(s)^{\beta}}\,\d r
			\approx \int_{0}^{1} g^{*}(r) \ell\ell(r)^{1-\beta}\,\d r
			= \|g\|_{L^{1,1;0,1-\beta}(0,1)}
\end{align*}
with equivalence constants depending on $\beta$.
Therefore, $X_{\lgn}'\RG = L^{1,1;0,1-\beta}\RG$, and hence $X_{\lgn}\RG = L^{\infty,\infty;0,\beta-1}\RG$.
This implies, via Theorem~\ref{T:OU-target}, the optimality of the target space in the fifth embedding in~\eqref{E:optimal-embedding-lz}.

We have shown that all the embeddings in \eqref{E:optimal-embedding-lz} hold, and that each target space is optimal (smallest possible) among all rearrangement-invariant spaces.
To finish the proof, we need only to verify that also the domain spaces are optimal.
Owing to the fact that the operator $\sou$ is self-adjoint, the optimality of a domain in an embedding is equivalent to that of the target in the embedding where the domain space and the target space are replaced by the associate of the target space and the associate of the domain space, respectively.
Hence, the optimality of the domain spaces follows from that of the target spaces via a well-known characterization of the respective associate spaces.
We omit the details, for brevity.
\end{proof}

\section{Ornstein--Uhlenbeck embeddings in Marcinkiewicz spaces} \label{endpoint}

We conclude our discussion by exhibiting optimal Ornstein--Uhlenbeck embeddings where either the domain, or the target is a Marcinkiewicz space.
This is the content of the following result.

\begin{theorem}[Optimal embeddings for Marcinkiewicz spaces]\label{T:optimal-embeddings-marcinkiewicz}
Let $\varphi$ and $\theta$ be quasiconcave functions on $(0,1)$.
\begin{enumerate}
\item\label{en:optimal-M-target}
Assume that
\begin{equation}\label{E:marcinkiewicz-condition-range}
    \int_{0}^{1}\ell\ell(s)\d\tvp(s)<\infty.
\end{equation}
Let $\psi\colon (0,1) \to [0, \infty)$ be the function given by
\begin{equation}\label{E:tilde-psi}
  \psi(s)=\int_{0}^{s}\frac{\d r}{\varphi(r)\ell(r)}+s\int_{s}^{1}\frac {\d\tvp(r)}{r\ell(r)}
	\quad\text{for $s\in(0,1)$}.
\end{equation}
Then
\begin{equation}\label{dec41}
	\Wlgn  M_\varphi \RG \to M_{\tps}\RG,
\end{equation}
and $M_{\tps}\RG$ is the optimal rearrangement-invariant target space in~\eqref{dec41}.
Here, $\tvp$ and $\tps$ denote the functions associated with $\varphi$ and $\psi$ as in~\eqref{dec40}.
\item\label{en:optimal-M-domain}
Assume that
\begin{equation}\label{E:marcinkiewicz-condition-domain}
	\sup_{s\in(0,1)}\ell\ell(s)\theta(s)<\infty.
\end{equation}
Then the functional given by
\begin{equation}\label{E:marcinkiewicz-optimal-domain}
	\|g\|_{Z(0,1)}
		= \sup_{s\in(0,1)} \theta(s)\left(\frac{1}{s}\int_{0}^{s}\frac{g^{**}(r)}{\ell(r)}\d r + \int_{s}^{1}\frac{g^*(r)}{r\ell(r)}\d r\right)
\end{equation}
for $g\in\Mpl(0,1)$ is a rearrangement-invariant function norm.
Moreover,
\begin{equation}\label{dec42}
	\Wlgn Z \RG \to M_{\theta}\RG,
\end{equation}
and $Z\RG$ is the optimal rearrangement-invariant domain space in~\eqref{dec42}.
\end{enumerate}
\end{theorem}

\begin{proof}
\ref{en:optimal-M-target}
Set $X(0,1)=M^\varphi(0,1)$.
We use the description of the optimal rearrangement-invariant target $\xs$ given in Theorem~\ref{T:OU-target}.
By the monotonicity of $\sou g^{*}$, the definition of $\sou$ and Fubini's theorem, we have that
\begin{align*}
    \|g\|_{\xs'(0,1)}
        & = \|\sou g^{*}\|_{\Lambda_{\tvp}(0,1)} = \int_{0}^{1}\sou g^{*}(s)\d\tvp(s)
            = \int_{0}^{1}\left(\frac{1}{s\ell(s)}\int_{0}^{s}g^{*}(r)\d r + \int_{s}^{1}\frac{g^{*}(r)}{r\ell(r)}\d r\right)\d\tvp(s)
                \\
        & = \int_{0}^{1}g^{*}(s)\left(\int_{s}^{1}\frac{\d\tvp(r)}{r\ell(r)} + \frac{\tvp(s)}{s\ell(s)}\right)\d s
          = \int_{0}^{1}g^{*}(s)\left(\int_{s}^{1}\frac{\d\tvp(r)}{r\ell(r)} + \frac{1}{\varphi(s)\ell(s)}\right)\d s
\end{align*}
for every $g\in\Mpl(0,1)$.
This amounts to saying that $\|\cdot\|_{\xs'(0,1)}=\|\cdot\|_{\Lambda_{\psi}(0,1)}$, where
\begin{align}\label{NOV1}
    \psi(s)
        & = \int_{0}^{s}\int_{r}^{1}\frac{\d\tvp(\varrho)}{\varrho\ell(\varrho)}\d r + \int_{0}^{s}\frac{\d r}{\varphi(r)\ell(r)}
            \\
        & = \int_{0}^{s}\int_{r}^{s}\frac{\d\tvp(\varrho)}{\varrho\ell(\varrho)}\d r + s\int_{s}^{1}\frac{\d\tvp(r)}{r\ell(r)} + \int_{0}^{s}\frac{\d r}{\varphi(r)\ell(r)}
            \\
        & = \int_{0}^{s}\frac{\d\tvp(r)}{\ell(r)} + s\int_{s}^{1}\frac{\d\tvp(r)}{r\ell(r)} + \int_{0}^{s}\frac{\d r}{\varphi(r)\ell(r)}
			\quad\text{for $s\in (0,1)$.}
\end{align}
 Let $0<s_1 < s_2 <1$. Then
\begin{align}\label{july10}
\int_{s_1}^{s_2} \d \tvp(r) = \tvp(s_2) - \tvp(s_1) = \frac{s_2}{\varphi (s_2)} - \frac{s_1}{\varphi (s_1)}\le  \frac{s_2}{\varphi (s_2)}  -  \frac{s_1}{\varphi (s_2)} \le \int_{s_1}^{s_2}  \frac{\d r}{\varphi(r)}.
\end{align}
Hence, the first integral on the rightmost side of equation~\eqref{NOV1} is bounded by the third one.
The conclusion hence follows via property~\eqref{dic30}.

\ref{en:optimal-M-domain}
By Theorem~\ref{C:OU-domain}, the optimal rearrangement domain space $Z(0,1)=\ys(0,1)$ associated with the  target space $Y(0,1)=M_\theta(0,1)$
obeys
\begin{equation*}
    \|g\|_{Z(0,1)} = \sup_{s\in(0,1)}\theta(s)\left(\sou g^{*}\right)^{**}(s)
\end{equation*}
for $g\in\Mpl(0,1)$.
Moreover, for every $g\in\Mpl(0,1)$ and $s\in(0,1)$,
\begin{align*}
	\left(\sou g^{*}\right)^{**}(s)
		&  = \frac{1}{s}\int_{0}^{s}\left(\frac{1}{r\ell(r)}\int_{0}^{r}g^*(\varrho)\,\d \varrho
				+ \int_{r}^{1}\frac{g^{*}(\varrho)}{\varrho\ell(\varrho)}\d \varrho\right)\d r
    	\\
		& = \frac{1}{s}\int_{0}^{s}\frac{g^{**}(r)}{\ell(r)}\d r
				+ \frac{1}{s}\int_{0}^{s}\int_{r}^{s} \frac{g^{*}(\varrho)}{\varrho\ell(\varrho)}\d\varrho\,\d r
				+ \int_{s}^{1}\frac{g^{*}(r)}{r\ell(r)}\d r
    	\\
		& = \frac{1}{s}\int_{0}^{s}\frac{g^{**}(r)}{\ell(r)}\d r
				+ \frac{1}{s}\int_{0}^{s}\frac{g^{*}(r)}{\ell(r)}\d r
				+ \int_{s}^{1}\frac{g^{*}(r)}{r\ell(r)}\d r
   		 \\
		& \approx \frac{1}{s}\int_{0}^{s}\frac{g^{**}(r)}{\ell(r)}\d r
				+ \int_{s}^{1}\frac{g^{*}(r)}{r\ell(r)}\d r,
\end{align*}
with absolute equivalence constants, inasmuch as, in the last but one line, the second integral is bounded by the first one.
Hence, the conclusion follows.
\end{proof}

\begin{theorem}[Optimal embeddings for Marcinkiewicz spaces -- examples]\label{T:optimal-embeddings-marcinkiewicz-examples}
One has
\begin{equation}\label{E:optimal-embedding-endpoints}
	\left\{
	\renewcommand{\arraystretch}{1.2}
	\begin{array}{@{}l@{{}\to{}}l@{\quad}l}
		\Wlgn L^{(1,\infty;0,\beta)}
			& L^{(1,\infty;0,\beta-1)}
				& \text{if $\beta>1$}
					\\
		\Wlgn L^{(1,\infty;\alpha,\beta)}
			& L^{(1,\infty;\alpha,\beta)}
				& \text{if $\alpha>0$ and $\beta\in\R$}
					\\
		\Wlgn L^{(p,\infty;\alpha,\beta)}
			& L^{(p,\infty;\alpha+1,\beta)}
				& \text{if $p\in(1,\infty)$ and $\alpha,\beta\in\R$}
					 \\
		\Wlgn\exp L^{\beta}
		 	& \exp L^{\beta}
				& \text{if $\beta>0$}
					 \\
		\Wlgn\exp\exp L^{\beta}
			& \exp\exp L^{\frac{\beta}{\beta+1}}
				& \text{if $\beta>0$}
				 \\
		\Wlgn L^{\infty}
			& \exp\exp L
				& \text{if $\beta>0$,}
	\end{array}
	\right.
\end{equation}
where all the spaces are over $\RG$.
Moreover, all target spaces and all domain spaces are optimal in \eqref{E:optimal-embedding-endpoints} among rearrangement-invariant spaces.
\end{theorem}

\begin{proof}
We begin by showing the optimality of target spaces in \eqref{E:optimal-embedding-endpoints} via formula~\eqref{E:tilde-psi} from Theorem~\ref{T:optimal-embeddings-marcinkiewicz}.

First, when $\beta>1$, we have that $L^{(1,\infty;0,\beta)}=M_{\varphi}$ with $\varphi(s)=s\ell\ell(s)^\beta$ for $s\in(0,1)$.
Since $\tvp(s)=\ell\ell(s)^{-\beta}$ for $s\in(0,1)$,
\begin{equation*}
	\d\tvp(s) \approx s^{-1}\ell(s)^{-1}\ell\ell(s)^{-\beta-1}\d s
		\quad\text{for $s\in(0,1)$.}
\end{equation*}
Hence,
\begin{equation*}
	\int_{0}^{1}\ell\ell(s)\d\tvp(s)
		\approx \int_0^1 \frac{\d s}{s\ell(s)\ell\ell(s)^{\beta}}
		< \infty.
\end{equation*}
Thus, condition~\eqref{E:marcinkiewicz-condition-range} is satisfied.
The function $\psi$ given by~\eqref{E:tilde-psi} obeys
\begin{align*}
	\psi(s)
		& \approx \int_0^s \frac{\d r}{r\ell(r)\ell\ell(r)^{\beta}}
			+ s \int_{s}^{1} \frac{\d r}{r^{2}\ell(r)^{2}\ell\ell(r)^{\beta+1}}
    \approx \ell\ell(s)^{1-\beta}+\ell(s)^{-2}\ell\ell(s)^{-\beta-1}
		\approx \ell\ell(s)^{1-\beta}
\end{align*}
for $s\in(0,1)$.
Thus, $\tps(s)\approx s\ell\ell(s)^{\beta-1}$ for $s\in(0,1)$, whence $M_{\tps}=L^{(1,\infty;0,\beta-1)}$, and, by Theorem~\ref{T:optimal-embeddings-marcinkiewicz}~\ref{en:optimal-M-target}, the first embedding in~\eqref{E:optimal-embedding-endpoints} holds and its target space is optimal among all rearrangement-invariant spaces.

If $\alpha>0$ and $\beta\in\R$, then $L^{(1,\infty;\alpha,\beta)}=M_{\varphi}$,
where $\varphi(s)=s\ell(s)^{\alpha}\ell\ell(s)^\beta$ for $s\in(0,1)$.
One has that $\tvp(s)=\ell(s)^{-\alpha}\ell\ell(s)^{-\beta}$ and $\d\tvp(s)\approx s^{-1}\ell(s)^{-\alpha-1}\ell\ell(s)^{-\beta}\d s$ for $s\in(0,1)$.
Hence
\begin{equation*}
	\int_{0}^{1}\ell\ell(s)\d\tvp(s)
		\approx \int_0^1 \frac{\ell\ell(s)^{1-\beta}}{s\ell(s)^{\alpha+1}}\d s
		< \infty.
\end{equation*}
Thus, condition~\eqref{E:marcinkiewicz-condition-range} is satisfied.
The function $\psi$ from~\eqref{E:tilde-psi} satisfies
\begin{align*}
	\psi(s)
		& \approx \int_0^s r^{-1}\ell(r)^{-\alpha-1}\ell\ell(r)^{-\beta}\d r
			+ s\int_{s}^{1}r^{-2}\ell(r)^{-\alpha-2}\ell\ell(r)^{-\beta}\d r
        \\
    & \approx \ell(s)^{-\alpha}\ell\ell(s)^{-\beta}
			+ \ell(s)^{-\alpha-2}\ell\ell(s)^{-\beta}
			\approx \ell(s)^{-\alpha}\ell\ell(s)^{-\beta}
\end{align*}
for $s\in(0,1)$.
Thereby, $\tps(s)\approx \varphi(s)$ for $s\in(0,1)$, whence $M_{\tps}=L^{(1,\infty;\alpha,\beta)}$.

Next, assume that $p\in(1,\infty)$ and $\alpha,\beta\in\R$.
Then $L^{(p,\infty;\alpha,\beta)}=M_{\varphi}$, where $\varphi(s)=s^{\ip}\ell(s)^{\alpha}\ell\ell(s)^\beta$ for $s\in(0,1)$.
We have that $\tvp(s)=s^{1-\ip}\ell(s)^{-\alpha}\ell\ell(s)^{-\beta}$ and $\d\tvp(s)\approx s^{-\ip}\ell(s)^{-\alpha}\ell\ell(s)^{-\beta}\d s$ for $s\in(0,1)$.
Consequently,
\begin{equation*}
    \int_{0}^{1}\ell\ell(s)\d\tvp(s) \approx \int_0^1\ell\ell(s)^{1-\beta}\ell(s)^{-\alpha}s^{-\ip}\d s <\infty.
\end{equation*}
Therefore, condition~\eqref{E:marcinkiewicz-condition-range} is satisfied.
The function $\psi$ given by~\eqref{E:tilde-psi} fulfills
\begin{align*}
	\psi(s)
		& \approx \int_0^s \frac{\d r}{r^{\ip}\ell(r)^{\alpha+1}\ell\ell(r)^{\beta}}
			+ s\int_{s}^{1} \frac{\d r}{r^{\ip+1}\ell(r)^{\alpha+1}\ell\ell(r)^{\beta}}
      \approx s^{1-\ip}\ell(s)^{-\alpha-1}\ell\ell(s)^{-\beta}
\end{align*}
for $s\in(0,1)$.
Thus, $\tps(s)\approx s^{\ip}\ell(s)^{\alpha+1}\ell\ell(s)^{\beta}$ for $s\in(0,1)$, whence $M_{\tps}=L^{(p,\infty;\alpha+1,\beta)}$.

If  $\beta>0$, then $\exp L^{\beta}=M_\varphi$, where
$\varphi(s)=\ell(s)^{-\frac{1}{\beta}}$ for $s\in(0,1)$.
Therefore, $\tvp(s)=s\ell(s)^{\frac{1}{\beta}}$  and $\d\tvp(s) \approx \ell(s)^{\frac{1}{\beta}}\d s$ for $s\in(0,1)$.  Hence,
\begin{equation*}
    \int_{0}^{1}\ell\ell(s)\d\tvp(s) \approx \int_0^1\ell\ell(s)\ell(s)^{\frac{1}{\beta}}\d s <\infty.
\end{equation*}
Condition~\eqref{E:marcinkiewicz-condition-range} is satisfied, and the function $\psi$ from~\eqref{E:tilde-psi} obeys
\begin{align*}
	\psi(s)
		& \approx \int_0^s \ell(r)^{\frac{1}{\beta}-1}\d r
				+ s\int_{s}^{1}r^{-1}\ell(r)^{\frac{1}{\beta}-1}\d r
      \approx s\ell(s)^{\frac{1}{\beta}-1} + s\ell(s)^{\frac{1}{\beta}}
			\approx  s\ell(s)^{\frac{1}{\beta}}
\end{align*}
for $s\in(0,1)$.
Therefore, $\tps(s)\approx \varphi(s)$ for $s\in(0,1)$, whence $M_{\tps}=\exp L^{\beta}$.

If $\beta>0$, then $\exp\exp L^{\beta}=M_\varphi$, where $\varphi(s)=\ell\ell(s)^{-\frac{1}{\beta}}$ for $s\in(0,1)$.
We have that $\tvp(s)=s\ell\ell(s)^{\frac{1}{\beta}}$ and $\d\tvp(s) \approx \ell\ell(s)^{\frac{1}{\beta}}\d s$ for $s\in(0,1)$.
Condition~\eqref{E:marcinkiewicz-condition-range} is fulfilled, since
\begin{equation*}
	\int_{0}^{1}\ell\ell(s)\,\d\tvp(s)
		\approx \int_0^1\ell\ell(s)^{1+\frac{1}{\beta}}\,\d s
		< \infty.
\end{equation*}
The function $\psi$ given by~\eqref{E:tilde-psi} satisfies
\begin{align*}
	\psi(s)
		& \approx \int_0^s \frac{\ell\ell(r)^{\frac{1}{\beta}}}{\ell(r)}\,\d r
				+ s\int_{s}^{1} \frac{\ell\ell(r)^{\frac{1}{\beta}}}{r\ell(r)}\,\d r
      \approx s\ell(s)^{-1}\ell\ell(s)^{\frac{1}{\beta}}
				+ s\ell\ell(s)^{\frac{1}{\beta}+1}
			\approx  s\ell\ell(s)^{\frac{1}{\beta}+1}
\end{align*}
 for $s\in(0,1)$.
Thus, $\tps(s)\approx \ell\ell(s)^{-\frac{\beta+1}{\beta}}$ for $s\in(0,1)$, whence $M_{\tps}=\exp\exp L^{\frac{\beta}{\beta+1}}$.

As for the last embedding, we have that $L^\infty=M_\varphi$, where $\varphi(s)=1$ for $s\in(0,1)$.
Then $\tvp(s)=s$ and $\d\tvp(s) = \d s$ for $s\in(0,1)$, and
\begin{equation*}
	\int_{0}^{1}\ell\ell(s)\d\tvp(s)
		= \int_0^1\ell\ell(s)\d s
		< \infty.
\end{equation*}
Thus, condition~\eqref{E:marcinkiewicz-condition-range} is satisfied, and
\begin{align*}
	\psi(s)
		& \approx \int_0^s \frac{\d r}{\ell(r)}
				+ s\int_{s}^{1} \frac{\d r}{r\ell(r)}
      \approx s\ell(s)^{-1}+s\ell\ell(s)
			\approx  s\ell\ell(s)
\end{align*}
 for $s\in(0,1)$.
Consequently, $\tps(s)\approx \ell\ell(s)$ for $s\in(0,1)$, whence $M_{\tps}=\exp\exp L$.

We have thus shown that the embeddings in \eqref{E:optimal-embedding-endpoints} hold and that the target spaces are optimal.
It remains to prove the optimality of the domain spaces By Theorem~\ref{T:optimal-embeddings-marcinkiewicz}~\ref{en:optimal-M-target}, given a domain space of Marcinkiewicz type $M_\varphi$, its optimal rearrangement-invariant target space is also a Marcinkiewicz space $M_\theta$.
Hence,
\begin{equation} \label{prd}
    \Wlgn M_\varphi \to M_\theta.
\end{equation}
Now, thanks to Theorem~\ref{T:optimal-embeddings-marcinkiewicz}~\ref{en:optimal-M-domain}, there exists an optimal rearrangement-invariant domain space $X$ for the target $M_\theta$.
By the optimality of $X$, we have that $M_\varphi\to X$.
Our goal will be to show that the converse embedding $X\to M_\varphi$ holds as well.
To this end, it suffices to prove the inequality between their fundamental functions $\varphi_X\gtrsim\varphi$.
Indeed, then one has
\begin{equation*}
	X\to M_{\varphi_X}\to M_\varphi,
\end{equation*}
where the first embedding holds owing to~\eqref{E:fundamental-endpoint-embeddings}, whereas the second follows immediately from the definition of the Marcinkiewicz functional.
With formula~\eqref{E:marcinkiewicz-optimal-domain} at hand, we have that
\begin{align*}
	\varphi_X(a)
		& \gtrsim \sup_{s\in(0,1)} \frac{\theta(s)}{s} \int_{0}^s \frac{\chi_{(0,a)}^{**}(r)}{\ell(r)}\d r
			= \sup_{s\in(0,1)} \frac{\theta(s)}{s}
				\biggl(
					\int_0^s \frac{\chi_{(0,a)}(r)}{\ell(r)}\d r
				+	a\int_{a}^s \frac{\chi_{(a,1)}(r)}{r\ell(r)}\d r
				\biggr)
			\\
		& \approx \max
			\biggl\{
			\sup_{s\in(0,a)} \frac{\theta(s)}{\ell(s)},
		  a \sup_{s\in(a,1)} \frac{\theta(s)}{s} \bigl[\ell\ell(a) - \ell\ell(s)\bigr]
			\biggr\}
\end{align*}
for $a\in(0,1)$.
Consequently, the domain space in \eqref{prd} is optimal if
\begin{equation} \label{E:cond-opt-ms}
	\sup_{s\in(0,a)} \frac{\theta(s)}{\ell(s)}
		\gtrsim \varphi(a)
	\quad\text{or}\quad
	a\sup_{s\in(a,1)} \frac{\theta(s)}{s} \bigl[\ell\ell(a) - \ell\ell(s)\bigr]
		\gtrsim \varphi(a)
\end{equation}
for every $a\in(0,1)$.
The remaining part of this proof is devoted to showing that equation~\eqref{E:cond-opt-ms} is fulfilled for each embedding in \eqref{E:optimal-embedding-endpoints}.
Note also that condition~\eqref{E:marcinkiewicz-condition-domain} is satisfied for each of them.

Let $\beta>1$ and $\theta(s)=s\ell\ell(s)^{\beta-1}$ for $s\in(0,1)$.
Fix $a\in(0,e^{1-e})$ and let
 $\eta$ be the function defined by~\eqref{E:eta}. Then,
\begin{align*}
	a \sup_{s\in(a,1)} \frac{\theta(s)}{s} \bigl[\ell\ell(a) - \ell\ell(s)\bigr]
		\ge a\ell\ell\bigl(\eta(a)\bigr)^{\beta-1}\bigl[\ell\ell(a)-\ell\ell(\eta(a))\bigr]
		= 2^{-\beta}a\ell\ell(a)^{\beta}.
\end{align*}
Since $\varphi_{L^{(1,\infty;0,\beta)}}(a)\approx a\ell\ell(a)^{\beta}$ for $a\in(0,e^{1-e})$, the second inequality in~\eqref{E:cond-opt-ms} follows.
Altogether, this proves that the domain space in the first embedding in~\eqref{E:optimal-embedding-endpoints} is optimal.

As for the second embedding, let $\alpha>0$, $\beta\in\R$, and $\theta(s)=s\ell(s)^{\alpha}\ell\ell(s)^{\beta}$ for $s\in(0,1)$.
Let $\sigma$ be the function defined by~\eqref{E:sigma}. Then,  for every $a\in(0,e^{1-e})$,
\begin{align*}
	a \sup_{s\in(a,1)} \frac{\theta(s)}{s} \bigl[\ell\ell(a) - \ell\ell(s)\bigr]
		\ge a\ell\bigl(\sigma(a)\bigr)^\alpha\ell\ell\bigl(\sigma(a)\bigr)^{\beta}\bigl[\ell\ell(a) - \ell\ell(\sigma(a))\bigr]
		\approx a\ell(a)^{\alpha}\ell\ell(a)^{\beta}.
\end{align*}
This shows that the second embedding in~\eqref{E:optimal-embedding-endpoints} has an optimal domain.

Concerning the third embedding, let $p\in(1,\infty)$, $\alpha,\beta\in\R$, and
$\theta(s)=s^{\ip}\ell(s)^{\alpha+1}\ell\ell(s)^{\beta}$ for $s\in(0,1)$.
Hence, for every $a\in(0,1)$,
\begin{align*}
	\sup_{s\in(0,a)} \frac{\theta(s)}{\ell(s)}
		\ge a^\ip\ell(a)^\alpha\ell\ell(a)^\beta
		= \varphi_{L^{(p,\infty;\alpha,\beta)}}(a).
\end{align*}
Hence, the optimality of the domain in the third embedding in~\eqref{E:optimal-embedding-endpoints} follows.

Let us focus on the fourth embedding.
Assume that $\beta>0$, and let $\theta(s)=\ell(s)^{-\beta}$ for $s\in(0,1)$.
If $\sigma$ is again the function given by equation \eqref{E:sigma}, then we have that
\begin{align*}
	a \sup_{s\in(a,1)} \frac{\theta(s)}{s} \bigl[\ell\ell(a) - \ell\ell(s)\bigr]
		\gtrsim \frac{a}{\sigma(a)} \ell\bigl(\sigma(a)\bigr)^{-\beta} \bigl[\ell\ell(a) - \ell\ell(\sigma(a))\bigr]
		\approx \ell(a)^{-\beta}
		= \varphi_{\exp L^{\beta}}(a)
\end{align*}
for $a\in(0,e^{1-e})$, whence we deduce that the fourth embedding in~\eqref{E:optimal-embedding-endpoints} has an optimal domain.

If $\beta>0$ and $\theta(s)=\ell\ell(s)^{-1-\frac{1}{\beta}}$ for $s\in(0,1)$, then the optimality of the fifth embedding follows from the fact that
\begin{align*}
	a \sup_{s\in(a,1)} \frac{\theta(s)}{s} \bigl[\ell\ell(a) - \ell\ell(s)\bigr]
		\gtrsim \frac{a}{\eta(a)} \ell\ell\bigl(\eta(a)\bigr)^{-1-\beta} \bigl[\ell\ell(a) - \ell\ell(\eta(a))\bigr]
		\approx \ell\ell(a)^{-\ib}
		= \varphi_{\exp\exp L^{\beta}}(a)
\end{align*}
for $a\in(0,e^{1-e})$, where $\eta$ is the function given by \eqref{E:eta}.

Finally, in the last embedding we have  that $\theta(s)=\ell\ell(s)^{-1}$ for $s\in(0,1)$, whence
\begin{align*}
	a \sup_{s\in(a,1)} \frac{\theta(s)}{s} \bigl[\ell\ell(a) - \ell\ell(s)\bigr]
		\gtrsim \frac{a}{\eta(a)} \ell\ell\bigl(\eta(a)\bigr)^{-1} \bigl[\ell\ell(a) - \ell\ell(\eta(a))\bigr]
		\approx 1
		= \varphi_{L^\infty}(a)
\end{align*}
for $a\in(0,e^{1-e})$. This implies that the domain in the last embedding in~\eqref{E:optimal-embedding-endpoints} is optimal.
\end{proof}

\section*{Acknowledgement}

This research was partly funded by:

\begin{enumerate}
\item Research Project 2201758MTR2  of the Italian Ministry of University and
Research (MIUR) Prin 2017 ``Direct and inverse problems for partial differential equations: theoretical aspects and applications'';
\item GNAMPA of the Italian INdAM -- National Institute of High Mathematics
(grant number not available);
\item Operational Programme Research, Development and Education, Project Postdoc2MUNI no.\ CZ.02.2.69/0.0/0.0/18\_053/0016952;
\item Grants P201-18-00580S and P201-21-01976S of the Czech Science Foundation.
\end{enumerate}

\bibliographystyle{abbrvnat}

\end{document}